\documentclass[11pt]{article}
\usepackage[margin=1.2in]{geometry}
\usepackage{amsfonts,amsmath,amssymb,amsthm}
\usepackage{pifont}
\usepackage{subcaption}
\usepackage{graphicx}
\usepackage{authblk}
\usepackage[
colorlinks,
citecolor=blue,
urlcolor=black,
linkcolor=blue, 
backref=page,
linktocpage=true]{hyperref}
\usepackage[numbers]{natbib}
\allowdisplaybreaks 
\numberwithin{equation}{section}

\newtheorem{theorem}{Theorem}[section]
\newtheorem{lemma}[theorem]{Lemma}
\newtheorem{proposition}[theorem]{Proposition}
\newtheorem{remark}[theorem]{Remark}

\newcommand{\rob}{\textnormal{\tiny Rob}}
\newcommand{\Oas}{O_{\textnormal{a.s.}}}
\newcommand{\oas}{o_{\textnormal{a.s.}}}

\newcommand{\eps}{\epsilon}
\newcommand{\muhat}{\widehat{\mu}}
\newcommand{\sigmahat}{\widehat{\sigma}}
\newcommand{\Xbar}{\overline{X}}
\newcommand{\E}{\mathbb{E}}
\newcommand{\calF}{\mathcal{F}}
\newcommand{\hatX}{\widehat{X}}
\newcommand{\Xhat}{\widehat{X}}
\newcommand{\hatV}{\widehat{V}}
\renewcommand{\d}{\textnormal{d}}
\renewcommand{\Re}{\mathbb{R}}
\newcommand{\calP}{\mathcal{P}}
\newcommand{\calS}{\mathcal{S}}
\newcommand{\ber}{\textnormal{Ber}}
\newcommand{\xmark}{\ding{55}}%

\newcommand{\Tr}{\textnormal{tr}}
\renewcommand{\b}[1]{\mathbf{#1}}
\newcommand{\bX}{\b{X}}
\newcommand{\bS}{\b{S}}
\newcommand{\bM}{\b{M}}
\newcommand{\bV}{\b{V}}
\newcommand{\bhatX}{\b{\Xhat}}
\newcommand{\bXbar}{\b{\Xbar}}

\newcommand{\mix}{{\textnormal{mix}}}
\DeclareMathOperator{\erf}{\textnormal{erf}}
\newcommand{\apx}{{\textnormal{apx}}}
\newcommand{\wsr}{{\textnormal{\tiny WSR}}}

\newcommand{\prpl}{\textnormal{\tiny WSR}}
\newcommand{\unif}{\textnormal{unif}}
\newcommand{\stitch}{\textnormal{stch}}
\newcommand{\hrms}{\textnormal{\tiny HRMS}}
\newcommand{\mat}{\textnormal{\tiny mat}}
\newcommand{\hoeff}{\textnormal{\tiny Hoeff}}
\newcommand{\obmix}{\textnormal{\tiny Bern}}
\newcommand{\bern}{\textnormal{\tiny Bern}}
\newcommand{\obstitch}{\textnormal{\tiny Bern-Stch}}

\title{Closed-form empirical Bernstein confidence \\ sequences for scalars and matrices}
\author[1]{Ben Chugg}
\author[1]{Aaditya Ramdas}
\affil[1]{Departments of Machine Learning and Statistics, CMU}
\affil[ ]{\texttt{\{benchugg, aramdas\}@cmu.edu}}
\date{December 2025}

\begin{document}

\maketitle

\begin{abstract}
    We derive a new closed-form variance-adaptive confidence sequence (CS) for estimating the average conditional mean of a sequence of bounded random variables. 
    Empirically, it yields the tightest closed-form CS we have found for tracking time-varying means, across sample sizes up to $\approx 10^6$.
    When the observations happen to have the same conditional mean, our CS is asymptotically tighter than the recent closed-form CS of \citet{waudby2024estimating}. It also has other desirable properties: it is centered at the unweighted sample mean and has limiting width (multiplied by $\sqrt{t/\log t}$) independent of the significance level.  We extend our results to provide a CS with the same properties for random matrices with bounded eigenvalues. 
\end{abstract}

\section{Introduction}

Consider a stream of random variables $(X_t)_{t\geq 1}$ lying in $[0,1]$ with average conditional means $\mu_t = t^{-1}\sum_{j\leq t} \E[X_j |X_1,\dots,X_{j-1}]$. We are interested in obtaining nonasymptotic, time-uniform confidence sequences (CSs) for $(\mu_t)_{t \geq 1}$. Given a parameter $\alpha\in(0,1)$ and an initial peeking time $t_0$, a $(1-\alpha)$-CS for $(\mu_t)_{t \geq t_0}$ is a sequence of sets $(C_t)_{t\geq t_0}$ such that 
\begin{equation}
    P( \mu_t\in C_t\text{ simultaneously for all } t\geq t_0)\geq 1-\alpha,
\end{equation}
or, equivalently, $P(\mu_\tau\in C_\tau)\geq 1-\alpha$ for all stopping times $\tau\geq t_0$. The study of confidence sequences goes back at least to \citet{paulson1964sequential} and \citet{darling1967confidence}, who both gave results for  (sub-)Gaussian distributions. CSs are now applied in a wide range of sequential and online settings, and their study has blossomed in recent years.

Confidence sequences for bounded random variables is an important special case. First, many applications naturally yield bounded data. Second, boundedness opens the door to variance-adaptive bounds, which require no distributional knowledge beyond the range itself. This is in contrast to unbounded data, where nonasymptotic CSs require a priori knowledge of an upper bound on some moment- or tail-type quantity of the distribution~\citep{bahadur1956nonexistence}. For sub-Gaussian data one needs the sub-Gaussian variance proxy; for data with finite $p$-th moment one needs a bound on this moment, and so on. In many settings such information is unavailable or unrealistic.

For bounded data meanwhile, one can provide CSs which adapt to the unknown variance over time. Such bounds aim to approach the width of the oracle Bernstein bound (or rather its time-uniform version in our case; see \citep{howard2020time,howard2021time}) which assumes knowledge of the true variance. Such variance-adaptive CSs have come to be known as \emph{empirical} Bernstein CSs.

Empirical Bernstein bounds have enjoyed a tremendous amount of recent attention (e.g., \citep{howard2021time,waudby2024estimating, martinez2024empirical,wang2024sharp,voravcek2025star,whitehouse2025time,chugg2025variational,orabona2023tight}), owing both to their wide applicability and exciting new techniques for constructing them. A longer discussion on related work is held in Appendix~\ref{sec:related-work}, beginning with the original (fixed-time) bounds of \citet{audibert2009exploration} and \citet{maurer2009empirical} which kick-started this research program. 

This article is focused on developing \underline{closed-form} CSs \emph{without assuming a constant conditional mean}. We focus on closed-form bounds because they are straightforward to implement and allow for an asymptotic analysis of their widths. We allow for a non-constant conditional mean because most past work has required a shared mean, yet in many applications the mean drifts over time.

Our main points of comparison will be the empirical Bernstein CS of \citet{waudby2024estimating}, $C_t^\prpl$, which is empirically the tightest closed-form bound in the literature, and 
that of \citet{howard2021time}, $C_t^\hrms$, which has an optimal $\sqrt{\log \log t/t}$ rate of decay.  
We give more detail on these two bounds in Section~\ref{sec:sota}.

\paragraph{Main results.} 
Theorem~\ref{thm:closed-form} presents a new closed-form empirical Bernstein CS, $C_t^\apx$ which has several appealing properties. 
First, unlike $C_t^\wsr$, it tracks a time-varying conditional mean.  That is, we do not require that $\mu_t =\mu_{t-1}$, thus handling observations that are neither identically distributed nor independent.  Second, also unlike the bound of \citet{waudby2024estimating}, it is  centered at the unweighted sample mean $\Xbar_t = t^{-1}\sum_{i\leq t}X_i$.

One motivation for our work is the following observation: When the data are iid with mean $\mu$ and variance $\sigma^2$, the \emph{oracle} Bernstein CS\footnote{The word oracle refers to the fact that one must know $\sigma$ to implement this CS. The  oracle Bernstein CS of \citet{howard2021time} is one based on conjugate mixtures has width $O(\sqrt{\log t/t})$. We discuss iterated logarithm bounds later on.} of \citet{howard2021time} has asymptotic halfwidth $W_t^\bern$ that satisfies
\begin{equation}
W_t^\bern \sim  \sigma \sqrt{\frac{\log(t)}{t}} \quad \text{~ or equivalently, ~}\quad 
\sqrt\frac{t}{\log t}W_t^\bern \xrightarrow{a.s.} \sigma \text{ as } t \to \infty.    
\end{equation}
See Appendix~\ref{sec:oracle-bernstein} for details. 
%We use a $\sim$ notation to succinctly write the above as
Robbins' famous mixture CS for $\sigma$-sub-Gaussian observations also has the same asymptotic halfwidth (Appendix~\ref{sec:robbins-mixture}). Both of these closed-form CSs require knowledge of  $\sigma$. A natural question is thus whether we can design a closed-form fully empirical Bernstein CS which also has limiting width equal to $\sigma$, a property we call \emph{asymptotic sharpness}, following the use of this terminology for empirical Bernstein confidence \emph{intervals}~\citep{martinez2024empirical,wang2024sharp}. 

Our bound resolves this question up to constants. To elaborate, 
Lemma~\ref{lem:precise-width} shows that, when the data are iid as above, then $W_t^\apx$, the halfwidth of $C_t^\apx$, satisfies 
\begin{equation}
\label{eq:apx-width}
W_t^\apx \sim \sqrt{\frac{2\E[\psi_E(|X-\mu|)]\log(t)}{t}},
%\text{~ and ~}
\end{equation}
where $\psi_E(x) = -\log(1-x) -x$. 
% That is, $\lim_{t\to\infty} \sqrt{t/\log(t)} W_t^\apx = \sqrt{2\E[\psi_E(|X- \mu|)}$. 
Since $\psi_E(x) \geq x^2/2$, with $\psi_E(x) / (x^2/2) \to 1$ as $x \downarrow 0$, we see that
$\E[\psi_E(|X- \mu|)] \geq \sigma^2/2,$
but the two sides are almost identical for small $\sigma$ (see Table~\ref{tab:EpsiE}). More precisely, their difference is  $O(|X-\mu|^3)$, as seen by the Taylor expansion of $-\log(1-x)$ around $x=0$. (This is reminiscent of the leading term of the Berry-Esseen bound, and if the distribution is not too skewed, then it is $O(\sigma^3)$, but if the distribution is very skewed, it is $O(\sigma^2)$.) Further, one can show that $\E[\psi_E(|X-\mu|]\leq C_\mu \sigma^2$ for a constant $C_\mu$ depending on $\mu$. See Appendix~\ref{app:psiE-vs-sigma} for more detail. 
% Thus, we write
% \[
% \E[\psi_E(|X- \mu|) = \frac{\sigma^2}{2}(1 + o_\sigma(1)),
% \]
% where the $o_\sigma(1)$ term vanishes as $\sigma \to 0$.
% \begin{equation}
% \label{eq:apx-width-2}
% W_t^\apx \sim (1+o_\sigma(1)) \cdot \sigma \sqrt{\frac{\log(t)}{t}},
% %\text{~ and ~}
% \end{equation}

% Despite having a slightly larger leading constant than $\sigma$, as a function of $t$ the width of $C_t^\apx$ decays at the same $O(\sqrt{\log(t)/t})$ rate as the oracle Bernstein bound and Robbins' mixture. 
It is worth remarking that the above scaling is also independent of $\alpha$. Considering the dependence on both $t$ and $\alpha$, $C_t^\apx$ differs from the empirical Bernstein CS of \citet{waudby2024estimating}, whose halfwidth $W_t^\prpl$ satisfies
\begin{equation}
        W_t^\prpl \sim \frac{\sigma}{2} \sqrt{\frac{\log(2/\alpha) \log t}{2t}} \log\log t, 
\end{equation}
thus implying that $C_t^\apx$ is asymptotically tighter as $t \to \infty$, and their gap increases for smaller $\alpha$. 
% Intriguingly, despite being asymptotically sharp as a fixed-time confidence interval, $C_t^\prpl$ is not asymptotically sharp as a CS. 

%Third, when the data have a constant conditional mean, $C_t^\apx$ is provably asymptotically tighter than $C_t^\prpl$. In particular, Thus, $C_t^\apx$ shrinks at rate $\sqrt{\log(t)/t}$ and $C_t^\prpl$ at rate $\sqrt{\log(t)/t}\cdot \log\log(t)$. Moreover, the width of $C_t^\apx$ scales independently of the significance level $\alpha$ in the limit. 

Empirically, $C_t^\apx$ is tighter than $C_t^\prpl$ at large sample sizes (starting at roughly 500-1000 observations, depending on the distribution), but looser at small sample sizes. It is tighter than $C_t^\hrms$ except at extremely large sample sizes. This is to be expected: As we'll discuss below, $C_t^\hrms$ was engineered to have optimal rates with respect to $t$, thereby sacrificing some of its tightness at moderate sample sizes. 
%They perform remarkably similarly to $C_t^\hrms$ but are much faster to compute (being closed-form). Indeed, computing the CS from time 1 to 10,000 for 10,000 iid Bernoulli observations on a Macbook Pro 15,  $C_t^\apx$ takes $\approx 0.13$ seconds, $C_t^\mix$ takes $\approx 3.08$ seconds, and $C_t^\hrms$ takes $\approx 94.15$ seconds.\footnote{To compute the CSs of \citet{waudby2024estimating} and \citet{howard2021time} we use the \texttt{confseq} python package~\citep{howardconfseq}.} 

As illustrated by~\eqref{eq:apx-width}, the width of $C_t^\apx$ decays at rate $O(\sqrt{\log(t)/t})$. Using a now well-known approach called stitching~\citep{howard2021time}, Theorem~\ref{thm:stitching} provides a variant, $C_t^\stitch$, which decays at the (optimal) iterated logarithm rate of $O(\sqrt{\log\log(t)/t})$ instead. This matches the rate of $C_t^\hrms$. However, as is often the case with stitching, 
the benefit of the second bound is only visible at rather large sample sizes, potentially larger than one may encounter in typical experiments.

In terms of their limiting behavior, we show that 
\begin{equation}
W_t^\stitch \sim 2\sqrt{\frac{\E[\psi_E(|X - \mu|)] \log\log t}{t}} \text{~ while ~} W_t^\hrms \sim \sqrt{\frac{2 \sigma^2 \log\log t}{t}}, 
\end{equation}
where $W_t^\stitch$ is the halfwidth of $C_t^\stitch$ and $W_t^\hrms$ that of $C_t^\hrms$. 
Thus, $C_t^\hrms$ is always tighter in the limit, though $C_t^\stitch$ typically outperforms it at smaller sample sizes. Note that $W_t^\hrms$ matches the law of the iterated logarithm~\citep{darling1967confidence,stout1970hartman} in terms of constants. In other words, for CSs which decay at rate $O(\sqrt{\log\log(t)/t})$, the optimal limiting width (when multiplid by $\sqrt{\log\log(t)/t}$) is $\sqrt{2\sigma^2}$. Thus, the notion of being `asymptotically sharp' differ for CSs which decay at rate $O(\sqrt{\log(t)/t})$ (for which limiting width $\sigma$ is achievable) and those that decay and iterated logarithm rates. 
%We give an overview of the limiting widths of the various bounds discussed in this paper in Appendix~\ref{sec:limiting-widths}. 

Table~\ref{tab:limiting-widths} summarizes the limiting widths of the bounds (for scalars) discussed throughout this paper. The bounds $C_t^\obmix$ and $C_t^\obstitch$ are the oracle Bernstein bounds discussed in Appendix~\ref{sec:oracle-bernstein} and $C_t^\rob$ is Robbins' mixture discussed in Appendix~\ref{sec:robbins-mixture}. The remaining of the bounds are discussed in the main paper. 

\begin{table}[h!]
    \centering
    \small 
    \begin{tabular}{c|c|c|c}
    \emph{Closed-form CS} & \emph{Setting} & $\sigma$-known? & \emph{Asymptotic Width} \\ \hline 
    $C_t^\apx$ (Thm~\ref{thm:closed-form}) & Bounded & \xmark & $\sqrt{\frac{ 2\E[\psi_E(|X - \mu|)]\log t}{t}}$  \\ 
    $C_t^\stitch$ (Thm~\ref{thm:stitching}) & Bounded &\xmark & $\sqrt{\frac{4\E[\psi_E(|X - \mu|)] \log\log t}{t}}$ \\
    $C_t^\prpl$ \citep[Thm 2]{waudby2024estimating}  & Bounded & \xmark & $\sqrt{\frac{\sigma^2\log(2/\alpha)\log t}{8\cdot t}}\log\log t$ \\ 
    $C_t^\prpl$ \citep[Thm 2]{waudby2024estimating}  & Bounded & \xmark & $\sqrt{\frac{\sigma^2\log t}{8\cdot t}}\log\log t$ \\ 
    $C_t^\hrms$ \citep[Thm 1]{howard2021time} & Bounded & \xmark & $\sqrt{\frac{2\sigma^2 \log\log t}{t}}$ \\ 
    $C_t^\hoeff$ \citep[Prop 1]{waudby2024estimating}  & $\sigma$-sub-Gaussian &\checkmark & $\sqrt{\frac{\sigma^2\log(2/\alpha)\log t}{8\cdot t}}\log\log t$ \\
    $C_t^\hoeff$ \citep[Prop 1]{waudby2024estimating}  & $\sigma$-sub-Gaussian &\checkmark & $\sqrt{\frac{\sigma^2\log t}{8\cdot t}}\log\log t$ \\
    $C_t^\rob$ \citep{robbins1970statistical}, App~\ref{sec:robbins-mixture} & $\sigma$-sub-Gaussian & \checkmark &$\sqrt{\frac{\sigma^2\log(t)}{t}}$ \\
    $C_t^{\obmix}$ \citep[Prop 9]{howard2021time} & Bernstein, $c=1$ & \checkmark & $\sqrt{\frac{\sigma^2\log(t)}{t}}$ \\ 
    $C_t^{\obstitch}$ \citep[Thm 1]{howard2021time}& Bernstein, $c=1$ &\checkmark & $\sqrt{\frac{2\sigma^2\log\log t}{t}}$\\
    \end{tabular}
    \caption{Summary of the limiting widths of the bounds discussed in this paper. The data are assumed to be iid with mean $\mu$ and variance $\sigma^2$ in all cases except for $C_t^\rob$ and $C_t^\hoeff$, in which case $\sigma$ is the sub-Gaussian variance proxy. `Bounded' refers to bounded in $[0,1]$. $C_t^\obmix$ and $C_t^\obstitch$ are the two oracle Bernstein bounds of \citet{howard2021time}, discussed in further detail in Appendix~\ref{sec:oracle-bernstein}. 
    $C_t^\hoeff$ is a very mild generalization of \citep[Prop 1]{waudby2024estimating}; see Appendix~\ref{sec:hoeff}. We include $C_t^\prpl$ and $C_t^\hoeff$ twice as the limiting width depends on the choice of parameters; see Appendices~\ref{app:wsr-width} and \ref{sec:hoeff} for longer discussions. 
}
    \label{tab:limiting-widths}
\end{table}

We also extend our results to the matrix case, providing a CS for the maximum eigenvalue of a stream of symmetric $d\times d$ matrices with eigenvalues in $[0,1]$. This CS has similar properties to $C_t^\apx$: It handles time-varying conditional means and scales independently of $\alpha$ asymptotically. We compare our result with the recent matrix empirical Bernstein CS  of \citet{wang2024sharp}. 

Let us briefly describe the technical details behind our methodology. Many recent empirical Bernstein bounds are derived based on a bivariate inequality given by \citet{howard2021time}, itself an extension of an inequality of \citet{fan2015exponential}. Our approach is based on an extension of the latter, which leads to a version the former in which the role of the two variables are ``flipped.'' This small change opens the door to different techniques for optimizing the parameters of the bound and, in particular, allowing us to more easily leverage the method of mixtures---a technique pioneered by \citet{wald1947sequential} and \citet{robbins1970statistical}. 

\paragraph{Contributions.} To summarize the above discussion, we might enumerate our contributions as follows: 
\begin{enumerate}
    \item We present a new empirical Bernstein CS $C_t^\apx$ for bounded random variables which tracks a time-varying mean. Empirically, $C_t^\apx$ is the tightest (up to reasonable sample sizes) such CS that we've found. 
    \item We analyze its limiting width, showing that $C_t^\apx$ is asymptotically tighter than the closed-form CS of \citet{waudby2024estimating}. We also analyze the limiting width of several other popular CSs---see Table~\ref{tab:limiting-widths}---providing an analysis which is thus far mostly absent from the literature but that we believe is important for understanding the behavior of these CSs. 
    \item We extend our results in two ways by (i) providing a version of $C_t^\apx$ which decays at an iterated logarithm rate (theoretically optimal but often empirically less powerful), and (ii) providing a CS for the maximum eigenvalue of symmetric matrices. 
\end{enumerate}

\paragraph{Outline.} 
First we present the bounds of \citet{waudby2024estimating} and \citet{howard2021time}. 
The treatment of our modification of Fan's and Howard's inequalities is then given in Section~\ref{sec:new-nsm}. A new family of bounds based on these modifications is presented in Section~\ref{sec:mom} along with experimental and theoretical results. In particular, Section~\ref{sec:closed-form} provides our closed-form bound $C_t^\apx$ and Section~\ref{sec:asymp-width} analyzes its asymptotic width when assuming that the data are iid. 
Section~\ref{sec:stitching} then provides a CS that decays at rate $O(\sqrt{\log\log(t)/t}$ instead of $O(\sqrt{\log(t)/t}$. 
Section~\ref{sec:matrix} extends these results to the matrix setting.

\subsection{State-of-the-art closed-form CSs}
\label{sec:sota}
The closed-form CS of \citet{waudby2024estimating} reads as follows. Consider a sequence $(X_t)_{t\geq 1}$ of random variables drawn from some unknown distribution $P$ over $[0,1]^\infty$ with conditional mean $\mu$. (Note that they require that $\mu$ does not change with time; our new methods will omit this requirement.) Let $(\lambda_t)$ be a predictable\footnote{That is, $\lambda_t$ is computable given $X_1,\dots,X_{t-1}$, i.e.\ $\lambda_t$ is $\calF_{t-1}=\sigma(X_1,\dots,X_{t-1})$ measurable.} sequence of real values in $(0,1)$ and $(\hatX_t)$ a predictable sequence in $[0,1]$. 
Then,  for all $\alpha\in(0,1)$, $P(\forall t\geq 1: \mu \in C_t^\prpl) \geq 1-\alpha$, where 
\begin{equation}
\label{eq:wsr-prpl}
    C_t^\prpl := \bigg(\frac{\sum_{i\leq t} \lambda_iX_i}{\sum_{i\leq t}\lambda_i} \pm \frac{\log(2/\alpha) + \sum_{i\leq t} (X_i - \hatX_i)^2 \psi_E(\lambda_i)}{\sum_{i\leq t} \lambda_i}\bigg), 
\end{equation}
and $\psi_E(\lambda) = -\log(1-\lambda)-\lambda$ is the CGF of a centered unit-rate negative-exponential random variable~\citep{howard2020time}. 
While the bound holds for any predictable process $(\Xhat_t)$, \citet{waudby2024estimating} opt for $\Xhat_{t} = t^{-1}( 1/2 + \sum_{i< t}X_i)$, $t\geq 1$, as the (smoothed) empirical mean.
We will discuss later how $(\lambda_t)$ should be set. 

% The second CS of \citet{waudby2024estimating} is motivated by the framework of `testing by betting'~\citep{shafer2021testing}. Consider a process $\calK_t(m,\lambda_1^t) = \prod_{i\leq t} ( 1 + \lambda_k(X_k - m))$ for a predictable sequence $(\lambda_t)$ where $\lambda_k \in [-1/(1-m), 1/m]$. This is a martingale for $m=\mu$, implying that $C_t = \{m\in[0,1]: \calK_t(m) < 1/\alpha\}$ is a $(1-\alpha)$-CS for $\mu$. (Again, $\mu$ is constant.)

% The process $(\calK_t(m,\lambda))$ can be seen as the amount of accumulated capital of a gambler who is betting against the mean being $m$. When  $m\neq \mu$, a clever betting strategy (i.e., choice of $\lambda_t$) can generate exponential wealth, whereas when $m=\mu$ no betting strategy can ever accumulate capital $1/\alpha$ with probability more than $1-\alpha$ by Ville's inequality. See \citet{ramdas2023game} for more discussion of game-theoretic techniques for statistical inference. 

% Waudby-Smith and Ramdas form their tightest CS by running two capital processes simultaneously, one optimizing for $m<\mu$ and the other for $m>\mu$. Any convex combination of these two processes remains a martingale, and upper bounds their maximum. Thus, for $\theta\in(0,1)$, they take 
% \begin{equation}
% \label{eq:wsr-betting}
%     C_t^\bet := \left\{m\in[0,1]: \max\left\{\theta \calK^t(m,\lambda_1^t), (1-\theta)\calK_t(m,-\lambda_1^t)\right\}<1/\alpha\right\}.
% \end{equation}
% This CS cannot be solved in closed-form but it provably forms an interval, and thus may be quickly approximated. 

Meanwhile,  the closed-form bound of \citet{howard2021time} allows for a time-varying average conditional mean $\mu_t$. Let $\hatV_t =(\sum_{i\leq t}(X_i - \Xhat_i)^2) \vee 1$, where $(\Xhat_t)$ is once again a predictable sequence taking values in $[0,1]$. For any $\eta>1$, $s>1$, define 
\[H_t := s\log\log(\eta \hatV_t) + \log\left(\frac{2\zeta(s)}{\alpha \log^s(\eta)}\right),\] 
where $\zeta(s)=\sum_{k=1}^\infty 1/k^s$ is the Riemann zeta function. 
Then $P(\forall t\geq 1: \mu_t\in C_t^\hrms) \geq 1-\alpha$, where 
\begin{equation}
\label{eq:hrms}
    C_t^\hrms = \left(\Xhat_t \pm \frac{k_1(\hatV_t H_t)^{1/2} + k_2 H_t}{t}\right),
\end{equation}
and $k_1 = (\eta^{1/4} + \eta^{-1/4}) / \sqrt{2}$, $k_2 = (\sqrt{\eta}+1)/2$. We have omitted a few of the parameters in the most general version the CS of \citet{howard2021time} for clarity,  but have chosen them as suggested to get the best results. See \citet[Theorem 1]{howard2021time} for the most general statement. See also \citet[Equation (24)]{howard2021time} for an optimized empirical Bernstein CS for $\alpha=0.05$, which we will use in some of our experiments below.

\subsection{Notation and conventions}
\label{sec:prelims}

Throughout, we let $(X_t)_{t\geq 1}$ be a sequence of random variables taking values in $[0,1]$. We let $\mu_t = t^{-1}\sum_{j\leq t} \E[X_j |\calF_{j-1}]$ be the average conditional mean at time $t$, where $\calF_j = \sigma(X_1,\dots,X_j)$ is the $\sigma$-algebra generated by the first $j$ observations.  
Our results place no assumptions on $\mu_t$ other than measurability. That is, we require that $\E[X_t|\calF_{t-1}]$ be defined for each $t\geq 1$. 
Of course, this encapsulates the case when the observations are iid or have a constant conditional mean. 
%Let $\calP_\mu$ be the set of distributions over sequences of observations $(X_t)_{t=1}^\infty\subset[0,1]$ with conditional mean $\mu$, i.e., $\E_{t-1}[X_t] \equiv \E[X_t | \calF_{t-1}] = \mu$ for all $t\geq 1$, where $\calF_t = \sigma(X_1,\dots,X_t)$ is the information available at time $t$. Requiring a constant conditional mean is strictly weaker than the iid assumption. Such dependence structure is important in many applications, such as online learning and sequential A/B testing. 

%For $\alpha\in(0,1)$ and a prespecified sample size $n$, a $(1-\alpha)$-CI is a  $\calF_n$-measurable set $C_n$ which satisfies $P(\mu \in C_n)\geq 1-\alpha$. A $(1-\alpha)$-CS is a sequence of sets $(C_t)_{t\geq 1}$ such that $C_t$ is $\calF_{t}$-measurable and $P(\forall t\geq 1: \mu \in C_t)\geq 1-\alpha$. That is, a confidence sequence is the time-uniform equivalent of a confidence interval.  In this work we are focused on CSs. 

For a set of distributions $\calP$, a $\calP$-supermartingale is a stochastic process $(M_t)$ that is adapted to $(\calF_t)$ (i.e., $M_t$ is $\calF_t$-measurable) and $\sup_{P\in\calP}\E_P[M_t|\calF_{t-1}]\equiv \E_\calP[M_t |\calF_{t-1}] \leq M_{t-1}$.  
Nonnegative supermartingales lie at the heart of constructing confidence sequences due to Ville's inequality~\citep{ville1939etude}, which is a time-uniform version of Markov's inequality. 
Ville's inequality states that for a nonnegative $\calP$-supermartingale $(M_t)$, $\sup_{P\in\calP}P(\exists t\geq t_0: M_t\geq 1/\alpha)\leq \alpha\E_\calP[M_{t_0}]$ for any $\alpha\in(0,1)$ and $t_0$. A common strategy for constructing CSs for a parameter $\theta_t$ is thus to find a nonnegative $\calP$-supermartingale $(N_t)$ where (i) $\calP$ is the set of distributions with this parameter, (ii) $N_t$ is a function of $\theta_t$, and (iii) $N_0 = 1$, and then apply Ville's inequality. (In other words, we look for suitable \emph{e-processes}~\citep{ramdas2025hypothesis}). This is also the approach we take here. 

In our case, $\calP$ will be the set of distributions with average conditional mean $\mu_t$. We will, therefore, simply refer to nonnegative $\calP$-supermartingales as nonnegative supermartingales. Further, if we say that $(C_t)$ is a $(1-\alpha)$-CS for $\mu_t$ it should be assumed unless stated otherwise that it starts from time $t=1$.

We will occasionally use asymptotic notation in an almost sure sense. In particular, for two functions $f,g:\mathbb{N}_0 \to \Re$, we write $f = \Oas(g)$ (resp., $f = \oas(g)$) to mean that $f = O(g)$ (resp., $f = o(g)$) $P$-almost surely, where the measure $P$ will be clear from context. We also write $f\sim g$ to mean that $\lim_{t\to\infty} f(t)/g(t) =1$. If $f(t)$ and $g(t)$ are random variables, then the limit should be understood to hold almost surely. 
All logarithms refer to the natural logarithm unless otherwise specified.

\section{A new empirical Bernstein supermartingale}
\label{sec:new-nsm}

The empirical Bernstein bounds of \citet{howard2021time} and \citet{waudby2024estimating} are constructed by first considering a deterministic inequality first given by \citet{fan2015exponential}, commonly referred to as \emph{Fan's inequality}:    
\begin{equation}
\label{eq:fans-inequality}
    \textnormal{Fan's inequality:} \quad \text{For all }\lambda\in [0,1) \text{ and }\xi\geq -1, \; \exp\{\lambda\xi - \psi_E(\lambda)\xi^2\} \leq 1 + \lambda\xi, 
\end{equation}
where $\psi_E(\lambda) = -\log(1-\lambda) - \lambda$. It is tempting to apply Fan's inequality with $\xi = X_t-\E_{t-1}X_t$ and then take expectations, in which case one obtains that $\E_{t-1} \exp\{\lambda (X_t-\mu) - \psi_E(\lambda) (X_t-\mu)^2\} \leq 1$. While the process with increments $\exp\{ \lambda(X_t -\mu) - \psi_E(\lambda)(X_t - \mu)^2\}$ thus defines a nonnegative supermartingale, the resulting CS cannot be solved in closed form for $\mu$. 

To solve this problem, \citet{howard2021time} proved that one can replace $(X_t - \mu)^2$ with $(X_t - \Xhat_t)^2$ after taking expectations, where $\Xhat_t\in[0,1]$ is predictable (i.e., $\calF_{t-1}$-measurable). This gives rise to \emph{Howard's inequality}:
\begin{equation}
    \text{Howard's inequality:} \quad \E_{t-1}\exp\left\{\lambda (X_t -\mu) -\psi_E(\lambda)(X_t - \hatX_{t})^2\right\} \leq 1, 
\end{equation}
where again $\lambda\in[0,1)$. 
The empirical Bernstein bound \citet{howard2021time} and the closed-form bound of \citet{waudby2024estimating} follow from Howard's inequality after optimizing for $\lambda$. \citet{howard2021time} ``stich'' $\lambda$ over geometrically spaced epochs (see Section~\ref{sec:stitching} where we do something similar), while \citet{waudby2024estimating} choose a distinct $\lambda_t$ at every timestep $t$, which has come to be known as the ``method of predictable plug-ins'' and can be traced back to \citet{wald1947sequential}.  

%Both \citet{howard2021time} and \citet{waudby2024estimating} apply Fan's inequality with $\xi = Y_t - \delta_t$. Applying expectations and rearranging eventually give rise to the two results (which deviate in terms of how they choose $\lambda$). The width of their bounds thus rely on $\xi^2 = (X_t - \hatX_t)^2$, giving them their empirical flavor. 

Our approach is motivated by the following question: Can we reverse the roles of $\lambda$ and $\xi$ in Fan's inequality? That is, can we take $\lambda = X_t - \hatX_t$? The empirical term would thus show up in the argument of $\psi_E$ and we would be able to optimize over the parameter $\xi$. This is appealing, since the bound looks sub-Gaussian in terms of $\xi$. 

Of course, we cannot directly apply Fan's inequality with $\lambda = X_t - \hatX_t$, since $\lambda$ must be nonnegative. This motivates the following straightforward modification of Fan's inequality to handle negative values of $\lambda$. 

\begin{lemma}[Modified Fan's Inequality]
\label{lem:extended-Fan}
For all $\lambda\in (-1,1)$ and $\xi\in[-1,1]$,
\begin{equation}
\label{eq:extended-Fan}
  \exp\left\{\lambda\xi  - \psi_E(|\lambda|)\xi^2\right\}\leq 1 + \lambda\xi .  
\end{equation}
\end{lemma}
\begin{proof}
If $\lambda\geq 0$, this reduces to~\eqref{eq:fans-inequality}. If $\lambda<0$, set $\lambda^\circ = -\lambda \in (0,1)$ and $\xi^\circ = -\xi \geq -1$. Applying~\eqref{eq:fans-inequality} with $\xi^\circ$ and $\lambda^\circ$ gives $\exp\{ \lambda \xi  - \psi_E(|\lambda|) \xi^2\} = \exp\{ \lambda^\circ \xi^\circ - \psi(\lambda^\circ)\xi^2\} \leq 1 + \lambda^\circ\xi^\circ = 1+\lambda\xi$. 
\end{proof}

Just as Howard's inequality can be derived from Fan's inequality, we can derive a modified Howard's inequality from our modified Fan's inequality. 

\begin{lemma}[Modified Howard's inequality]
\label{lem:modified-howard}
Let $X$ be a random variable in $[0,1]$ with law $P$ and let $\Xhat\in[0,1)$ be nonrandom. Then, for any $\xi\in[-1,1]$, 
%Let $(X_t)$ be a stochastic process taking values in $[0,1]$ and let $(\Xhat_t)$
\begin{equation}
    \E_P \exp\left\{ \xi(X - \E_P[X]) - \xi^2\psi_E(|X-\hatX|)\right\} \leq 1. 
\end{equation}
\end{lemma}
\begin{proof}
    Let $Y = X - \E[X]$ and $\delta = \hatX - \E[X]$   
    so that $Y - \delta = X - \hatX \in (-1,1)$. 
Applying~\eqref{eq:extended-Fan} with $\lambda = Y - \delta$, we obtain $\exp\{ (Y -\delta)\xi - \psi_E(|X - \hatX|)\xi^2\} \leq 1 + (Y- \delta)\xi$. 
Note that this is a legal choice for $\lambda$ since we've assumed that $\hatX$ is strictly less than 1.  
Taking expectations and noting both that $\delta$ is predictable and that $\E[Y] = 0$, we have 
\begin{align*}
    \E_P\exp\left\{ \xi Y - \psi_E(|X-\hatX|)\xi^2\right\} \leq e^{\delta \xi}(1 - \delta\xi)\leq 1, 
\end{align*}
as claimed. 
\end{proof}

Lemma~\ref{lem:modified-howard} quickly leads to a new empirical Bernstein-style supermartingale. In particular, if $(\Xhat_t)$ is a predictable sequence in $[0,1)$  and $(\xi_t)$ is a predictable sequence in $[-1,1]$, the process defined by 
\begin{equation}
 B_t(\xi_1^t)=\prod_{i\leq t}\exp\left\{ \xi_i (X_i - \E_{i-1}X_i) - \xi_i^2\psi_E(|X_i-\hatX_i|)\right\}.  
\end{equation}
is a nonnegative supermartingale. 
%In particular, if $(X_t)$ is drawn from some $P\in \calP_\mu$ then 
% \begin{equation}
% \label{eq:Bt-nsm}
% B_t(\mu,\xi_1^t) := \prod_{i\leq t}\exp\left\{ \xi_i (X_i - \mu) - \psi_E(|X_i-\hatX_i|)\xi_i^2\right\},
% \end{equation}
% is a nonnegative $P$-supermartingale. 
From here, one could follow the predictable plug-in strategy of \citet{waudby2024estimating}, choosing $\xi_t$ based on $X_1,\dots,X_{t-1}$. Such an approach requires that we take $\mu_t$ to be constant, however. Further, even under this assumption, we find that this strategy leads to worse results in our case. 
%We illustrate this approach in Appendix~\ref{sec:prpl} for those interested in the details. 
Instead, we will obtain our CS by choosing $\xi_t=\xi$, and integrating $B_t(\xi)$ over $\xi$.

\section{An empirical Bernstein CS via the method of mixtures} 
\label{sec:mom}

Following an approach which has come to be known as the ``method of mixtures''~\citep{darling1967confidence,howard2021time}, we form another nonnegative supermartingale from $(B_t(\xi))$ by mixing over the parameters $\xi$. 
Fix $\xi = \xi_t$ and rewrite $B_t(\xi)$ as  $B_t(\mu_t, \xi) = \exp\{ \xi(S_t - t\mu_t) - \xi^2 V_t\}$ 
where $S_t = \sum_{i\leq t} X_i$ and $V_t = \sum_{i\leq t} \psi_{E}(|X_i - \hatX_i|)$. For any data-free distribution over $[-1,1]$ with density $\rho$, the mixture  
\begin{equation}
	M_t(\mu_t) := \int_{-1}^1 B_t(\mu_t;\xi) \rho(\xi) \d\xi, 
\end{equation}
is again a nonnegative supermartingale by Fubini's theorem. We may thus form a $(1-\alpha)$-CS for $\mu_t$ by taking our confidence set at time $t$ to be $\{m\in[0,1]: M_t(m) <1/\alpha\}$, since $M_t(\mu_t) < 1/\alpha$ with probability $1-\alpha$ by Ville's inequality. 

A particularly tight CS comes from taking $\rho$ to be the density of a truncated Gaussian distribution over $[-1,1]$ with mean 0 and variance $\kappa^2$. To state the result, first define 
\begin{equation}
	I(y; v) := \int_{-1}^1 \exp\left\{ y\xi - v\xi^2\right\}\d\xi. 
\end{equation}
The function $I(y;v)$ is convex with a minimum at $y=0$. One should think of it as behaving roughly as $f(y) = y^2 + c_v$ for a constant $c_v>0$ depending on $v$ (and decreasing as $v$ grows). Writing the set $\{m: M_t(m) < 1/\alpha\}$ in terms of $I$ leads to the following CS.

\begin{proposition}
\label{prop:mixture-cs-tn}
    Let $(X_t)$ be a sequence of random variables in $[0,1]$ with average conditional mean $\mu_t$.  Let $(\hatX_t)$ be any predictable sequence in $[0,1)$. Fix $\kappa>0$ and let $Z = \Phi(1/\kappa) - \Phi(-1/\kappa)$. Set $S_t = \sum_{i\leq t}X_i$ and $U_t = (2\kappa^2)^{-1} + \sum_{i\leq t} \psi_E(|X_i - \hatX_i|)$.  Then 
    \begin{equation}
    \label{eq:cs-mixture}
	C_t^\mix := \bigg\{ m: I(S_t - tm; U_t) < \frac{\kappa Z \sqrt{2\pi}}{\alpha}\bigg\},
\end{equation}
constitutes a $(1-\alpha)$-CS for $\mu_t$. Further, if $\mu_t \equiv \mu$ is constant over time, the running intersection $(\cap_{j\leq t} C_j^\mix)_{t\geq 1}$ is also a $(1-\alpha)$-CS for $\mu$. 
\end{proposition}

The proof is in Appendix~\ref{proof:mixture-cs-tn}. As in $C_t^\prpl$, we recommend taking $\Xhat_t = t^{-1}(1/2 + \sum_{j<t}X_j)= \muhat_{t-1}$ for all $t\geq 1$.  Notice that since $\hatX_1 = 1/2$, $\hatX_t <1$ for all $t$, ensuring that the requirements of Proposition~\ref{prop:mixture-cs-tn} are met. Further, $\hatX_t$ depends only on $X_1,\dots,X_{t-1}$, and is thus predictable. However, it's worth emphasizing that $\Xhat_t$ might in principle be the result of any algorithm using $X_1,\dots,X_{t-1}$ to predict $X_t$.

\begin{remark}
\label{rem:alternative-dists}
    One can of course consider distributions other than a truncated Gaussian. For instance, if we take $\rho$ to be uniform on $[-1,1]$, then a simple modification of the proof of Proposition~\ref{prop:mixture-cs-tn} shows that 
    \begin{equation}
        C_t^{\unif}:= \bigg\{m: I(S_t - tm; V_t) < \frac{2}{\alpha}\bigg\},
    \end{equation}
    is a $(1-\alpha)$-CS for $\mu$. This CS has the benefit of not needing to specify an additional parameter $\kappa$, unlike $(C_t^\mix)$. However, we find that in practice $(C_t^\mix)$ reliably outperforms $(C_t^\unif)$ (see the experiments in Appendix~\ref{app:experiments}), so we work with the former throughout this section. That said, we emphasize that these two CSs differ only in terms of lower-order constants, so the discussion on asymptotics in Section~\ref{sec:asymp-width} holds for both. Additionally, one can show that $\kappa Z = \kappa(\Phi(1/\kappa) - \Phi(1/-\kappa))\to \sqrt{2/\pi}$ as $\kappa\to\infty$, to $C_t^\mix$ becomes $C_t^\unif$ in the limit. 
\end{remark}

To better understand how $C_t^\mix$ behaves, note that by a change of variables we can write 
\begin{equation}
\label{eq:It-explicit}
	I(y;v) = \frac{\exp\{y^2/(4v)\}}{2\sqrt{v/\pi}}\left(\erf\bigg(\sqrt{v} - \frac{y}{2\sqrt{v}}\bigg) + \erf\bigg(\sqrt{v} + \frac{y}{2\sqrt{v}}\bigg)\right), 
\end{equation}
where
\begin{equation}
    \erf(z) = \frac{2}{\sqrt{\pi}}\int_0^z \exp(-u^2) \d u. 
\end{equation}
Since $I(y;v)$ is minimized at $y=0$, $m^* = S_t/t = \Xbar_t\in C_t^\mix$ as long as $I(0;U_t) < \kappa Z\sqrt{2\pi}/\alpha$. And indeed, as a sanity check we can verify that this final inequality holds for all $t$, implying that $C_t^\mix$ is never empty. To see this, note that $I(y;U_t)$ is decreasing in $U_t$ by~\eqref{eq:It-explicit}. Hence, $I(0;U_t) \leq I(0,1/(2\kappa^2)) = \kappa \sqrt{2\pi}(\erf(\sqrt{1/(2\kappa^2)}) = \kappa Z \sqrt{2\pi} < \kappa Z \sqrt{2\pi}/\alpha$ for $\alpha\in(0,1)$. Here we've used the identity $\Phi(x) = \frac{1}{2} (1 + \erf(\sqrt{x}/2)$ to conclude that $\erf(1/\sqrt{2\kappa^2}) = \Phi(1/\kappa) - \Phi(-1/\kappa)$.

The representation in~\eqref{eq:It-explicit} further implies that $I(y;v)$ is symmetric about $y=0$ and increasing in $|y|$. It follows that $C_t^\mix = (m^* - \widehat{m}, m^* + \widehat{m}) = (\Xbar_t - \widehat{m}, \Xbar_t + \widehat{m})$, 
for some $\widehat{m}>0$. In fact, since $I(y;v)$ is surjective and increasing on $[I(0;v),\infty)$, $\widehat{m}$ solves the equation $I(Z_t - t\widehat{m}; U_t) = \kappa Z \sqrt{2\pi}/\alpha$, and can thus be found by root finding or binary search. Thus, while the interval $C_t^\mix$ cannot be written down explicitly, it may be easily approximated. 

Still, in this work we want to derive closed-form CSs, so let us turn to that task next.

\subsection{Closed-form approximation}
\label{sec:closed-form}

\begin{figure}[t]
    \centering
    \begin{subfigure}{0.32\textwidth}
    \includegraphics[height=3.5cm]{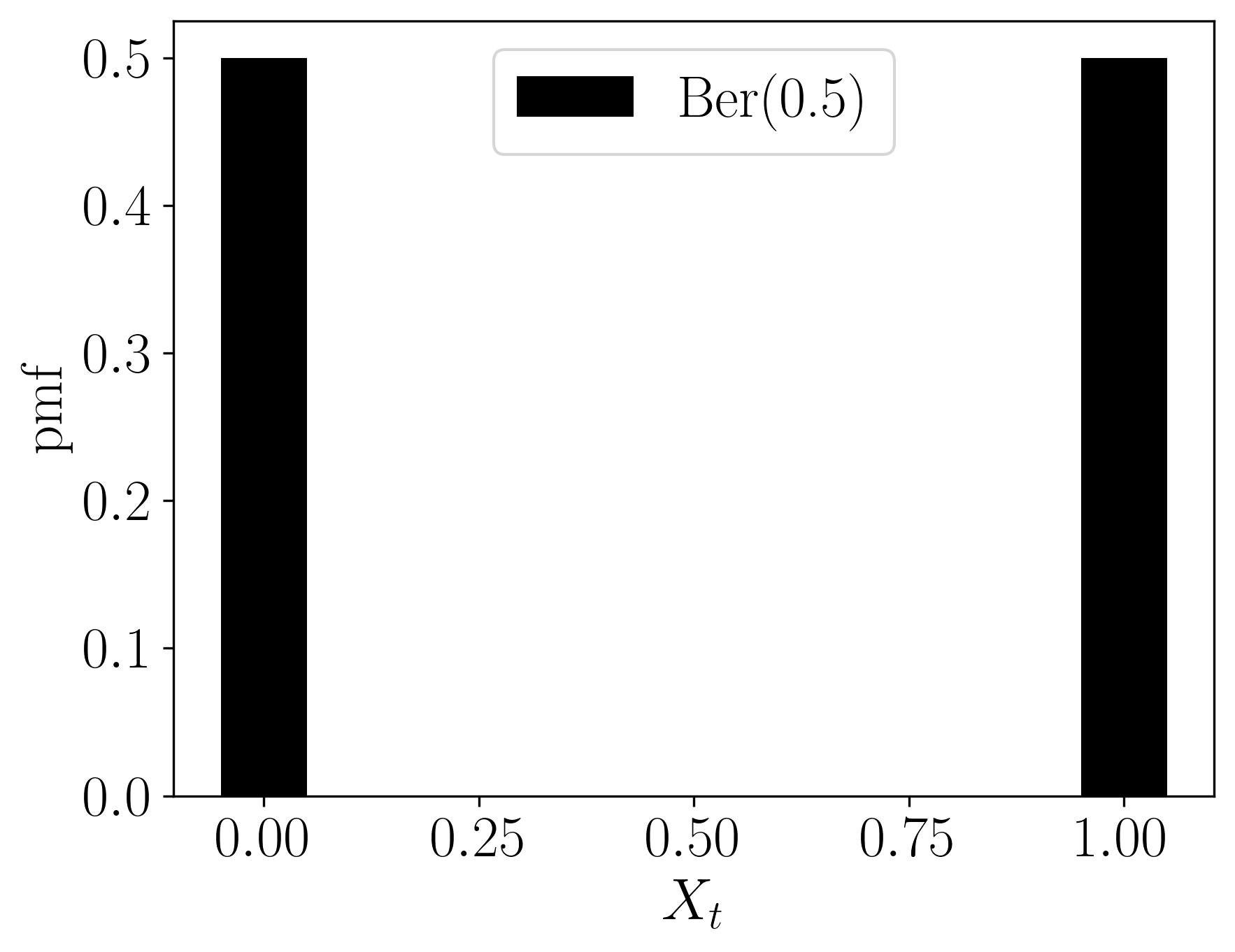}
    \end{subfigure}
    \begin{subfigure}{0.33\textwidth}
    \includegraphics[height=3.5cm]{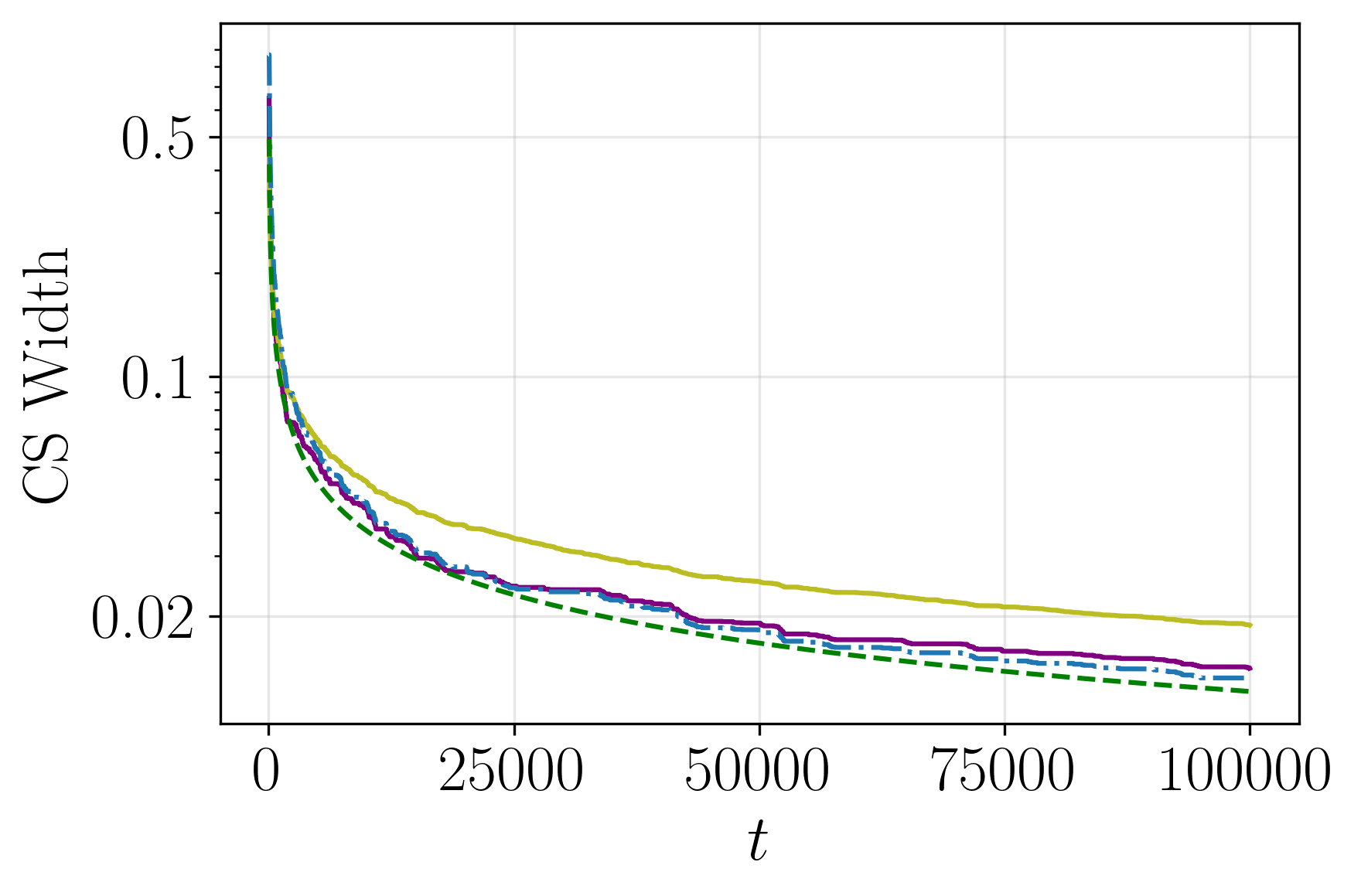}
    \end{subfigure}
    \begin{subfigure}{0.32\textwidth}
    \includegraphics[height=3.5cm]{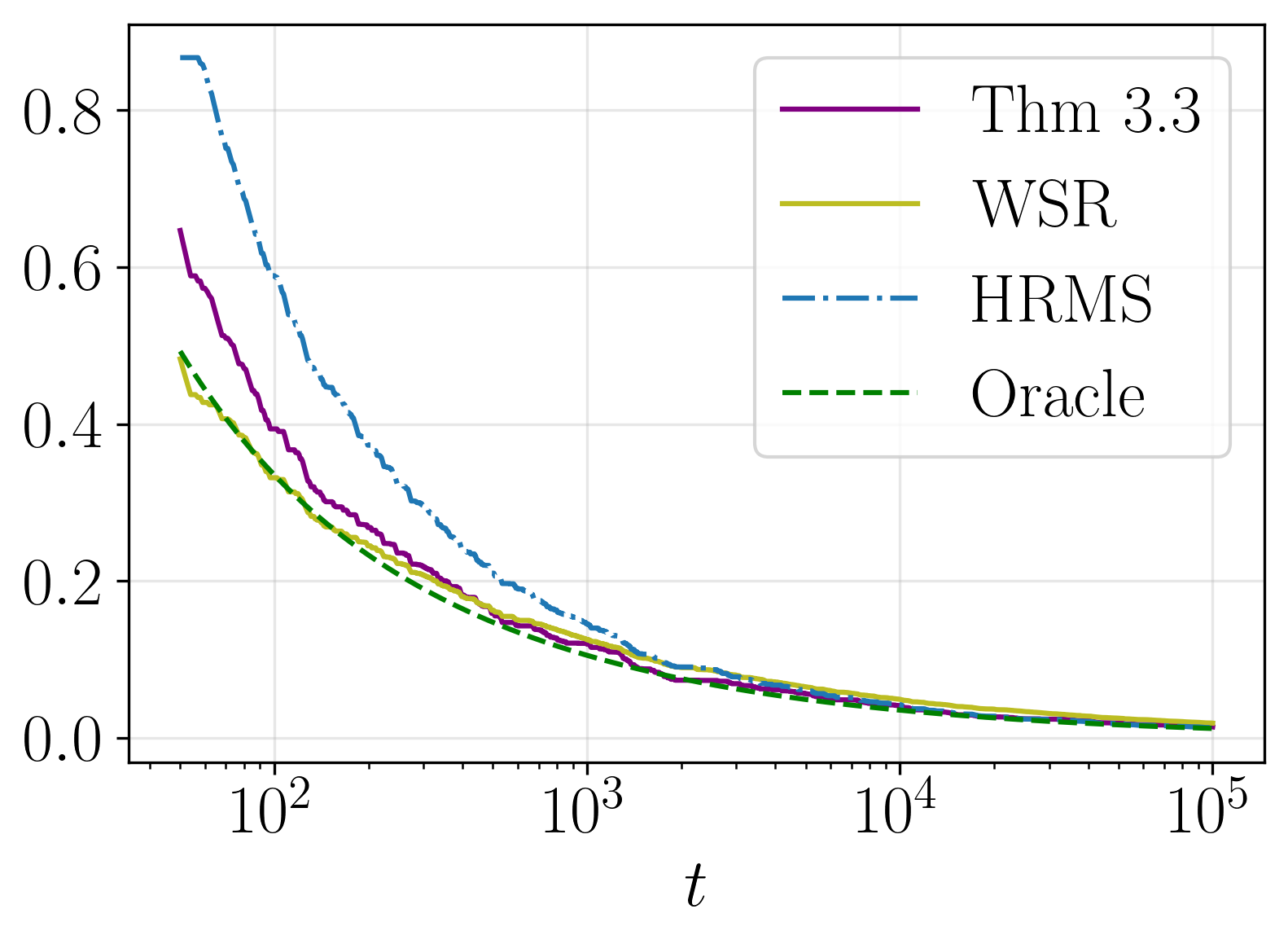}
    \end{subfigure}
    \begin{subfigure}{0.32\textwidth}
    \includegraphics[height=3.5cm]{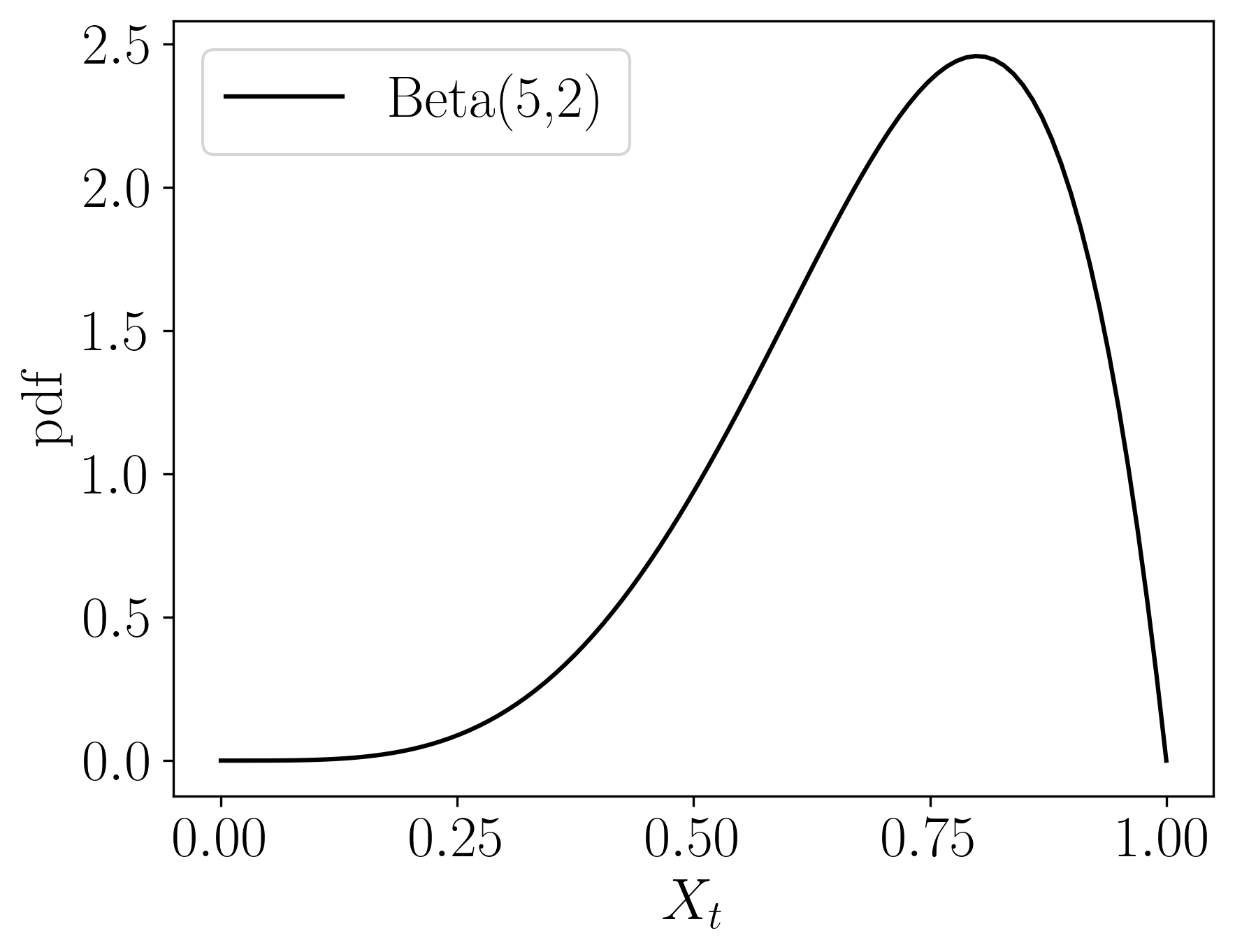}
    \end{subfigure}
    \begin{subfigure}{0.33\textwidth}
    \includegraphics[height=3.5cm]{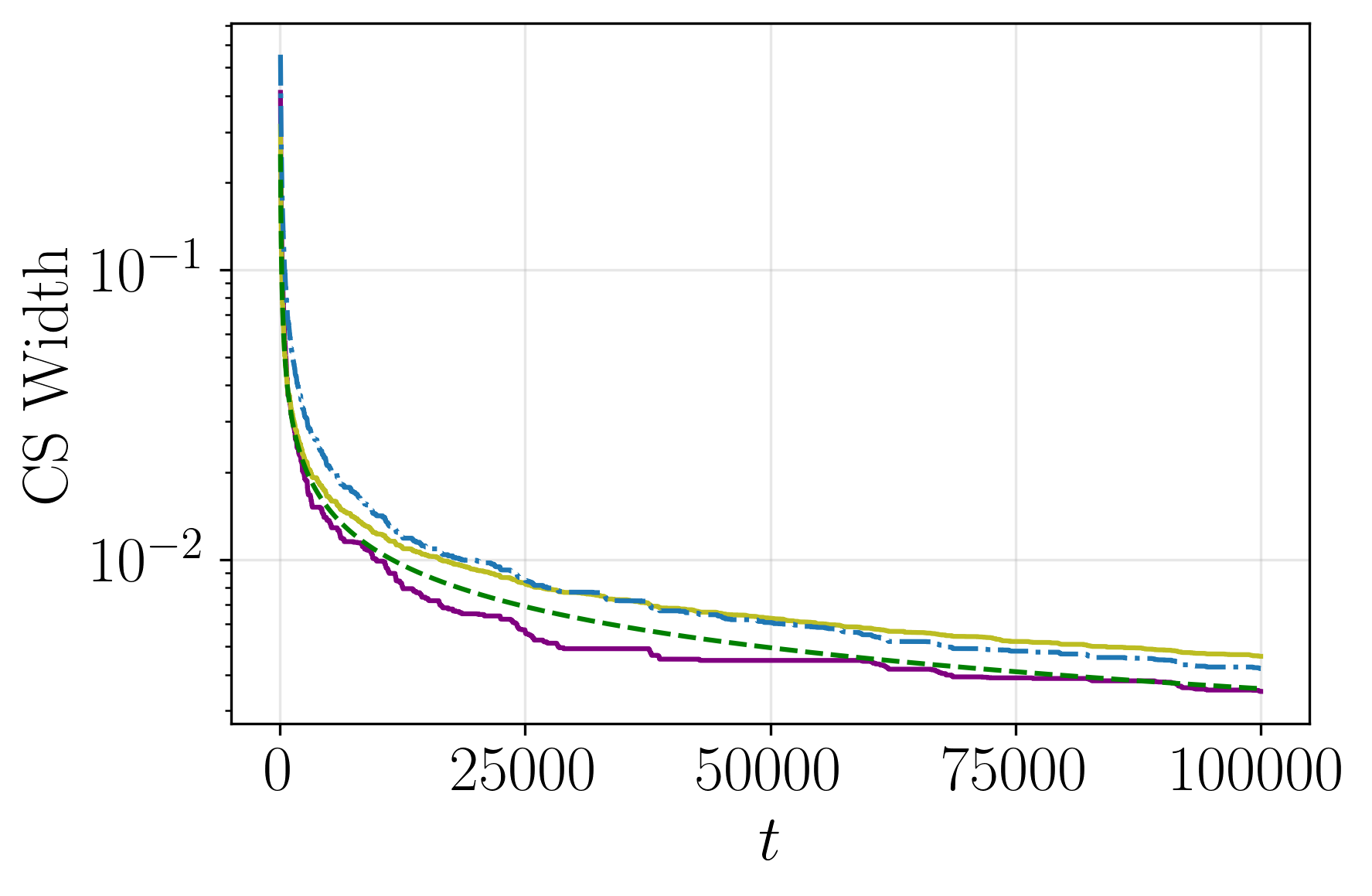}
    \end{subfigure}
    \begin{subfigure}{0.32\textwidth}
    \includegraphics[height=3.5cm]{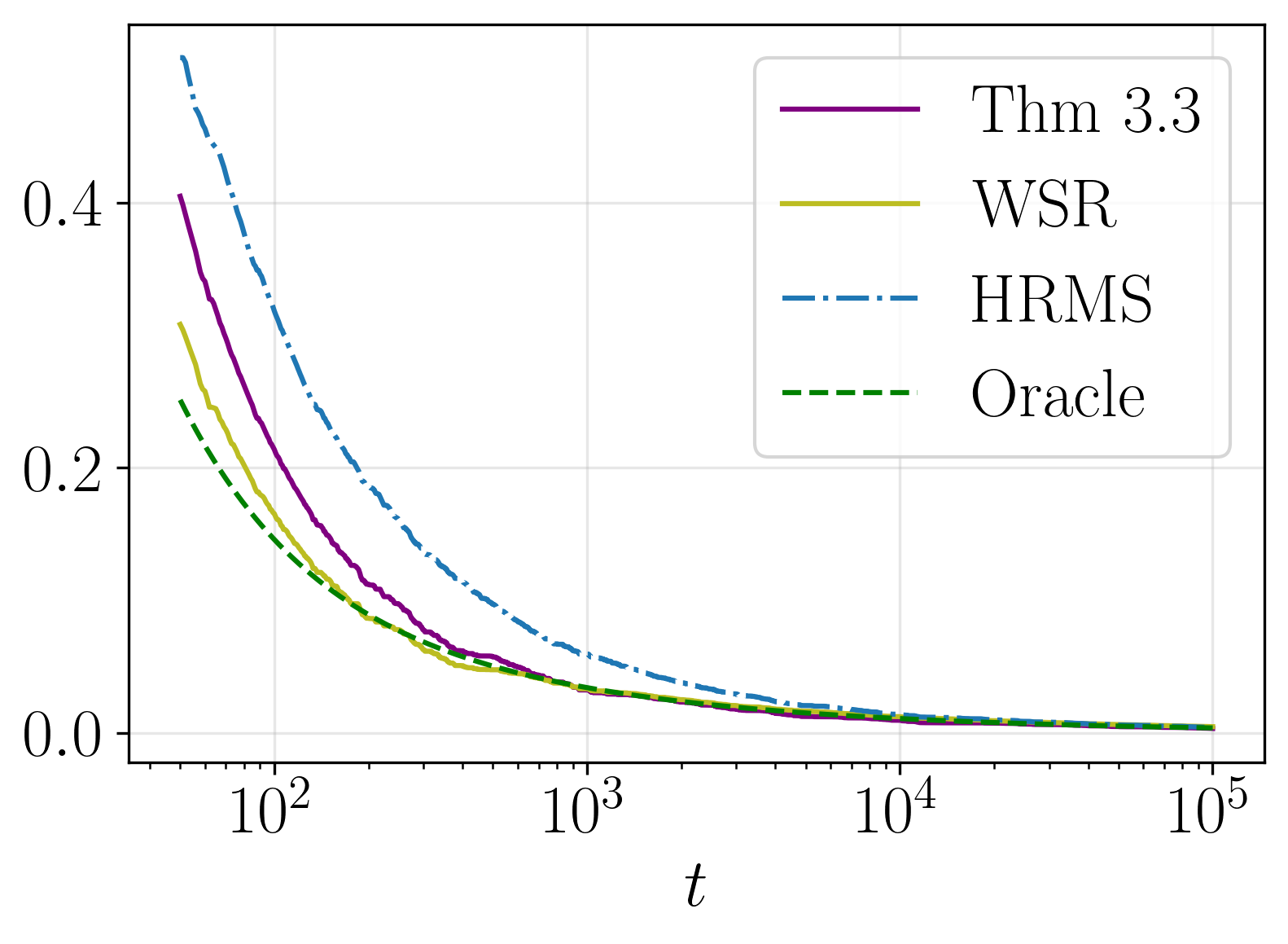}
    \end{subfigure}
    \begin{subfigure}{0.32\textwidth}
    \includegraphics[height=3.5cm]{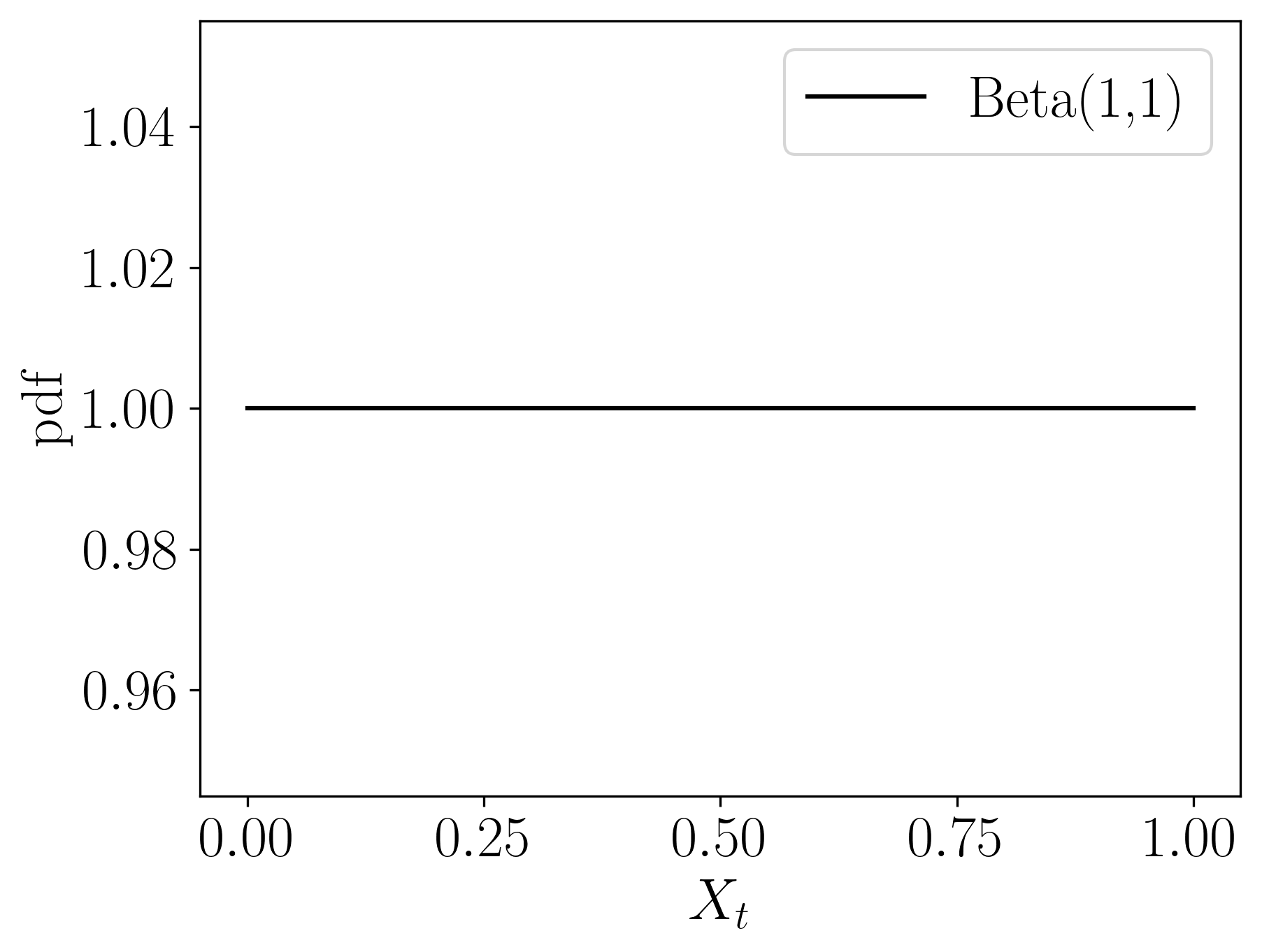}
    \end{subfigure}
    \begin{subfigure}{0.33\textwidth}
    \includegraphics[height=3.5cm]{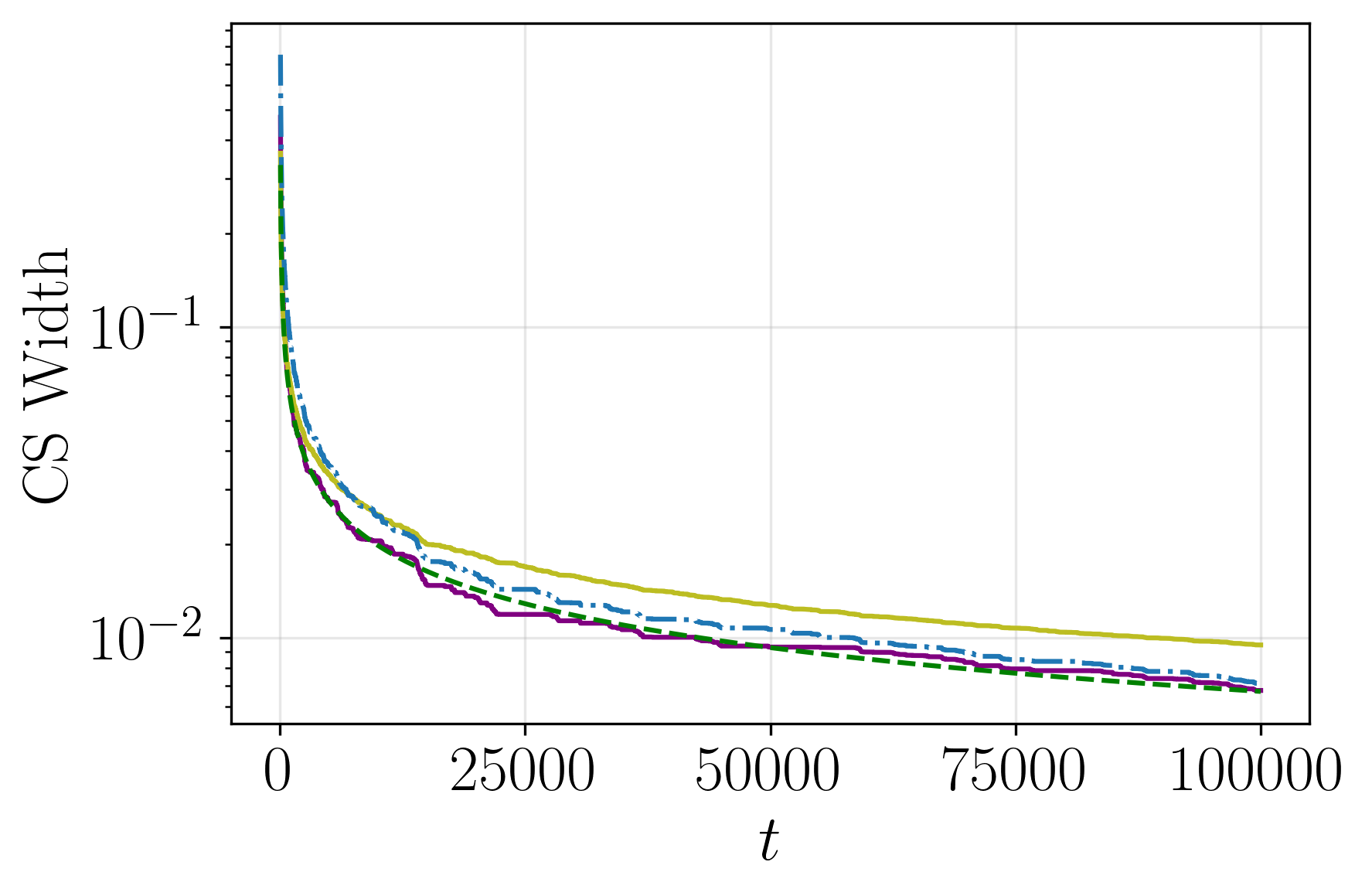}
    \end{subfigure}
    \begin{subfigure}{0.32\textwidth}
    \includegraphics[height=3.5cm]{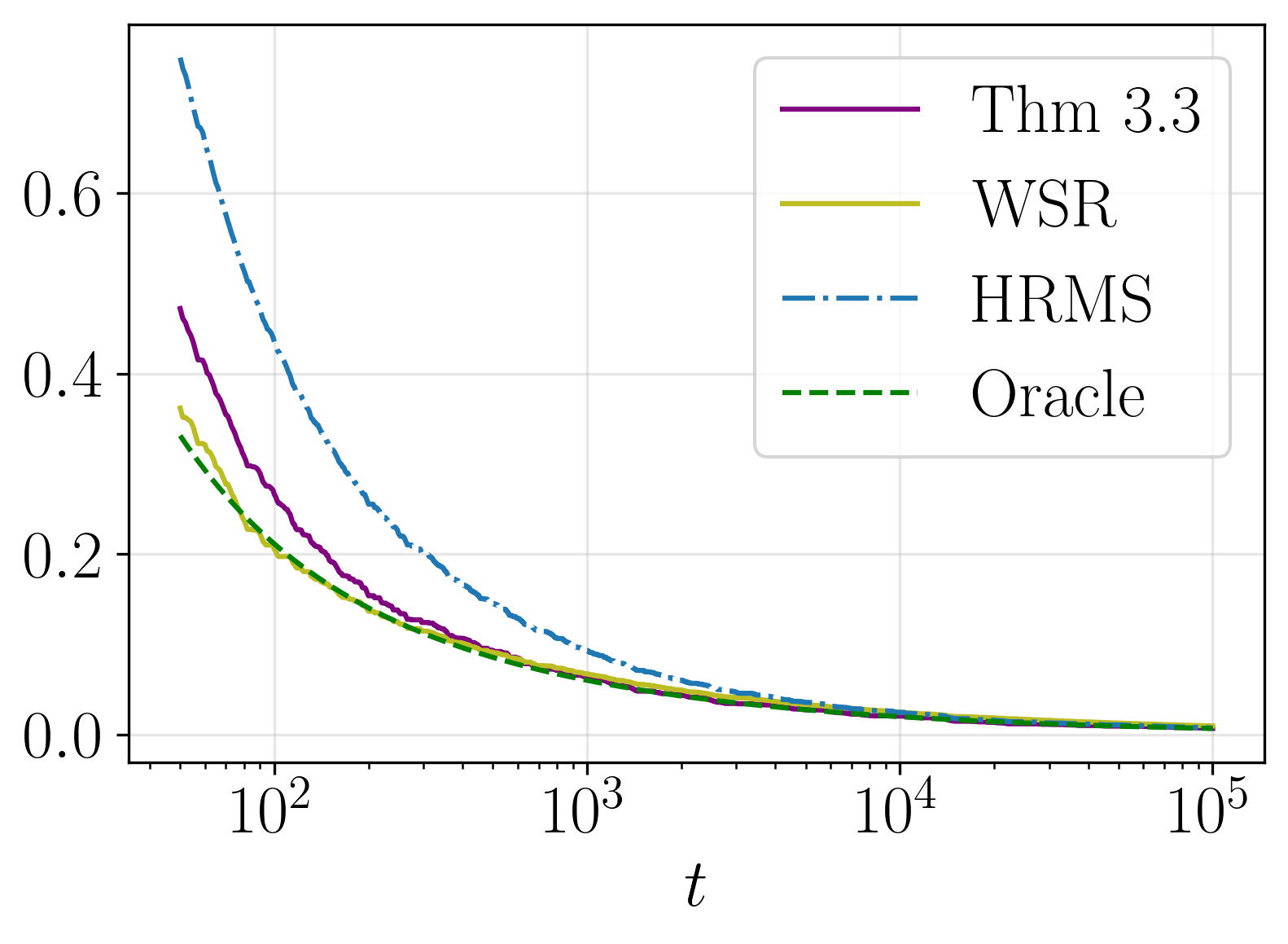}
    \end{subfigure}
    \caption{Comparison of the width of our closed-form  CS $C_t^\apx$ (Theorem~\ref{thm:closed-form}) with that of 
    the closed-form CS of \citet{waudby2024estimating} (WSR) and that of \citet{howard2021time} (HRMS) under the assumption of a constant conditional mean.  
    We fix $\kappa = 0.25$ and $\alpha=0.05$ for all experiments and plot the oracle Bernstein CS of \citet{howard2021time} for comparison. 
    %curve $\sqrt{\log(t)/t}$ for comparison. 
    The middle column plots the width of the CSs with the y-axis on a log-scale. The final column plots the x-axis on log-scale. 
    }
    \label{fig:comparison}
\end{figure}

\begin{figure}
    \centering
    \begin{subfigure}[t]{0.45\linewidth}
        \includegraphics[height=5cm]{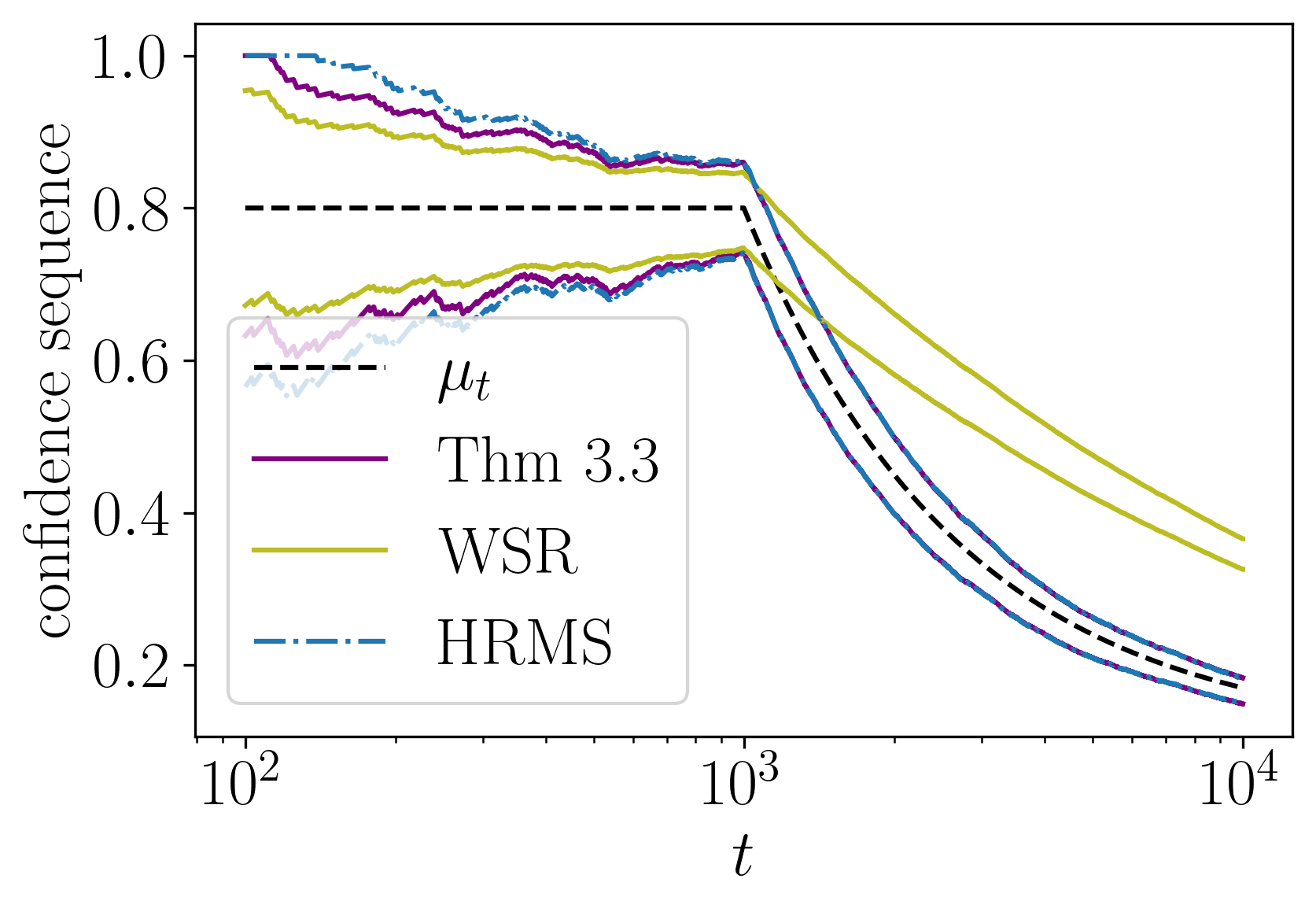} 
    \end{subfigure}
    \begin{subfigure}[t]{0.45\linewidth}
        \includegraphics[height=5cm]{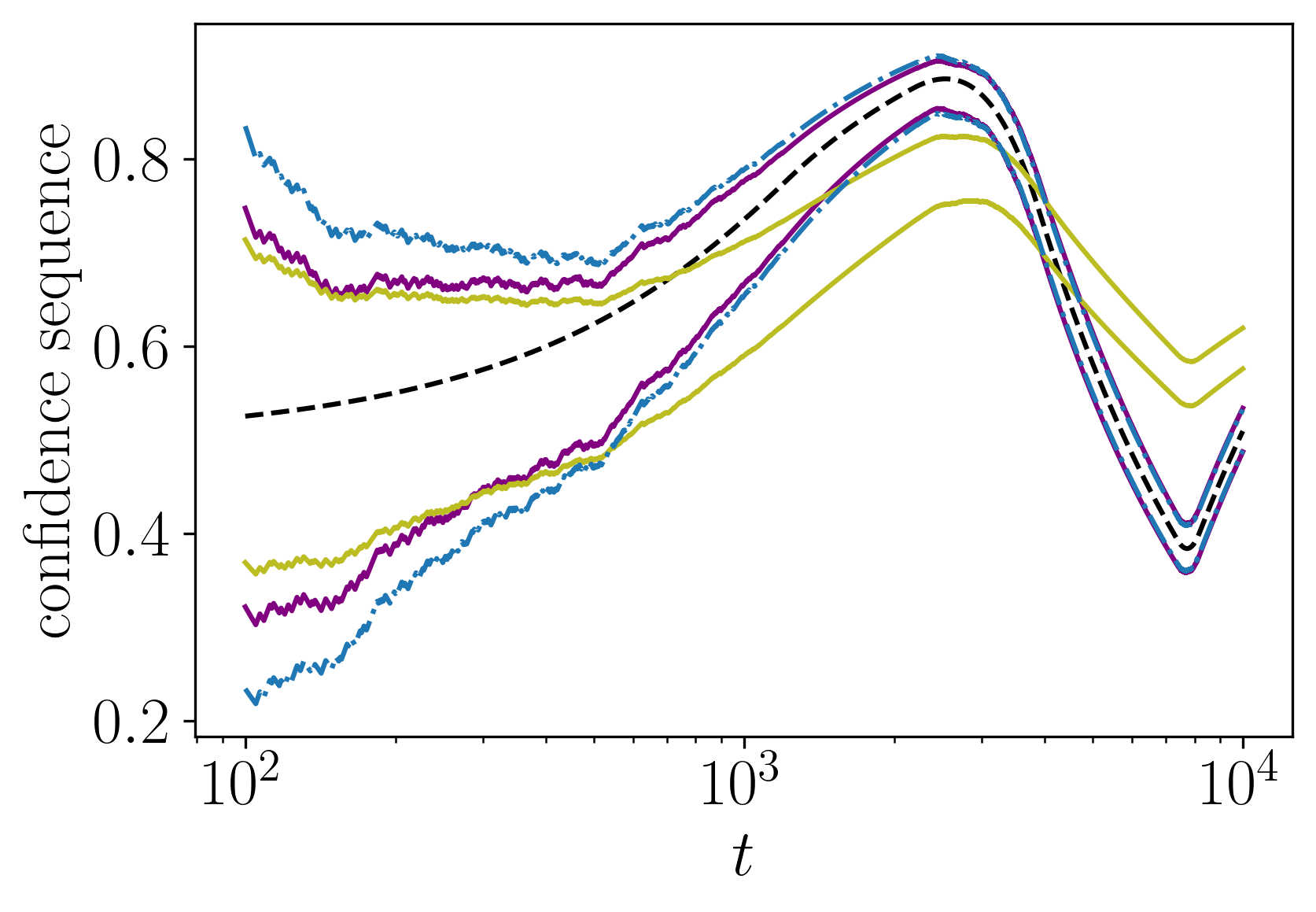}
    \end{subfigure}    
    \caption{The performance $C_t^\apx$ (Theorem~\ref{thm:closed-form}), $C_t^\hrms$ and $C_t^\prpl$ under a time-varying mean (dotted black line). For the left plot we draw the observations iid from $\ber(0.8)$ for $t\leq N/10$, and then for $\ber(0.2)$ thereafter. The second plot generates $p_t$ to target a sinusoidal-like conditional mean, and then draws $X_t \sim \ber(p_t)$. Here $\alpha=0.05$ and we use Theorem~\ref{thm:closed-form} with $\kappa=0.25$.   }
    \label{fig:time-varying}
\end{figure}

To move from Proposition~\ref{prop:mixture-cs-tn} to a closed-form bound, we can write down an explicit CS $(C_t^\apx)$ such that $C_t^\apx\supset C_t^\mix$. This superset only becomes viable after a data-dependent number of samples $t_0$ (i.e., a hitting-time), but $t_0$ is typically quite small in our experience (often less than 100; see Figure~\ref{fig:exact-vs-approx}). Constructing $C_t^\apx$ relies on noticing that the solution $y_t = t(X_t - m)$ to $I(y_t;U_t) = G$ will eventually be less than $U_t$ (this is what defines $t_0$; see Lemma~\ref{lem:yt<Ut} in the appendix). Thus, for $t$ large enough,  the arguments of both erf terms in~\eqref{eq:It-explicit} are nonnegative, allowing us deploy the  lower bound $\erf(x) \geq 1 - \exp(-x^2)$ for $x\geq 0$ and solve for $y_t$ explicitly. The result is Theorem~\ref{thm:closed-form}.  Its proof can be found in Appendix~\ref{proof:mom-closed-form}.

\begin{figure}[t]
	\centering
    \begin{subfigure}[t]{0.45\linewidth}
        \includegraphics[height=5cm]{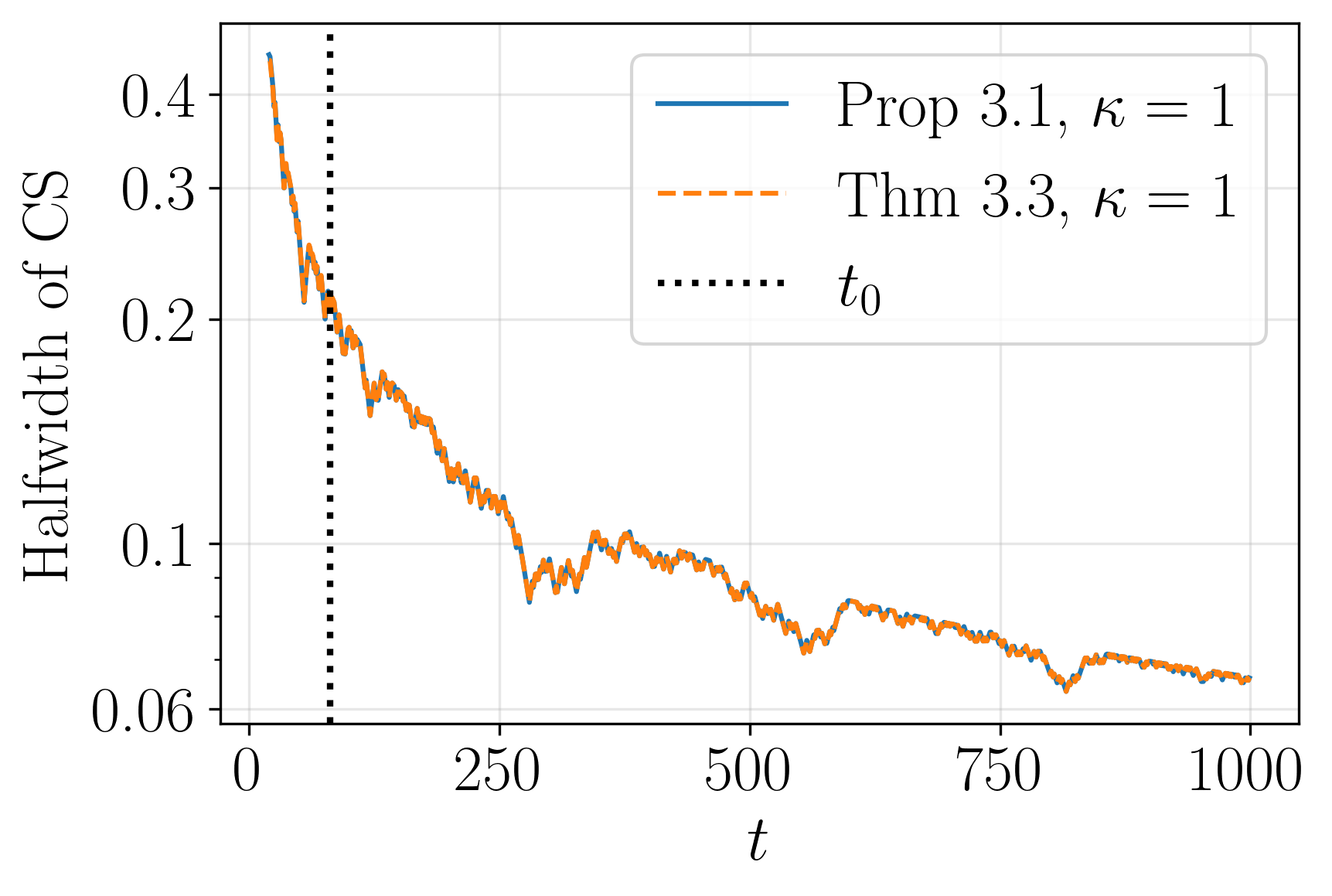}  
        \caption{}
        \label{fig:exact-vs-approx}
    \end{subfigure}
    \begin{subfigure}[t]{0.45\linewidth}
        \includegraphics[height=5cm]{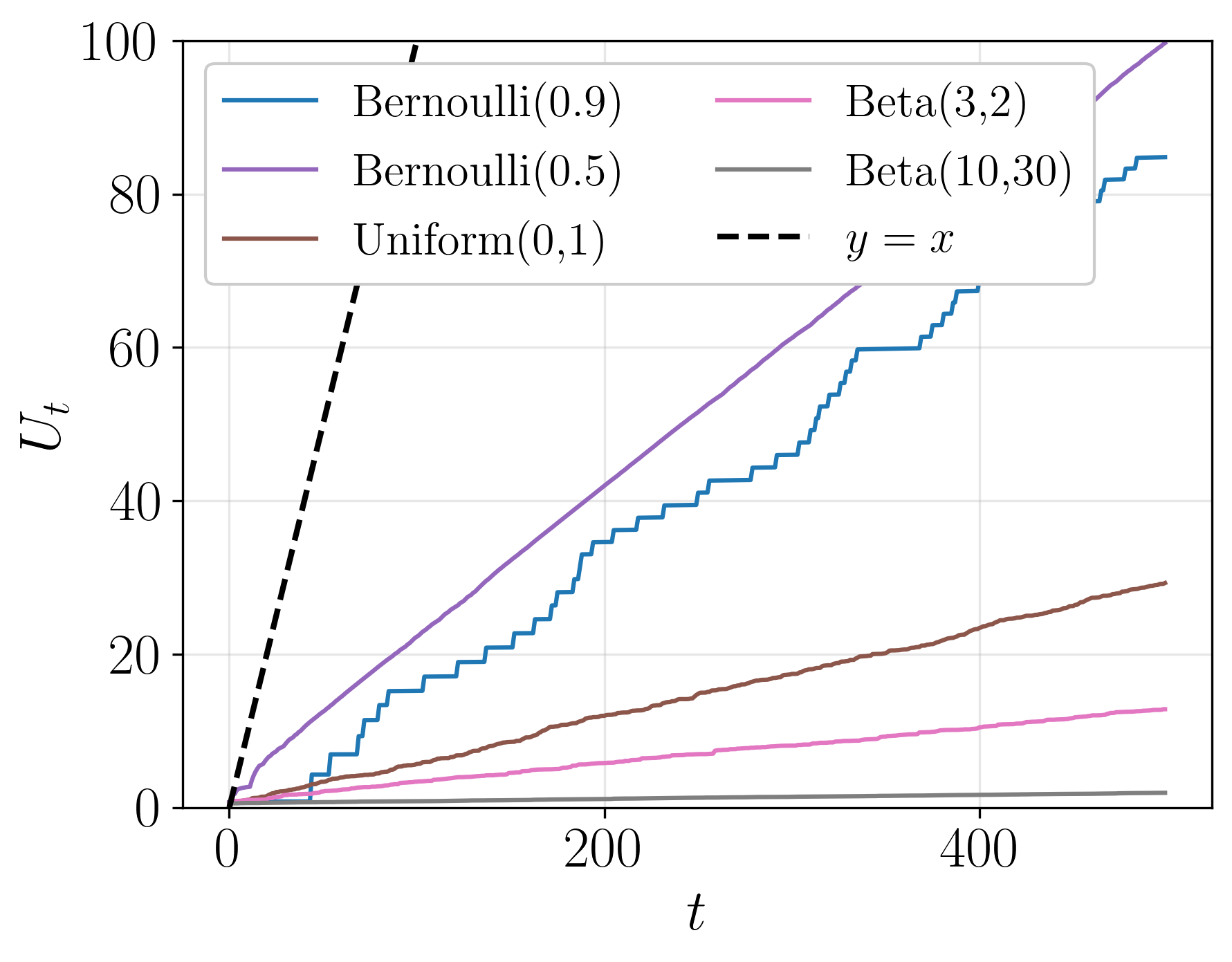}
        \caption{}
        \label{fig:Ut-growth}
    \end{subfigure}    
	\caption{(a) Comparison of the exact CS $C_t^\mix$ defined in Proposition~\ref{prop:mixture-cs-tn} and the relaxation $C_t^\apx$ defined in Theorem~\ref{thm:closed-form}. The differences between the two bounds are negligible. We also plot $t_0$, the time at which $C_t^\apx$ becomes valid (vertical dotted line). Here the observations were drawn independently as $X_t \sim \text{Ber(0.5)}$. (b) The growth of $U_t$ under various distributions. We plot the line $y=x$ for reference, demonstrating that $U_t$ scales as $ct$ for some $c<1$ in each case. 
    } 
\end{figure}

\begin{theorem}
\label{thm:closed-form}
Let $(X_t)_{t\geq 1}$ be a sequence of random variables in $[0,1]$ with average conditional means $(\mu_t)_{t\geq 1}$. Let $(\hatX_t)_{t\geq 1}$ be any predictable sequence in $[0,1)$. Fix $\kappa>0$ and take $Z = \Phi(1/\kappa) - \Phi(-1/\kappa)$. Set $U_t = (2\kappa^2)^{-1} + \sum_{i\leq t} \psi_E(|X_i - \hatX_i|)$ 
and $G_\alpha = \kappa Z \sqrt{2\pi}/\alpha$. 
Let $t_0$ be the minimum time $t$ at which 
    \begin{equation}
    \label{eq:t-condition}
        \sqrt{\frac{\pi}{U_{t}}}\left( \exp\bigg(\frac{U_{t}}{4}\bigg)  -  \frac{1}{2}\right) \geq G_\alpha. 
    \end{equation}
    Then the sequence $(C_t^\apx)_{t=t_0}^\infty$ is a $(1-\alpha)$-CS for $\mu_t$, where $C_t^\apx = (\Xbar_t \pm W_t^\apx)$, 
    \begin{equation}
    \label{eq:cs-mom-closed-form}
    W_t^\apx := \frac{2}{t}\sqrt{U_t\left( \ell_\alpha  + \frac{1}{2}\log(2U_t)\right)} \text{~ and ~} \ell_\alpha := \log\left(\frac{\kappa Z / \alpha}{1 - \exp(-U_t/4)}\right).
    \end{equation}
    Moreover, if $\mu_t\equiv\mu$ is constant over time, then the running intersection $(\cap_{j=t_0}^t C_j^\apx)_{t\geq t_0}$ is also a $(1-\alpha)$-CS for $\mu$. 
\end{theorem}

Figure~\ref{fig:comparison} compares the width of  $C_t^\apx$ to $C_t^\prpl$ and $C_t^\hrms$ under a constant conditional mean.  Across all three distributions, the width of $C_t^\apx$ is mostly less than that of $C_t^\prpl$ at larger sample sizes. At smaller sample sizes (roughly $<1000$), its width is reliably larger. 
The width of $C_t^\hrms$ is larger than both $C_t^\apx$ and $C_t^\prpl$ across all three experiments. The gains of its iterated logarithm rates become apparent at much larger sample sizes ($>10^6$). 
The third column plots the sample size on a log scale to emphasize earlier timesteps and the good performance of $C_t^\prpl$. The second column plots the width on a log scale to emphasize the difference between the bounds at larger sample sizes.

Figure~\ref{fig:comparison} fixes $\kappa=0.25$, which we find to be a good default choice. Appendix~\ref{app:experiments} investigates the effect of $\kappa$ empirically, showing that as $\kappa$ grows, the widths of converge. This is in line with Remark~\ref{rem:alternative-dists}, which notes that the truncated Gaussian approaches a uniform distribution on $[-1,1]$ as $\kappa\to\infty$.

Figure~\ref{fig:time-varying} considers the case of a time-varying mean average conditional mean $\mu_t$. The left plot considers an abrupt change, drawing Bernoulli observations with mean 0.8 for the first $\approx 1000$ time steps, and then Bernoulli observations with mean 0.2 thereafter. (Recall that $\mu_t = \sum_{j\leq t} \E_{j-1}[X_j]$ is the average conditional mean until time $t$, not simply the mean at time $t$). The right plot considers a more continuous change of conditional mean. In both cases, we see that while both $C_t^\apx$ and $C_t^\hrms$ track $\mu_t$, $C_t^\prpl$ does not. 
The envelope of $C_t^\apx$ and $C_t^\hrms$ are remarkably close, especially at larger sample sizes. For the early timesteps we see that $C_t^\apx$ remains tighter than $C_t^\hrms$.

Figure~\ref{fig:exact-vs-approx} illustrates the width of the exact CS $C_t^\mix$ versus the relaxation $C_t^\apx$. The relaxation is remarkably tight, even after plotting the widths on a log-scale. Indeed, the differences are invisible to the naked eye. Unless differences on the order of thousandths are important for a particular application, we recommend using $C_t^\apx$ in practice when $t\geq t_0$ as it is both easier and faster to compute than $C_t^\mix$.

\subsection{Asymptotic width}
\label{sec:asymp-width}

One should think of $U_t$ as growing linearly (see Figure~\ref{fig:Ut-growth}) in which case the width of $C_t^\apx$ (and hence $C_t^\mix$) scales as $O(\sqrt{\log(t)/t})$. To get a more precise sense of how the width of $C_t^\apx$ (and hence $C_t^\mix)$ behaves,
we can compare it directly to $\sqrt{\log(t)/t}$ as follows. Throughout this section we assume a constant conditional mean $\mu = \E_{t-1}[X_t]$ for all $t$ and we assume that the observations are iid.  

\begin{lemma}
    \label{lem:precise-width}
    Let $(X_t)$ be drawn iid from $P$ with mean $\mu$. Let $W_t^\apx$ be the halfwidth of $C_t^\apx$ with $\hatX_t = t^{-1}(1/2 + \sum_{i<t}X_i)$. Then 
    \begin{equation}
    \label{eq:apx-asymp-width}
    W_t^\apx \sim A_t^\apx := 
     \sqrt{\frac{2 \E_P[\psi_E(|X - \mu|)]\log(t)}{t}}.%\quad P\text{-a.s.}  
    \end{equation}
    In particular, letting $R_t := W_t^\apx / A_t^\apx$, we have 
    \begin{equation}
    \label{eq:apx-asymp-ratio}
        R_t^2 = \frac{2U_t}{t\cdot \E_P[\psi_E(|X - \mu|)]}\left(\frac{\log(1/\alpha)}{\log(t)} +   \frac{1}{2} + \oas\left(1\right)\right). 
    \end{equation}
\end{lemma}
The proof is in Appendix~\ref{proof:precise-width}. To get a sense of the size of $A_t^\apx$, one can show that $\sigma^2/2 \leq \E_P[\psi_E(|X-\mu|)] \leq \log(2) - 1/2\approx 0.194$, with Bernoulli(0.5) random variables reaching the upper bound. See Table~\ref{tab:EpsiE} for several examples. In particular, this shows that $\E_P[\psi_E(|X-\mu|)] \leq 1/2$ so $A_t^\apx \leq \sqrt{\log(t)/t}$. See Appendix~\ref{app:psiE-vs-sigma} for a discussion of the worst case ratio between $\E\psi_E(|X-\mu|)$ and $\sigma^2$ and, in particular, a proof that $\E\psi_E(|X-\mu|)\leq C_\mu \sigma^2$ for a constant $C_\mu$ depending on $\mu$.

An intriguing fact about Lemma~\ref{lem:precise-width} is that the asymptotic width of $C_t^\apx$ does not depend on $\alpha$. This is because $\log(1/\alpha)$ is attached to a lower order term in $W_t^\apx$ (it scales with $U_t$ as opposed to $U_t\log(U_t)$) and its influence thus gets washed out asymptotically. This fact is perhaps unsurprising given that we are using the method of mixtures. Indeed, Robbins' original mixture CS for sub-Gaussian observations exhibits the same behavior, which we demonstrate in Appendix~\ref{sec:robbins-mixture}. 

However, we emphasize that in practice the effect of $\alpha$ is not negligible. As shown by~\eqref{eq:apx-asymp-ratio}, the influence of $\log(1/\alpha)$ vanishes with rate $\log(t)$. Thus, even for sample sizes of $10^6$, we have $\log(1/\alpha) / \log(t) \approx \log(1/\alpha)/13.8$. For $\alpha = 0.01$, $\log(1/\alpha) \approx 4.6$, so this term will add $\approx 1/3$ to $R_t^2$ and hence $\approx 1/9$ to the halfwidth $W_t^\apx$ as compared to $A_t^\apx$.

In Appendix~\ref{app:wsr-width} we study the asymptotics of $(C_t^\prpl)$ and show that, if $W_t^\prpl$ is the halfwidth of $C_t^\prpl$ as given in~\eqref{eq:wsr-prpl} and $(X_t)$ are iid with mean $\mu$ drawn from $P$, then 
\begin{equation}
\label{eq:wsr-asymp-width}
        W_t^\prpl \sim A_t^\prpl := \frac{\sigma}{2} \sqrt{\frac{\log(2/\alpha) \log(t)}{2t}} (1 + \log\log(t)). %\quad P\text{-a.s.} 
\end{equation}

Note that unlike $A_t^\apx$, $A_t^\prpl$ does depend on $\alpha$. Further, while $A_t^\apx$ scales exactly as $\sqrt{\log(t)/t}$, $A_t^\prpl$ has an extra factor of $\log\log t$. Therefore, 
\begin{equation*}
    \lim_{t\to\infty} \sqrt{\frac{t}{\log t}} W_t^\apx = \Oas(1) \text{~ whereas ~} \lim_{t\to\infty} \sqrt{\frac{t}{\log t}} W_t^\prpl = \infty\quad P\text{-a.s.}
\end{equation*}
In words, $C_t^\apx$ will decay at a rate proportional to $\sqrt{\log(t)/t}$, whereas $C_t^\prpl$ decays at a strictly asymptotically larger rate. If we correct for the extra $\log\log(t)$ factor and examine the asymptotics, we obtain: 
\begin{equation}
    \sqrt{\frac{t}{\log t}} \frac{W_t^\prpl}{\log\log t} \xrightarrow[t\to\infty]{a.s.} \sqrt{\frac{\sigma^2\log(2/\alpha)}{8}}. %\quad P\text{-almost surely.}
\end{equation}
Unlike $W_t^\apx$, this width depends on $\alpha$. As we demonstrate in Appendix~\ref{app:wsr-width}, this can be avoided if one chooses $\lambda_t$ independently of $\log(1/\alpha)$, in which case one obtains the limiting width $\sqrt{\frac{t}{\log t}} \frac{W_t^\prpl}{\log\log t} \xrightarrow{t\to\infty} \sqrt{\sigma^2/8}$ instead. But in this case the factor $\log(1/\alpha)$ disappears at rate $1/\log\log(t)$, which is extremely slow. For practical purposes therefore, the choice of $\lambda_t$ made by \citet{waudby2024estimating} performs much better. 

%The displays~\eqref{eq:apx-asymp-width} and~\eqref{eq:wsr-asymp-width} can also help us understand why $C_t^\apx$ is tighter for larger sample sizes. For $t\geq 2000$, we have $\log\log(t)\approx 2$ or larger.  And for $\alpha = 0.01$, $\log(2/\alpha) \approx 5.3$. Therefore $W_t^\prpl \approx \sigma \sqrt{\log(t)/t} \cdot \sqrt{5.3/8}(1 + 2) \approx 2.4 \sigma \sqrt{\log(t)/t}$.  Meanwhile, $\sigma^2/2 \leq \E[\psi_E(|X- \mu|)]\leq 1/2 - \sigma^2$, where the upper bound reaches equality for Bernoullis. Thus, for Bernoulli(0.5) for instance, we have $\sigma^2=1/4$ so $W_t^\apx \approx 0.71 \sqrt{\log(t)/t}$ and $W_t^\prpl \approx 1.2\sqrt{\log(t)/t}$. 

While we have compared the width of $C_t^\apx$ to $\sqrt{\log(t)/t}$, the 
the law of the iterated logarithm dictates that the optimal rate at which a CS can shrink towards 0 width is $\Theta(\sqrt{\log\log(t)/t})$~\citep{darling1967confidence,stout1970hartman}. 
To obtain such rates, we can turn to a method known as ``stitching''~\citep{howard2021time}. 
We note, however, that this is mostly of theoretical interest. Iterated logarithm bounds, while being asymptotically tighter than the $\sqrt{\log t}$ bounds, tend to be looser at practically achievable sample sizes, only displaying their tightness at much larger sample sizes.  

\begin{table}[t]
    \centering
    \begin{tabular}{c|ccccc}
    & Ber(0.5) & Ber(0.1) & Unif(0,1) & Beta(5,2) & Beta(10,30) \\
    \hline 
         $\sigma^2/2$& 0.125 & 0.045 &  $\approx 0.042$  & $\approx 0.013$ & $\approx 0.002$  \\
          $\E[\psi_E(|X - \mu|)] $ & $\approx 0.194$ & $\approx 0.144$ & $\approx 0.057$ & $\approx 0.016$ & $\approx 0.002$
    \end{tabular}
    \caption{Comparison of $\E[\psi_E(|X-\mu|)$ with its lower bound of $\sigma^2/2$ for several distributions. As the variance decreases, $\E[\psi_E(|X-\mu|)]$ approaches $\sigma^2/2$. }
    \label{tab:EpsiE}
\end{table}

\section{Iterated Logarithm Rates}
\label{sec:stitching}

As usual, let $(X_t)_{t\geq 1}$ lie in $[0,1]$ and let $(\Xhat_t)\subseteq[0,1)$ be predictable. 
For a fixed $\xi\in[-1,1]$ recall that Section~\ref{sec:new-nsm} showed that $B_t(\mu_t,\xi) = \prod_{i\leq t}\exp\{ \xi(X_t - \mu_t) - \xi^2 \psi_E(|X_t-\Xhat_t|)\}$ is a nonnegative supermartingale. Applying Ville's inequality to both $(B_t(\mu,\xi))$ and $(B_t(\mu;-\xi))$ and taking a union bound we obtain that with probability $1-\alpha$, for all $t\geq 1$,  
\begin{equation}
\label{eq:fixed-xi-bound}
	|\Xbar_t - \mu| \leq \frac{|\xi| V_t}{t} + \frac{\log(2/\alpha)}{t|\xi|}, 
\end{equation}
where, as usual, $V_t = \sum_{i\leq t} \psi_E(|X_i - \hatX_i|)$. To obtain a confidence sequence which scales as $O(\sqrt{\log\log(t)/t})$ instead of $O(\sqrt{\log(t)/t})$, we apply~\eqref{eq:fixed-xi-bound} in geometrically spaced epochs $[\eta^j,\eta^{j+1})$, $j=0,1,\dots$, choosing a different parameter $\xi$ in each epoch. This technique, known as stitching, was popularized by \citet{howard2021time} and has been frequently used in recent works to achieve optimal rates (e.g., \citep{whitehouse2025time,waudby2024estimating,chugg2023unified,chugg2025time,orabona2023tight}).   

Stitching is particularly smooth in our case, owing to the sub-Gaussian style behavior of $\xi$ in the bound.  As we've described, other empirical Bernstein bounds have the free parameters $\xi$ inside $\psi_E$, and it's challenging to stitch over $\psi_E$ directly. Indeed, for this reason \citet{howard2021time} only provide stitching for sub-gamma distributions. One can always relax $\psi_E$ to a sub-gamma tail bound by noticing that $\psi_{E,1}(\xi) \leq \frac{\xi^2}{2(1-\xi)}$  but this will result in a looser CS. 

The following result relies on three parameters. The first is the scalar $\eta>1$ which, as above, controls the geometry of the epochs. The second is a parameter $L_0$ which controls the minimum time at which the bound starts. The third is a ``stitching function'' $h:\Re_{\geq 0}\to \Re_{\geq 0}$, which controls how much mass is placed on each epoch in the union bound. This function must be increasing, strictly positive, and obey $\sum_{j\geq 0} h^{-1}(j)\leq 1$. 
A reasonable choice, and one often used, is $h(j) = \zeta(s)(j+1)^s$ where $\zeta$ is the Riemann-zeta function and $s>1$. This is the choice that was made when presenting $C_t^\hrms$  in~\eqref{eq:hrms}. 

\begin{theorem}
	\label{thm:stitching} 
	Let $h:\Re_{\geq 0}\to\Re_{\geq 0}$ be increasing and satisfy $\sum_{j=0}^\infty 1/h(j) \leq 1$. Fix $\eta>1$. 
	Let $(\hatX_t)$ be a predictable sequence in $[0,1)$ and set $\alpha\in(0,1)$ and $L_0\geq 1$. 
	Then, with probability $1-\alpha$, simultaneously for all $t$ such that $V_t \geq L_0$ and $h(\log_\eta(V_t/L_0))\leq \frac{\alpha}{2}\exp(V_t)$, 
    \begin{equation}
        |\Xbar_t - \mu_t| \leq \frac{\sqrt{\eta}+1}{t} \sqrt{V_t \left(\log\left(\frac{2}{\alpha}\right) + \log h\left(\log_\eta\left(\frac{V_t}{L_0}\right)\right)\right)}.
    \end{equation}
\end{theorem}

The proof is in Appendix~\ref{proof:stitching}. Let us note that the constraint on $t$ is not very demanding. Indeed, $h(j)$ is typically a polynomial in $j$, hence $h(\log_\eta(V_t/L_0))\leq \frac{\alpha}{2}\exp(V_t)$ is achieved quickly. 

To analyze the width of this bound, let us take $L_0=1$, and $h(j) = \zeta(s)(j+1)^s$. Then $\log(h(\log_\eta(V_t)) = \log(\zeta(s)) + s\log(\log_\eta(V_t) + 1) = \log(\zeta(s) /\log^r(\eta)) + \log\log(\eta V_t))$ so  we have that with probability $1-\alpha$, simultaneously for all $t$ large enough, $|\Xbar_t - \mu_t|<W_t^\stitch$ where 
\begin{equation}
    W_t^\stitch := \frac{\sqrt{\eta}+1}{t}\sqrt{V_t \left(\log\left(\frac{2\zeta(s)}{\alpha\log^s(\eta)}\right) + s\log\log(\eta V_t)\right)}.
\end{equation}
To see explicitly that this bound has iterated logarithm rates, let us compare it against $\sqrt{\log\log(t)/t}$. 
Suppose that we are in the fixed mean setting, so $\mu_t = \mu$. 
First note that $V_t \sim tu$ $P$-almost surely (by calculations in the previous section), where $u = \E_P[\psi_E(|X- \mu|)]$. Therefore,  $\log\log(\eta V_t) \sim \log\log(\eta tu) = \log(\log(t) (1 + \log(\eta u)/\log(t))) - \log\log(2) \sim \log\log(t)$, and 
\begin{align*}
    \sqrt{\frac{t}{\log\log t}} W_t^\stitch &= (\sqrt{\eta}+1) \sqrt{\frac{V_t \log(2\zeta(s) / \alpha\log^s(\eta)) + rV_t \log\log(\eta V_t)}{t\log\log(t)}} \\ 
    &\sim (\sqrt{\eta}+1) \sqrt{\frac{s tu \log\log(t)}{t\log\log (t)}} 
    = (\sqrt{\eta s}+\sqrt{s})\sqrt{u}. 
\end{align*}
Letting $\eta$ and $s$ tend towards 1, we see that $W_t^\stitch \sim 2\sqrt{u \log\log(t)/t}$. Of course, in practice, small values of $\eta$ and $s$ inflate the term $2\zeta(s)/(\alpha\log^s(\eta))$ which has a noticeable effect on the bound. We suggest using $s=1.4$ and $\eta=2$ in experiments. These are also the parameters suggested by \citet{howard2021time}. 

\begin{figure}[t]
	\centering
    \begin{subfigure}[t]{0.45\linewidth}
        \includegraphics[height=5cm]{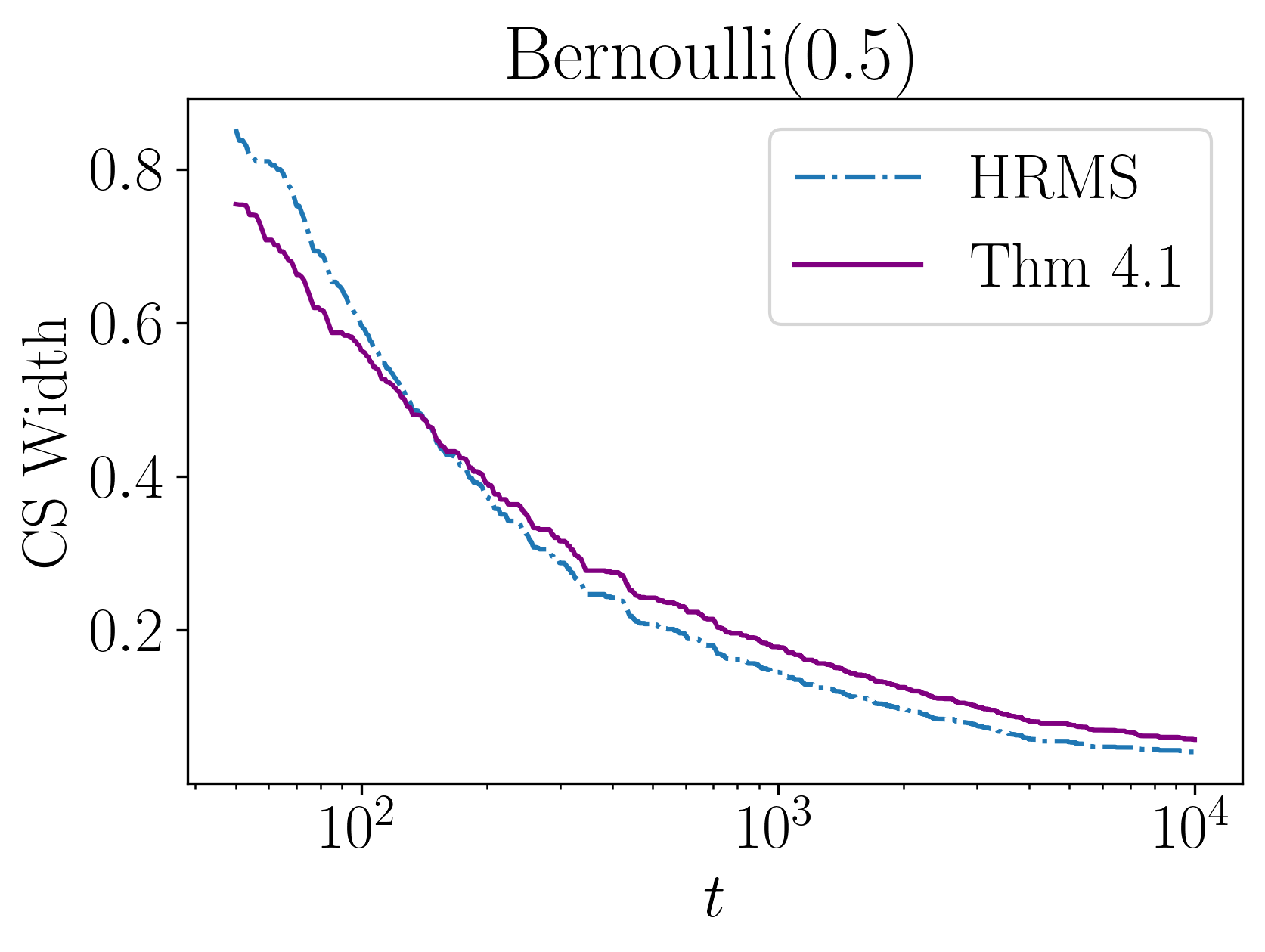}  
    \end{subfigure}
    \begin{subfigure}[t]{0.45\linewidth}
        \includegraphics[height=5cm]{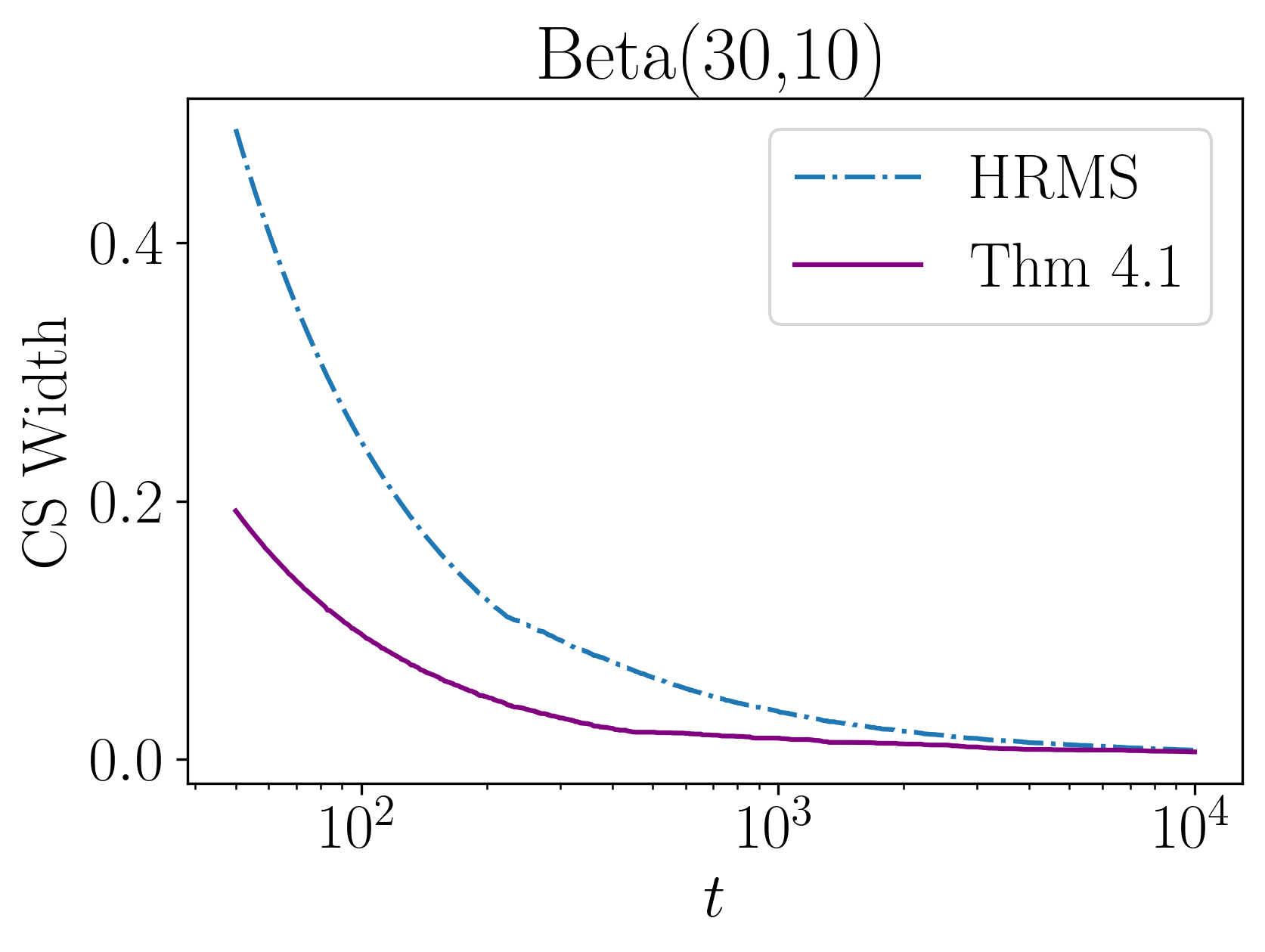}
    \end{subfigure}    
	\caption{
    Comparison of Theorem~\ref{thm:stitching} and $C_t^\hrms$, the two CSs which have iterated logarithm rates. Theorem~\ref{thm:stitching} is implemented with $h(j) = \zeta(s)(j+1)^s$ and we use $\eta=2$, $s=1.4$ for both bounds. While $C_t^\stitch$ is tighter than $C_t^\hrms$ at smaller sample sizes, $C_t^\hrms$ is tighter in the limit as discussed in the text. 
    } 
    \label{fig:stitching}
\end{figure}

Like $W_t^\apx$, the width of our stitched CS scales independently of $\alpha$ asymptotically. Meanwhile, 
the empirical Bernstein CS of \citet{howard2021time} satisfies $W_t^\hrms \sim \sqrt{2\sigma^2 \log\log(t)/t}$. Recalling that $u\geq \sigma^2/2$, this has smaller constants than our asymptotic width, though the two expressions approach each other as the variance of the distribution decreases. 

Figure~\ref{fig:stitching} plots $C_t^\stitch$ against $C_t^\hrms$. As expected from the preceding paragraph, $C_t^\hrms$ is tighter at large sample sizes. At small to moderate sample sizes, however, $C_t^\stitch$ is often tighter, sometimes significantly so. To explain this phenomenon, it helps to recall that our bound looks sub-Gaussian in terms of the free parameter $\xi$. To elaborate, if $(X_t)$ is $\sigma$-sub-Gaussian, then with probability $1-\alpha$, $|\Xbar_t - \mu_t|\leq \xi V'_t/t + \log(2/\alpha)/(t\xi)$ where $V'_t = t\sigma^2$ for all $\xi>0$. This is comparable to~\eqref{eq:fixed-xi-bound} modulo a missing factor of 2 on the first term and the fact that $\xi$ can range over all positive reals. (This final fact makes no difference as soon as the sample size is large enough, as the optimal $\xi$ for $t$ observations is a decreasing function of $t$.) 

%That $\xi$ is restricted to $(0,1]$ in~\eqref{eq:fixed-xi-bound} does not effect the argument much however, 
%not affect the bound, since the parameters are chosen to be small in the stitching argument.  

Thus, we are effectively stitching over a sub-Gaussian tail, while \citet{howard2021time} are stitching over a sub-gamma tail, which involves an extra linear term. This linear term vanishes at rate $O(\log\log(t)/t)$ from the bound, however, as opposed to $O(\sqrt{\log\log(t)/t})$,  thus only has a noticeable effect at smaller sample sizes. In other words, at larger sample sizes the bound is primarily affected by the size of the constants on the leading term, which are slightly larger in our case. 

We reiterate that while Theorem~\ref{thm:stitching} is theoretically satisfying, it is empirically looser than Theorem \ref{thm:closed-form} at most realistic sample sizes.  In practice we recommend using the latter.

\section{Extension to matrix empirical Bernstein bounds}
\label{sec:matrix}

Here we extend the techniques developed above to give a novel empirical Bernstein CS for the maximum eigenvalue of matrices. Let $\calS_d$ denote the set of $d\times d$ symmetric matrices with real entries. We consider a stream $(\bX_t)$ of $\calS_d$-valued random matrices with eigenvalues in $[0,1]$. Let $\bM_t = t^{-1}\sum_{i\leq t} \E_{i-1}[\bM_i]$ be the average conditional mean at time $t$ and $\bX_t = t^{-1}\sum_{i\leq t} \bX_i$ be the running average of the matrices at time $t$. We denote the maximum and eigenvalues of a matrix $\bX\in\calS_d$ by $\gamma_{\max}(\bX)$ and $\gamma_{\min}(\bX)$ respectively. 

Following \citet{tropp2012user,tropp2015introduction}, we are interested in providing bounds on $\gamma_{\max}(\bXbar_t - \bM_t)$. Unfortunately, naively following the strategy developed in Section~\ref{sec:new-nsm} runs into trouble almost immediately. In particular, while a matrix version of the modified Fan's inequality in Lemma~\ref{lem:extended-Fan} exists, a modified Howard's inequality does not (nor does a matrix version of the original Howard's inequality). Inspecting the proof of Lemma~\ref{lem:modified-howard} reveals that it uses the commutativity of real numbers; a property that cannot, of course, be lifted to the matrix setting. 

Instead, we rely on a result first proved by \citet{howard2020time} which gives conditions under which three sequences of matrices forms a supermartingale. We leverage this result and some elementary properties of eigenvalues to prove the following lemma. The proof is in Section~\ref{sec:proof-matrix-nsm}. 

\begin{lemma}
\label{lem:matrix-nsm}
   Let $(\bX_t)$ be a stream of random matrices in $\calS_d$ with eigenvalues in $[0,1]$. Let $(\b{\Xhat}_t)$ be a predictable sequence of matrices in $\calS_d$ with eigenvalues in $[0,1)$. Let $(\xi_t)$ be a predictable sequence in $[-1,1]$. Then $(L_t)_{t\geq 0}$ is a nonnegative supermartingale where  
   \begin{equation}
       L_t := \Tr\exp\left\{\sum_{i=1}^t \xi_i(\bX_i - \E_{i-1}[\bX_i]) - \sum_{i=1}^t \xi_i^2\psi_E(|\bX_i - \b{\Xhat}_i|)\right\},
   \end{equation}
   and $L_0 = d$. 
   Moreover, for any density $\rho$ over $[-1,1]$, $(N_t)_{t\geq 0}$ is upper bounded by a supermartingale with initial value $d$, where 
   \begin{equation}
       N_t := \int \exp\left\{\xi \gamma_{\max}(\bS_t - t\bM_t) - \xi^2 \gamma_{\max}(\bV_t) \right\}\rho(\xi)\d\xi,
   \end{equation}
   where 
   \begin{equation}
       \bS_t := \sum_{i\leq t}\bX_i, \quad \bM_t := \frac{1}{t}\sum_{i\leq t}\E_{i-1}[\bX_i], \quad \bV_t = \sum_{i\leq t} \psi_E(|\bX_i - \bhatX_i|).
   \end{equation}
\end{lemma}

As in Section~\ref{sec:mom} one can take $\rho$ to be a truncated Gaussian and write down the equivalent of Proposition~\ref{prop:mixture-cs-tn} (or of Remark~\ref{rem:alternative-dists}). Such a bound is useless in the matrix case however, as an implicit bound requires searching over the space of symmetric matrices. Instead, let us turn immediately to the closed-form approximation. The proof of Theorem~\ref{thm:eb-matrix} is in Appendix~\ref{proof:eb-matrix}. 

\begin{theorem}
\label{thm:eb-matrix}
Let $(\bX_t)$ be a sequence of random matrices in $\calS_d$ with eigenvalues in $[0,1]$ with average conditional mean $\bM_t$. Let $(\bhatX_t)$ be any predictable sequence of matrices in $\calS_d$ with eigenvalues in $[0,1)$. Fix $\kappa>0$ and take $Z = \Phi(1/\kappa) - \Phi(-1/\kappa)$. Set $U_t = (2\kappa^2)^{-1} + \gamma_{\max}(\bV_t)$.  
and $G_\alpha = \kappa Z \sqrt{2\pi}/\alpha$. 
Let $t_0$ be the minimum time $t$ at which 
    \begin{equation}
        \sqrt{\frac{\pi}{U_{t}}}\left( \exp\bigg(\frac{bU_{t}}{4}\bigg)  -  \frac{1}{2}\right) \geq G_\alpha. 
    \end{equation}
    Then, with probability $1-\alpha$,  simultaneously for all $t\geq t_0$, 
    \begin{equation}
    \left|\gamma_{\max}(\bXbar_t - \bM_t)\right| \leq W_t^\mat, \quad W_t^\mat := 
    \frac{2}{t}\sqrt{U_t\left( \ell_{\alpha,d}  + \frac{1}{2}\log(2U_t)\right)},
    \end{equation}
    where 
    \begin{equation}
        \ell_{\alpha,d} := \log\left(\frac{d\kappa Z / \alpha}{1 - \exp(-U_t/4)}\right).
    \end{equation}
\end{theorem}

Like we did in the scalar case, we can analyze the asymptotic width of this bound. 
Let us suppose that the matrices are iid with mean $\bM$, so that $\bM_t  = \bM$ for all $t$. 

\begin{proposition}
\label{prop:mat-asymp-width}
    Let $(\bX_t)$ be a sequence of random matrices in $\calS_d$ drawn iid from $P$ with eigenvalues in $[0,1]$ and mean $\bM$. Let $\bhatX_t = t^{-1} (\b{I}_d/2 + \sum_{j<t} \bX_j)$. Then 
     \begin{equation}
    \label{eq:mat-asymp-width}
    W_t^\mat \sim A_t^\mat := 
     \sqrt{\frac{2 \gamma_{\max}(\E_P[\psi_E(|\bX - \bM|)])\log(t)}{t}}.%\quad P\text{-a.s.}  
    \end{equation}
\end{proposition}

The proof of Proposition~\ref{prop:mat-asymp-width} follows precisely the same logic as Lemma~\ref{lem:precise-width} after showing that $\gamma_{\max}(\bV_t)/t \to \gamma_{\max} (\E_P[\psi_E(\bX - \bM)])$. It is presented in Appendix~\ref{proof:mat-asymp-width}.

Let us compare Theorem~\ref{thm:eb-matrix} to the recent empirical Bernstein CS of \citet[Theorem 4.2]{wang2024sharp}, which states 
that if  $(\lambda_t)$ be a predictable sequence in $(0,1)$. Then, with probability $1-\alpha$, for all $t\geq 1$, 
    \begin{equation}
    \label{eq:eb-matrix-wr}
    \left|\gamma_{\max}\left(\frac{\sum_{i\leq t}\lambda_i(\bX_i - \bM_i)}{\sum_{i\leq t}\lambda_i}\right)\right| \leq \frac{\log(2d/\alpha) +\gamma_{\max}(\sum_{i\leq t} \psi_E(\lambda_i)(\bX_i - \bhatX_i)^2)}{\sum_{i\leq t}\lambda_i}.
    \end{equation}

We note that Wang and Ramdas state their result as a one-sided bound; we apply a union bound to make it two-sided and more easily comparable to Theorem~\ref{thm:eb-matrix}. The bound in~\eqref{eq:eb-matrix-wr} should be considered as the matrix equivalent of $C_t^\prpl$, insofar as it relies on similar techniques and uses the method of predictable plug-ins. A similar discussion can thus be had on the pros and cons of Theorem~\ref{thm:eb-matrix} versus the bound in~\eqref{eq:eb-matrix-wr}  as was had in the scalar setting. 

If, similarly to the scalar case, one chooses $\xi_t \asymp 1/ \sqrt{t\log(t)}$ and $\bhatX_t = \frac{1}{t}(\b{I}_d/2 + \sum_{j<t}\bX_j)$ then the halfwidth of~\eqref{eq:eb-matrix-wr} behaves as $(\log(d/\alpha) + \log\log(t))\cdot \sqrt{\log(t)/t}$ (via very similar arguments as in Section~\ref{sec:asymp-width}). Meanwhile, if one chooses $\xi_t  \asymp \frac{1}{\sqrt{t}}$ as in \citet{wang2024sharp} and $\bhatX_t = \sum_{i<t} \psi_E(\xi_i)\hatX_i/\psi_E(\xi_i)$, then the width behaves as $(\log(d/\alpha) + \log(t))/\sqrt{t}$. In both cases $W_t^\mat$ remains asymptotically tighter.

\section{Summary}

We have a presented a new closed-form CS for bounded random variables, $C_t^\apx$, which is easy to implement and tracks a time-varying mean. The CS requires no knowledge of the distribution other than boundedness, adapting to the samples via the function $\psi_E(\lambda) =-\log(1-\lambda)-\lambda$. Among such empirical closed-form CSs, we find that $C_t^\apx$ is the tightest in practice, making it a good candidate for `out-of-the-box' applications. 

We extended these results to (i) provide a CS whose width decays at the optimal rate of $O(\sqrt{\log\log(t)/t})$ and (ii) to provide a CS with similar properties for matrices with bounded eigenvalues. Our approach is based on modifications of Fan's and Howard's inequalities~\citep{fan2015exponential,howard2021time} which led to a new empirical Bernstein-style supermartingale and may be of independent interest.

When multiplied by $\sqrt{t/\log(t)}$, $C_t^\apx$ has limiting width $\sqrt{2\psi_E(|X-\mu|)}$ (or the maximum eigenvalue of that quantity in the matrix case). Since $\sqrt{2\E[\psi_E(|X-\mu|)]}\geq \sigma$, an intriguing open question is to design a CS with limiting with $\sigma$, which would match that of Robbins' mixture CS for sub-Gaussian observations, and the oracle Bernstein of  \citet{howard2021time} (see Table~\ref{tab:limiting-widths}). 
We expect that such a CS exists. 

\subsection*{Acknowledgments}
Both authors acknowledge support from NSF grants IIS-2229881 and DMS-2310718.

{\small 
\bibliographystyle{plainnat}
\bibliography{main}

@article{howard2021time,
  title={Time-uniform, nonparametric, nonasymptotic confidence sequences},
  author={Howard, Steven R and Ramdas, Aaditya and McAuliffe, Jon and Sekhon, Jasjeet},
  journal={The Annals of Statistics},
  volume={49},
  number={2},
  pages={1055--1080},
  year={2021},
  publisher={JSTOR}
}

@article{waudby2024time,
  title={Time-uniform central limit theory and asymptotic confidence sequences},
  author={Waudby-Smith, Ian and Arbour, David and Sinha, Ritwik and Kennedy, Edward H and Ramdas, Aaditya},
  journal={The Annals of Statistics},
  volume={52},
  number={6},
  pages={2613--2640},
  year={2024},
  publisher={Institute of Mathematical Statistics}
}

@article{robbins1970statistical,
  title={Statistical methods related to the law of the iterated logarithm},
  author={Robbins, Herbert},
  journal={The Annals of Mathematical Statistics},
  volume={41},
  number={5},
  pages={1397--1409},
  year={1970},
  publisher={JSTOR}
}

@article{martinez2025sharp,
  title={Sharp empirical {Bernstein} bounds for the variance of bounded random variables},
  author={Martinez-Taboada, Diego and Ramdas, Aaditya},
  journal={arXiv preprint arXiv:2505.01987},
  year={2025}
}

@article{ramdas2025hypothesis,
  title={Hypothesis testing with e-values},
  author={Ramdas, Aaditya and Wang, Ruodu},
  journal={Foundations and Trends in Statistics},
  year={2025}
}

@Misc{howardconfseq,
  author = {Steven R. Howard and Ian Waudby-Smith and Aaditya Ramdas},
  title = {{ConfSeq}: software for confidence sequences and uniform boundaries},
  year = {2021},
  url = "https://github.com/gostevehoward/confseq",
  note = {[Online; accessed December 5, 2025]}
}

@article{tropp2015introduction, title={An introduction to matrix concentration inequalities}, author={Tropp, Joel A}, journal={Foundations and Trends{\textregistered} in Machine Learning}, volume={8}, number={1-2}, pages={1--230}, year={2015}, publisher={Now Publishers, Inc.} }

@article{tropp2012user,
  title={User-friendly tail bounds for sums of random matrices},
  author={Tropp, Joel A},
  journal={Foundations of computational mathematics},
  volume={12},
  number={4},
  pages={389--434},
  year={2012},
  publisher={Springer}
}

@inproceedings{balsubramani2016sequential,
  title={Sequential nonparametric testing with the law of the iterated logarithm},
  author={Balsubramani, Akshay and Ramdas, Aaditya},
  booktitle={Proceedings of the Thirty-Second Conference on Uncertainty in Artificial Intelligence},
  pages={42--51},
  year={2016}
}

@article{chugg2025variational,
  title={A variational approach to dimension-free self-normalized concentration},
  author={Chugg, Ben and Ramdas, Aaditya},
  journal={arXiv preprint arXiv:2508.06483},
  year={2025}
}

@article{peel2010empirical,
  title={Empirical {Bernstein} inequalities for u-statistics},
  author={Peel, Thomas and Anthoine, Sandrine and Ralaivola, Liva},
  journal={Advances in Neural Information Processing Systems},
  volume={23},
  year={2010}
}

@article{martinez2025vector,
  title={Vector-valued self-normalized concentration inequalities beyond {sub-Gaussianity}},
  author={Martinez-Taboada, Diego and Gonzalez, Tomas and Ramdas, Aaditya},
  journal={Algorithmic Learning Theory (ALT) 2026},
  year={2025}
}

@article{tolstikhin2013pac,
  title={PAC-Bayes-empirical-{Bernstein} inequality},
  author={Tolstikhin, Ilya O and Seldin, Yevgeny},
  journal={Advances in Neural Information Processing Systems},
  volume={26},
  year={2013}
}

@article{waudby2024anytime,
  title={Anytime-valid off-policy inference for contextual bandits},
  author={Waudby-Smith, Ian and Wu, Lili and Ramdas, Aaditya and Karampatziakis, Nikos and Mineiro, Paul},
  journal={ACM/IMS Journal of Data Science},
  volume={1},
  number={3},
  pages={1--42},
  year={2024},
  publisher={ACM New York, NY}
}

@article{stout1970hartman,
  title={The {Hartman-Wintner} law of the iterated logarithm for martingales},
  author={Stout, William F},
  journal={The Annals of Mathematical Statistics},
  volume={41},
  number={6},
  pages={2158--2160},
  year={1970},
  publisher={JSTOR}
}

@inproceedings{finzi2025compute,
  title={Compute-Optimal {LLMs} Provably Generalize Better with Scale},
  author={Finzi, Marc Anton and Kapoor, Sanyam and Granziol, Diego and Gu, Anming and De Sa, Christopher and Kolter, J Zico and Wilson, Andrew Gordon},
  booktitle={The Thirteenth International Conference on Learning Representations},
year=2025
}

@inproceedings{jang2023tighter,
  title={Tighter {PAC-Bayes} bounds through coin-betting},
  author={Jang, Kyoungseok and Jun, Kwang-Sung and Kuzborskij, Ilja and Orabona, Francesco},
  booktitle={The Thirty Sixth Annual Conference on Learning Theory},
  pages={2240--2264},
  year={2023},
  organization={PMLR}
}

@inproceedings{shivaswamy2010empirical,
  title={Empirical {Bernstein} boosting},
  author={Shivaswamy, Pannagadatta and Jebara, Tony},
  booktitle={Proceedings of the Thirteenth International Conference on Artificial Intelligence and Statistics},
  pages={733--740},
  year={2010},
  organization={JMLR Workshop and Conference Proceedings}
}

@inproceedings{burgess2021approximating,
  title={Approximating the Shapley Value Using Stratified Empirical {Bernstein} Sampling.},
  author={Burgess, Mark Alexander and Chapman, Archie C},
  booktitle={IJCAI},
  pages={73--81},
  year={2021}
}

@inproceedings{mnih2008empirical,
  title={Empirical {Bernstein} stopping},
  author={Mnih, Volodymyr and Szepesv{\'a}ri, Csaba and Audibert, Jean-Yves},
  booktitle={Proceedings of the 25th international conference on Machine learning},
  pages={672--679},
  year={2008}
}

@article{whitehouse2025time,
  title={Time-uniform self-normalized concentration for vector-valued processes},
  author={Whitehouse, Justin and Wu, Zhiwei Steven and Ramdas, Aaditya},
  journal={Annals of Applied Probability},
  year={2025}
}

@article{voravcek2025star,
  title={{STaR-Bets}: Sequential Target-Recalculating Bets for Tighter Confidence Intervals},
  author={Vor{\'a}{\v{c}}ek, V{\'a}clav and Orabona, Francesco},
  journal={arXiv preprint arXiv:2505.22422},
  year={2025}
}

@article{orabona2023tight,
  title={Tight concentrations and confidence sequences from the regret of universal portfolio},
  author={Orabona, Francesco and Jun, Kwang-Sung},
  journal={IEEE Transactions on Information Theory},
  volume={70},
  number={1},
  pages={436--455},
  year={2023},
  publisher={IEEE}
}

@article{chugg2023unified,
  title={A unified recipe for deriving (time-uniform) {PAC-Bayes} bounds},
  author={Chugg, Ben and Wang, Hongjian and Ramdas, Aaditya},
  journal={Journal of Machine Learning Research},
  volume={24},
  number={372},
  pages={1--61},
  year={2023}
}

@book{wald1947sequential,
  title={Sequential analysis.},
  author={Wald, Abraham},
  year={1947},
  publisher={John Wiley \& Sons}
}

@article{bahadur1956nonexistence,
  title={The nonexistence of certain statistical procedures in nonparametric problems},
  author={Bahadur, Raghu R and Savage, Leonard J},
  journal={The Annals of Mathematical Statistics},
  volume={27},
  number={4},
  pages={1115--1122},
  year={1956},
  publisher={JSTOR}
}

@article{paulson1964sequential,
  title={Sequential estimation and closed sequential decision procedures},
  author={Paulson, Edward},
  journal={The Annals of Mathematical Statistics},
  pages={1048--1058},
  year={1964},
  publisher={JSTOR}
}

@article{wang2024sharp,
  title={Sharp Matrix {empirical Bernstein} Inequalities},
  author={Wang, Hongjian and Ramdas, Aaditya},
  journal={Advances in Neural Information Processing Systems},
  year={2026}
}

@article{chugg2025time,
  title={Time-Uniform Confidence Spheres for Means of Random Vectors},
  author={Chugg, Ben and Wang, Hongjian and Ramdas, Aaditya},
  journal={Transactions on Machine Learning Research},
year={2025}
}

@inproceedings{agrawal2021regret,
  title={Regret minimization in heavy-tailed bandits},
  author={Agrawal, Shubhada and Juneja, Sandeep K and Koolen, Wouter M},
  booktitle={Conference on Learning Theory},
  pages={26--62},
  year={2021},
  organization={PMLR}
}

@article{shekhar2023near,
  title={On the near-optimality of betting confidence sets for bounded means},
  author={Shekhar, Shubhanshu and Ramdas, Aaditya},
  journal={arXiv preprint arXiv:2310.01547},
  year={2023}
}

@inproceedings{honda2010asymptotically,
  title={An Asymptotically Optimal Bandit Algorithm for Bounded Support Models},
  author={Honda, Junya and Takemura, Akimichi},
  booktitle={COLT},
  pages={67--79},
  year={2010}
}

@article{cover1974universal, title={Universal gambling schemes and the complexity measures of {Kolmogorov and Chaitin}}, author={Cover, Thomas M}, journal={Technical Report, no. 12}, year={1974}, publisher={Department of Statistics, Stanford University} }

@article{martinez2024empirical,
  title={Empirical {Bernstein} in smooth {Banach} spaces},
  author={Martinez-Taboada, Diego and Ramdas, Aaditya},
  journal={Annals of Applied Probability},
  year={2025}
}

@article{maurer2009empirical,
  title={Empirical {Bernstein bounds} and sample variance penalization},
  author={Maurer, Andreas and Pontil, Massimiliano},
  journal={Conference on Learning Theory},
  year={2009}
}

@article{audibert2009exploration,
  title={Exploration--exploitation tradeoff using variance estimates in multi-armed bandits},
  author={Audibert, Jean-Yves and Munos, R{\'e}mi and Szepesv{\'a}ri, Csaba},
  journal={Theoretical Computer Science},
  volume={410},
  number={19},
  pages={1876--1902},
  year={2009},
  publisher={Elsevier}
}

@article{waudby2024estimating,
  title={Estimating means of bounded random variables by betting},
  author={Waudby-Smith, Ian and Ramdas, Aaditya},
  journal={Journal of the Royal Statistical Society Series B: Statistical Methodology},
  volume={86},
  number={1},
  pages={1--27},
  year={2024},
  publisher={Oxford University Press US}
}

@article{darling1967confidence,
  title={Confidence sequences for mean, variance, and median},
  author={Darling, Donald A and Robbins, Herbert},
  journal={Proceedings of the National Academy of Sciences},
  volume={58},
  number={1},
  pages={66--68},
  year={1967},
  publisher={National Acad Sciences}
}

@article{ville1939etude,
	title={\'{E}tude critique de la notion de collectif},
	author={Ville, Jean},
	journal={Bull. Amer. Math. Soc},
	volume={45},
	number={11},
	pages={824},
	year={1939}
}

@article{howard2020time,
  title={Time-uniform {Chernoff} bounds via nonnegative supermartingales},
  author={Howard, Steven R and Ramdas, Aaditya and McAuliffe, Jon and Sekhon, Jasjeet},
  journal={Probability Surveys},
  volume={17},
  pages={257--317},
  year={2020}
}

@article{fan2015exponential,
  title={Exponential inequalities for martingales with applications},
  author={Fan, Xiequan and Grama, Ion and Liu, Quansheng},
  journal={Electron. J. Probab},
  volume={20},
  number={1},
  pages={1--22},
  year={2015}
}
}

\newpage 
\appendix 

\section{Related Work and Applications}
\label{sec:related-work}
Some of the first empirical Bernstein bounds are due to \citet{maurer2009empirical} and \citet{audibert2009exploration} in 2009, who both work in the fixed-time setting. These techniques were extended to empirical Bernstein inequalities for u-statistics by \citet{peel2010empirical}. 
In 2016, \citet{balsubramani2016sequential} 
provided the first time-uniform results, giving an empirical Bernstein CS that scaled at an iterated logarithm rate. 

The bounds of \citet{maurer2009empirical,audibert2009exploration} and \citet{balsubramani2016sequential}
were improved both theoretically and empirically by the results of \citet{howard2021time} of \citet{waudby2024estimating}, who study both fixed-time bounds and confidence sequences. In addition to the closed-form bounds in these works which we've compared to here, both of these papers also provide implicit bounds which are usually tighter than the closed-form counterparts, especially at smaller sample sizes. Like us, all the bounds of \citet{howard2021time} allow for a time-varying conditional mean, whereas \citet{waudby2024estimating} require a constant conditional mean. 

Even though we do not study the fixed-time setting in this paper, it's worth mentioning that the fixed-time, closed-form confidence interval of \citet{waudby2024estimating} is ``sharp,'' in the sense that its width scales as the oracle Bernstein bound in the limit. That is, if $W_n^\prpl$ is the width of the CI for a fixed $n$, then $\sqrt{n}W_n^\prpl \xrightarrow{n\to\infty} \sqrt{\sigma^2 \log(2/\alpha)}$.

The techniques developed both by \citet{howard2021time} and \citet{waudby2024estimating} have spurred a significant amount of recent activity on empirical Bernstein bounds, both in the time-uniform and fixed-time settings. Empirical Bernstein bounds have now been developed for vectors~\citep{martinez2024empirical,chugg2025time}, matrices~\citep{wang2024sharp}, self-normalized processes~\citep{whitehouse2025time,chugg2025variational,martinez2025vector}
 and on other statistics of the data such as the variance~\citep{martinez2025sharp}. The scalar case has also been studied further, with \citet{orabona2023tight} developing (implicit) CSs via links to online learning (in particular, the regret of universal portfolio), and \citet{voravcek2025star} recently developing tighter fixed-time intervals using insights that go back to Thomas Cover~\citep{cover1974universal}. 

Related to these implicit CSs, there is a parallel bandit literature on divergence-based confidence bounds (e.g., KL-inf), which can be seen as variance-adaptive via local quadratic approximations; see \citep{agrawal2021regret,honda2010asymptotically,shekhar2023near}.

By their very nature, empirical Bernstein bounds are very applicable. This is part of their appeal. Still, it's perhaps worth noting several places in the literature where they have been employed, either as theoretical tools or as practical bounds. For instance, they have been used to construct adaptive sampling algorithms~\citep{mnih2008empirical,burgess2021approximating}, as tools in learning theory~\citep{chugg2023unified,jang2023tighter,tolstikhin2013pac}, for boosting algorithms ~\citep{shivaswamy2010empirical}, to provide generalization bounds in LLMs~\citep{finzi2025compute}, and for inference in contextual bandits~\citep{waudby2024anytime}.

\section{Omitted Proofs} 

\subsection{Proof of Proposition~\ref{prop:mixture-cs-tn}}
\label{proof:mixture-cs-tn}
We take $\rho$ to be a normal distribution with variance $\kappa^2$ truncated to $[-1,1]$ (the allowable range of $\xi$). 
That is, 
\[\rho(\theta;\kappa) = \frac{1}{\kappa\sqrt{2\pi}}\frac{1}{Z}\exp\left(\frac{-\theta^2}{2\kappa^2}\right).\]
where $Z = \Phi(1/\kappa) - \Phi(-1/\kappa)$ is the normalizing constant. Set $V_t =\sum_{i\leq t} \psi_E(|X_i - \hatX_i|)$ recall that $U_t = (2\kappa^2)^{-1} + V_t$. Then 
\begin{align*}
	M_t(m) &= \int_{-1}^1 B_t(m;\xi) \rho(\xi;\kappa)\d\xi \\
    &= \int_{-1}^1 \exp\{\xi (S_t - tm) - \xi^2 V_t\} \rho(\xi;\kappa) \d\xi \\ 
	&= \frac{1}{\kappa Z \sqrt{2\pi}} \int_{-1}^1 \exp\left\{ \xi (S_t - tm) - \xi^2\bigg( V_t + \frac{1}{2\kappa^2}\bigg)\right\}\d\xi  \\ 
	&= \frac{1}{\kappa Z \sqrt{2\pi}} I\left(S_t - tm; U_t\right).
\end{align*}
By Ville's inequality, $P(\exists t\geq 1: M_t(m)\geq 1/\alpha)\leq 1/\alpha$. 
Therefore, 
\begin{equation}
	C_t^* := \{ m: M_t(m) < 1/\alpha\} = \left\{ m: I(S_t - tm; U_t) < \frac{\kappa Z \sqrt{2\pi}}{\alpha}\right\}. 
\end{equation}
is a $(1-\alpha)$-CS for $\mu$. That the intersection $(\cap_{j=1}^t C_j^\mix)_{t\geq 1}$ also constitutes a $(1-\alpha)$-CS follows from the properties of confidence sequences.

\subsection{Proof of Theorem~\ref{thm:closed-form}}
\label{proof:mom-closed-form}
We begin with the following lemma. 
\begin{lemma}
\label{lem:yt<Ut}
    Fix $G>0$ and let $(U_t)$ be a nonnegative sequence such that $0\leq U_t\nearrow \infty$. Then, for all $t$ large enough, a nonnegative solution $y_t\geq 0$ to $I(y_t;U_t) = G$ exists and satisfies $y_t < U_t$. In fact, $y_t<U_t$ for all $t$ such that 
    \begin{equation}
        \sqrt{\frac{\pi}{U_t}}\left( \exp\bigg(\frac{U_t}{4}\bigg)  -  \frac{1}{2}\right) \geq G. 
    \end{equation}
\end{lemma}
\begin{proof}
    That $y_t$ exists is guaranteed by the fact that $I(0;U_t) \to 0$ as $U_t\to \infty$ and that $y\mapsto I(y;u)$ is surjective on $[I(0;u),\infty)$. Next, fix $t$ and consider $y_0 = U_t$. Then $\sqrt{U_t} - y_0 / (2\sqrt{U_t}) = \frac{1}{2}\sqrt{U_t}$ and $\sqrt{U_t} + y_0 / (2\sqrt{U_t}) = \frac{3}{2} \sqrt{U_t}$. Using that $\erf(x) \geq 1 - \exp(-x^2)$ for $x\geq 0$, we have 
    \begin{align*}
        I(y_0;U_t) &= \frac{\exp(y_0^2 / (4U_t))}{2\sqrt{U_t/\pi}}\left( \erf\bigg( \frac{\sqrt{U_t}}{2}\bigg) + \erf\bigg(\frac{3\sqrt{U_t}}{2}\bigg)\right) \\ 
        &\geq \frac{\exp(U_t/4)}{2\sqrt{U_t/\pi}}\left( 2  -  \exp\bigg( - \frac{U_t}{4}\bigg) - \exp\bigg( - \frac{9U_t}{4}\bigg)\right) \\ 
        &= \sqrt{\frac{\pi}{U_t}}\left( \exp\bigg(\frac{U_t}{4}\bigg)  -  \frac{1}{2}+ \frac{1}{2}\exp( - 2U_t)\right) \xrightarrow{U_t\to\infty} \infty. 
    \end{align*}
    Therefore, for large enough $t$, $I(y_0;U_t) > G$. Since $I_t(y;U_t)$ is monotone in $y$ for $y\geq 0$, it follows that $y_t < y_0 = U_t$ if $I(y_t;U_t) = G$ for large enough $t$. 
\end{proof}

Now let us proceed to the proof of Theorem~\ref{thm:closed-form}. 
Let $y_t$ satisfy $I(y_t;U_t) = G_\alpha$. 
    Lemma~\ref{lem:yt<Ut} guarantees that for all $t$ satisfying~\eqref{eq:t-condition}, $y_t<U_t$. Therefore, for all such $t$, $\erf(\sqrt{U_t} - y_t / (2\sqrt{U_t})) \geq \erf(\sqrt{U_t}/2)$. Using that $\erf(x) \geq 1 - \exp(-x^2)$ for all $x\geq 0$, we have 
    \begin{align*}
        &\erf\bigg( \sqrt{U_t} - \frac{y_t}{2\sqrt{U_t}}\bigg) + \erf\bigg(\sqrt{U_t} + \frac{y_t}{2\sqrt{U_t}}\bigg)  \\ 
        &\geq 2\erf\bigg( \sqrt{U_t} - \frac{y_t}{2\sqrt{U_t}}\bigg)\\ 
        &\geq 2\bigg(1 - \exp\bigg(-\frac{U_t}{4}\bigg)\bigg). 
    \end{align*}
    Therefore, again for such $t$, using the expression in~\eqref{eq:It-explicit}, 
    \begin{align*}
        G_\alpha = I(y_t;U_t) &\geq \sqrt{\frac{\pi}{U_t}}\exp\bigg(\frac{y_t^2}{4U_t}\bigg)\bigg( 1 - \exp\left\{-\frac{U_t}{4}\right\}\bigg). 
    \end{align*}
    Taking logs and rearranging results in the expression 
    \begin{equation}
        y_t^2 \leq 4U_t\log\left(\frac{G_\alpha \sqrt{U_t/\pi}}{1 - \exp(-U_t/4)}\right) =  
        4U_t\log\left(\frac{\kappa Z \sqrt{2U_t}/\alpha}{1 - \exp(-U_t/4)}\right)
    \end{equation}
    From here, note that the solution $y_t$ to $I(y_t;U_t) = G_\alpha$ defines the boundary of our confidence sequence as $y_t = t|\Xbar_t - \widehat{m}|$, where $C_t^\mix = (\Xbar_t \pm \widehat{m})$. Taking the square root of both sides of the above display and solving for $\widehat{m}$, the result follows. 
    
\subsection{Proof of Lemma~\ref{lem:precise-width}}
\label{proof:precise-width}
First we show that $W_t^\apx \sim A_t^\apx$. 
Let $u = \E[\psi_E(|X - \mu|)]$. 
First we claim that 
\begin{equation}
\label{eq:pf-precise-width-1}
        \frac{U_t}{t} = \frac{\frac{1}{2\kappa^2} + \sum_{i\leq t} \psi_E(|X_i - \hatX_i|)}{t}\xrightarrow{t\to\infty} u, \quad P\text{-almost surely}. 
    \end{equation}
    To see this, write 
    \begin{align*}
        \frac{1}{t}\sum_{i\leq t}\psi_E(|X_i - \hatX_i|) = \underbrace{\frac{1}{t}\sum_{i\leq t} \psi_E(|X_i - \mu|)}_{(i)} + \underbrace{\frac{1}{t}\sum_{i\leq t}[\psi_E(|X_i - \hatX_i|) - \psi_E(|X_i - \mu|)]}_{(ii)}.
    \end{align*}
    Term (i) converges to $u$ $P$-almost surely by the SLLN and an application of the continuous mapping theorem. Next we 
    bound (ii) by appealing to the mean-value theorem. Notice that  
    \[\bigg|\frac{\partial}{\partial y} \psi_E(|x-y|)\bigg| = \frac{|x-y|}{1 - |x-y|}.\]
    Fix any $0<\eps<1/2$. By the SLLN there exists some $N_\eps$ such that for all $n\geq N$, $|\hatX_n - \mu|\leq \eps$ $P$-almost surely. Thus, for any $\theta$ in between $\hatX_n$ and $\mu$, we have $|\theta - \hatX_n| \leq \eps$. Therefore, by the mean-value theorem, there exists some $\theta$ in between $\hatX_n $ and $\mu$ such that 
    \begin{align*}
        |\psi_E(|X_n - \hatX_n|) - \psi_E(|X_n - \mu|)| \leq \bigg(\frac{|\hatX_n - \theta|}{1 - |\hatX_n - \theta|}\bigg) | \hatX_n - \mu| \leq \frac{\eps^2}{1 - \eps} . 
    \end{align*}
    Therefore, letting $\psi_i(y) = \psi_E(|X_i-y|)$ for brevity, we can decompose (ii) as 
    \begin{align*}
     \frac{1}{t}\sum_{i\leq t} [\psi_i(\hatX_i) - \psi_i(\mu)] 
        &= \frac{1}{t}\left\{ \sum_{i<N_\eps} [\psi_i(\hatX_i) - \psi_i(\mu)] + \sum_{i= N_\eps}^t [\psi_i(\hatX_i) - \psi_i(\mu)] \right\} \\ 
        &\leq \frac{1}{t}\sum_{i<N_\eps} [\psi_i(\hatX_i) - \psi_i(\mu)] + \frac{\eps^2(t - N_\eps)}{t(1-\eps)}.
    \end{align*}
    The first term vanishes as $t\to\infty$, and the second vanishes as we let $\eps\downarrow 0$. We have thus shown~\eqref{eq:pf-precise-width-1}. 
    Next, by Slutsky's theorem, 
    \begin{equation*}
        \frac{U_t \ell_\alpha }{t \log(t)}  \xrightarrow{t\to\infty} 0, \quad P\text{-almost surely.}
    \end{equation*}
    Further, letting $u = \E[\psi_E(|X - \mu|)]$, we can write $U_t = tu( 1 + \oas(1))$. Therefore, 
    \begin{align*}
        \frac{\log(2U_t)}{\log t} = \frac{\log(2u) + \log(1 + o(1))}{\log t} + \frac{\log t}{\log t} \xrightarrow{t\to\infty}1, \quad P\text{-almost surely.}
    \end{align*}
    Hence, again by Slutsky's theorem, 
    \begin{align*}
        \frac{U_t \log(2U_t)}{t\log t} \xrightarrow{t\to\infty} u, \quad P\text{-almost surely.}
    \end{align*}
    Putting everything together, 
    \begin{align*}
        \sqrt{\frac{t}{\log t}} W_t^\apx &= 2\sqrt{\frac{U_t(\ell_\alpha + \frac{1}{2}\log(2U_t))}{t\log t}} \xrightarrow{t\to\infty} \sqrt{2u} \quad P\text{-a.s.}
    \end{align*}
    which was to be shown. 

    As for $R_t$, write 
    \begin{align*}
        R_t^2 = \frac{2U_t(\ell_\alpha + \frac{1}{2}\log(2U_t)}{tu \log(t)}, 
    \end{align*}
    and note that $\ell_\alpha = \log(1/\alpha) + \log(\kappa Z/(1 - \exp(-U_t/4)) = \log(1/\alpha) + \Oas(1)$, since $U_t \sim tu \to \infty$ $P$-a.s.\ by the above analysis. Moreover, $\log(2U_t) /\log(t) = 1 + \oas(1)$ also by the above analysis, hence 
    \begin{align*}
        R_t^2 &= \frac{2U_t}{tu}\left(\frac{\log(1/\alpha)}{\log t} + \frac{\log(\kappa Z/(1-\exp(-U_t/4)}{\log t} + \frac{\log(2U_t)}{2\log t} \right) \\
        &= \frac{2U_t}{tu}\left(\frac{\log(1/\alpha)}{\log t} + \oas(1) + \frac{1}{2}  + \oas(1)\right),
    \end{align*}
    which was the claim.

\subsection{Proof of Theorem~\ref{thm:stitching}}
\label{proof:stitching}
Recall the bound in~\eqref{eq:fixed-xi-bound}: For any fixed $\xi\in(0,1]$, with probability $1-\alpha$, simultaneously for all $t\geq 1$, 
\begin{equation*}
    |\Xbar_t - \mu_t| \leq \frac{\xi V_t}{t} + \frac{\log(2/\alpha)}{t\xi}. 
\end{equation*}
Define $E_j:=[L_0\eta^j,L_0\eta^{j+1})$, $j\in\mathbb{N}_0$, where we recall that $\eta>0$ is some fixed parameter. Consider applying the above bound once in every epoch, with parameters $\xi_j$ and $\alpha_j$ indexing each epoch. For $\alpha_j = \alpha/h(j)$ we have 
\begin{align*}
    &P\left(\exists t\geq 1, \exists j\in\mathbb{N}_0: |S_t-t\mu_t| \geq \xi_j V_t + \xi_j^{-1}\log(2/\alpha_j) \right) \\ 
    &= P\left(\bigcup_{j\in\mathbb{N}_0} \{\exists t\geq 1: |S_t-t\mu_t| \geq \xi_j V_t + \xi_j^{-1}\log(2/\alpha_j)\right) \\ 
    &\leq \sum_{j\in\mathbb{N}_0} P\left(\exists  t\geq 1: |S_t-t\mu_t| \geq \xi_j V_t + \xi_j^{-1}\log(2/\alpha_j)\right) \leq \alpha \sum_{j\geq 0} \frac{1}{h(j)}\leq \alpha. 
\end{align*}
In other words, with probability $1-\alpha$, simultaneously for all $t\geq 1$, 
\begin{equation*}
    |\Xbar_t - \mu_t| \leq \inf_{j\in\mathbb{N}_0} \left\{ \frac{\xi_j V_t}{t} + \frac{ \log(2/\alpha_j)}{t\xi_j} \right\}. 
\end{equation*}
We need to pick the parameters $\xi_j$ such that the bound shrinks at the desired rate. Consider 
\[\xi_j := \sqrt{\frac{\log(2/\alpha_j)}{L_0\eta^{j+1}}}\wedge 1.\]
For any given $t$, consider the $j=j(t)$ such that $V_t \in E_{j(t)} = [L_0\eta^{j(t)},L_0\eta^{j(t)+1})$,  which exists as long as $V_t \geq L_0 \eta^{0} = L_0$. 
Now, we claim that, as long as $h(\log_\eta(V_t/L_0))\leq \frac{\alpha}{2}\exp(V_t)$, then $\xi_{j(t)} = \sqrt{\log(2/\alpha_{j(t)})/L_0\eta^{j(t)+1}}$. Indeed, by definition of $j(t)$ we have $j(t) \leq \log_\eta(V_t/L_0)$ and since $h$ is nondecreasing, 
\begin{align*}
    \log\left(\frac{2}{\alpha_{j(t)}}\right) &= \log\left(\frac{2h(j(t))}{\alpha}\right) 
    \leq  
    \log \left(\frac{2}{\alpha}h\left(\log_\eta\left(\frac{V_t}{L_0}\right)\right)\right) \leq V_t  \leq L_0\eta^{j(t)+1},
\end{align*}
implying that $\xi_{j(t)}\leq 1$. 

Next, we have $L_0\eta^{j(t)+1} > V_t$ so $\xi_{j(t)} \leq \sqrt{\log(2/\alpha_{j(t)})/V_t}$. Also $L_0\eta^{j(t)+1} = \eta L_0\eta^{j(t)} \leq \eta V_t$ hence $\xi_{j(t)}^{-1} \leq \sqrt{\eta V_t / \log(2/\alpha_{j(t)})}$. Together these inequalities give that 
\begin{align*}
    \inf_{j\in\mathbb{N}_0} \left\{ \frac{\xi_j V_t}{t} + \frac{ \log(2/\alpha_j)}{t\xi_j} \right\} &\leq \frac{\xi_{j(t)} V_t}{t} + \frac{ \log(2/\alpha_{j(t)})}{t\xi_{j(t)}} \\ 
    &\leq \frac{1}{t}\left(\sqrt{V_t\log(2/\alpha_{j(t)})} + \sqrt{\eta V_t \log(2/\alpha_{j(t)})}\right) \\
    &= \frac{1 + \sqrt{\eta}}{t} \sqrt{V_t[\log(2/\alpha) + \log(h(j(t)))]}.  
\end{align*}
The final bound then follows from noting that $V_t \geq L_0\eta^{j(t)}$ so $j(t)\leq \log_\eta(V_t/L_0)$. This completes the proof.

\subsection{Proof of Theorem~\ref{thm:eb-matrix}}
\label{proof:eb-matrix}
The proof follows a similar logic to the scalar case. The calculation of $N_t$ is the same as the proof in~\ref{proof:mixture-cs-tn} after replacing $S_t - tm$ by $\gamma_{\max}(\bS_t - t\bM)$ and $V_t$ by $\gamma_{\max}(\bV_t)$. In particular, 
\begin{equation*}
    N_t = \frac{1}{\kappa Z\sqrt{2\pi}} I\left(\gamma_{\max}(\bS_t - t\bM_t); \gamma_{\max}(\bV_t) + \frac{1}{2\kappa^2)}\right).
\end{equation*}
By Ville's inequality, $P(\exists t\geq 1: N_t \geq  d/\alpha) \leq \E[N_0]\alpha/d \leq \E[L_0] \alpha/d\leq d$, since $L_0 = d$. Therefore, with probability $1-\alpha$, 
\begin{equation*}
    I\left(\gamma_{\max}(\bS_t - t\bM_t); \frac{1}{2\kappa^2} + \gamma_{\max}(\bV_t)\right) \leq \log\left(\frac{d \kappa Z\sqrt{2\pi}}{\alpha}\right).  
\end{equation*}
Overloading notation from the scalar case, let $U_t = 1/(2\kappa^2) + \gamma_{\max}(\bV_t)$ and consider $y_t$ such that $I(y_t;U_t) = d\kappa Z\sqrt{2\pi}/\alpha =: G_{\alpha,d}$. As per the discussion in Section~\ref{sec:mom}, $I(0;U_t) < G_{\alpha,d}/d \leq G_{\alpha,d}$, so $y_t$ always exists. Then, appealing to Lemma~\ref{lem:yt<Ut}, if $t\geq t_0$ we have $y_t<U_t$. Thus, just as in the scalar case, 
\[G_{\alpha,d} = I(y_t;U_t) \geq \sqrt{\frac{\pi}{U_t}}\exp\left(\frac{y_t^2}{4U_t}\right)\left( 1 - \exp\left\{- \frac{U_t}{4}\right\}\right).\]
Therefore, since $I$ is increasing in $|y_t|$, we have
\[\gamma_{\max}^2 (\bS_t -t\bM_t) \leq y_t^2 \leq 4U_t \left(\log (G_{\alpha,d})-\log\left(1 - \exp\left\{-\frac{U_t}{4}\right\}\right)\right),\]
which is the desired result. 

\subsection{Proof of Lemma~\ref{lem:matrix-nsm}}
\label{sec:proof-matrix-nsm}

We will require the following lemma, which originates in \citet[Lemma 4]{howard2020time} and was slightly generalized in \citet[Lemma 2.2]{wang2024sharp}. 
\begin{lemma}
\label{lem:matrix-tr-eprocess}
    Let $(\b{A}_t)$ be a sequence of symmetric matrices which forms a martingale difference sequence. Let $(\b{B}_t)$ be adapted and $(\b{C}_t)$ be predictable. If $\E_{t-1}\exp\{\b{A}_t - \b{B}_t\} \preceq \exp\{\b{C}_t\}$ for all $t$, then 
    \begin{equation*}
     L_t = \Tr\exp\bigg\{ \sum_{i\leq t} \b{A}_i - \sum_{i\leq t} (\b{B}_i + \b{C}_i) \bigg\},   
    \end{equation*}
    defines a nonnegative supermartingale.  
\end{lemma}

Recall our modified Fan's inequality: $\exp\{ \xi\lambda - \psi_E(|\lambda|)\xi^2\} \leq 1 + \lambda\xi$ for $\xi\in[-1,1]$ and $\lambda\in(-1,1)$. Note that by assumption, the eigenvalues of $\bX_t - \bhatX_t$ lie in $(-1,1)$. Therefore, for any $\xi\in[-1,1]$, by Tropp's transfer rule,\footnote{\citet[Equation 2.2]{tropp2012user} shows that for $I\subset \Re$ if $f,g:I\to\Re$ satisfy $f(x)\leq g(x)$ for all $x$ then $f(\bX) \leq g(\bX)$ for all symmetric matrices $\bX$ with eigenvalues in $I$. } 
\begin{equation*}
    \exp\{ \xi (\bX_t - \bhatX_t) - \xi^2 \psi_E(|\bX_t - \bhatX_t|)\} \preceq 1 + \xi (\bX_t - \bhatX_t). 
\end{equation*}
Taking expectations, 
\begin{align*}
    \E_{t-1} \exp\{ \xi(\bX_t - \bhatX_t) - \xi^2 \psi_E(|\bX_t - \bhatX_t|)\} &\preceq 1 + \xi(\E_{t-1}[\bX_t] - \bhatX_t) \\
    &\preceq \exp\{\xi (\E_{t-1}[\bX_t] - \bhatX_t)\}.  
\end{align*}
Now consider applying Lemma~\ref{lem:matrix-tr-eprocess} the above result with $A_t = \xi_t(\bX_t - \E_{t-1}[\bX_t])$, $B_t = \xi_t(\bhatX_t - \E_{t-1}[\bX_t]) + \xi_t^2\psi_E(|\bX_t - \bhatX_t|)$, and $C_t = \xi_t(\bM_t - \bhatX_t)$. Clearly $(C_t)$ is predictable and $(A_t)$ is a martingale difference sequence. This implies that $(L_t)$ is a nonnegative supermartingale, completing the proof of the first part of the lemma. 

For the second part, cnsider fixing $\xi_i=\xi$. Let $L_t(\xi) = \Tr\exp(\xi(\bS_t - t \bM_{t}) - \xi^2 \bV_t)$. By Fubini, $\int L_t(\xi) \rho(\xi)\d\xi$ is again a nonnegative supermartingale. Now, notice that for any $\xi\in[-1,1]$, 
\begin{equation}
    \gamma_{\max}(\xi (\bS_t - t \bM) - \xi^2 \bV_t) \geq \xi \gamma_{\max}(\bS_t - t\bM) - \xi^2 \gamma_{\max}( \bV_t). 
\end{equation}
To see this, let $\b{u}$ be a unit vector such that $\b{u}^\top (\bS_t - t\bM_t)\b{u} = \gamma_{\max}(\bS_t - t\bM_t)$. Then 
\begin{align*}
   \gamma_{\max}(\xi (\bS_t - t \bM_t) - \xi^2 \bV_t)  &\geq \b{u}^\top (\xi (\bS_t - t \bM_t) - \xi^2 \bV_t) \b{u} \\ 
   &= \xi \gamma_{\max}(\bS_t - t\bM_t) - \xi^2 \b{u}^\top \bV_t \b{u} \\
   &\geq \xi \gamma_{\max}(\bS_t - t\bM_t) - \xi^2 \gamma_{\max}(\bV_t), 
\end{align*}
since $\gamma_{\max} (\bV_t) \geq \b{u}^\top \bV_t \b{u}$ and $\xi^2\geq 0$. Therefore, 
\begin{equation*}
    \Tr L_t(\xi) \geq \exp\left\{ \gamma_{\max}(\xi (\bS_t - t\bM_t) - \xi^2 \bV_t)\right\} \geq \exp\left\{ \xi \gamma_{\max}(\bS_t - t\bM_t) - \xi^2 \gamma_{\max}(\bV_t)\right\} , 
\end{equation*}
and 
\begin{equation*}
    \int \Tr L_t(\xi)\rho(\xi)\d\xi  \geq \int \exp\left\{\xi \gamma_{\max}(\bS_t - t\bM_t) - \xi^2 \gamma_{\max}(\bV_t) \right\}\rho(\xi)\d\xi = N_t.  
\end{equation*}
Thus, $(N_t)$ is upper bounded by a  nonnegative supermartingale. 

\subsection{Proof of Proposition~\ref{prop:mat-asymp-width}}
\label{proof:mat-asymp-width}

We begin by showing that $\gamma_{\max}(\bV_t)/t \to \gamma_{\max}(\Sigma_\psi)$ almost surely, where $\Sigma_\psi = \E_P[\psi(|\bX - \bM|)]$. Note that for $\b{A},\b{B}\in\calS_d$, $\gamma_{\max}$ (the operator norm) is 1-Lipschitz, i.e., 
\[|\gamma_{\max}(\b{A}) - \gamma_{\max}(\b{B})|\leq \gamma_{\max}(\b{A} - \b{B}).\]
(This follows from Weyl's inequality.) 
Consequently, it suffices to show that 
\[\gamma_{\max}\bigg(\frac{1}{t}\sum_{i\leq t} \psi_E(|\b{X}_i - \bhatX_i|) - \Sigma_\psi\bigg)\to 0,\]
almost surely. The map $F:\calS_d\to\calS_d$ given by $F(\b{A}) = F(\psi_E(|\b{A}|)$ is continuous. Since $\calS_d$ is compact (it is closed and bounded), $F$ is in fact uniformly continuous. In particular, for every $\eps>0$ there exists some $\delta>0$ such that $\gamma_{\max}(\b{A} - \b{B})\leq \delta$ implies $\gamma_{\max}(F(\b{A}) - F(\b{B}))\leq \eps$.

Now, note that $\gamma_{\max}(\bhatX_i - \bM)\to 0$ almost surely, so there exists some sequence $(r_n)$ such that $r_n\to 0$ almost surely with  $\gamma_{\max}(\bhatX_i - \bM)  \leq r_i$. 
Fix $\eps>0$ and let $\delta>0$ be the corresponding modulus of uniform continuity. For all $\omega\in\Omega$ in a probability-one event, there exists some $N = N(\omega)$ such that $
\gamma_{\max}\big(\bhatX_i(\omega) - \bM\big)\leq \delta$ for all $i\geq N$. 
Set
\[
\b{A}_i(\omega) := \bX_i(\omega) - \bhatX_i(\omega),
\qquad
\b{B}_i(\omega) := \bX_i(\omega) - \bM.
\]
Then
\[
\gamma_{\max}\big(\b{A}_i(\omega) - \b{B}_i(\omega)\big)
=
\gamma_{\max}\big(\bhatX_i(\omega) - \bM\big)
\leq \delta,
\]
where 
\[\b{Y}_i = \psi_E(|\bX_i - \bhatX_i|),\quad \overline{\b{Y}}_i = \psi_E(|\bX_i - \bM|).\]
So, by uniform continuity of $F$,
\[
\gamma_{\max}\big(\b{Y}_i(\omega) - \overline{\b{Y}}_i(\omega)\big)
=
\gamma_{\max}\big(F(\b{A}_i(\omega)) - F(\b{B}_i(\omega))\big)
\leq \eps
\quad\text{for all } i\ge N.
\]
Therefore, for $t\geq N$,
\begin{align*}
    \gamma_{\max}\bigg(\frac{1}{t} \sum_{i\leq t}\big[\b{Y}_i(\omega) - \overline{\b{Y}}_i(\omega)\big]\bigg)
    &\leq \frac{1}{t}\sum_{i=1}^{t}  \gamma_{\max}\big(\b{Y}_i(\omega) - \overline{\b{Y}}_i(\omega)\big) \\ 
    &\leq \frac{1}{t}\sum_{i=1}^{N-1} \gamma_{\max}\big(\b{Y}_i(\omega) - \overline{\b{Y}}_i(\omega)\big)
        + \frac{1}{t}\sum_{i=N}^t \eps \\
    &= O(1/t) + \eps,
\end{align*}
where the $O(1/t)$ term may depend on $\omega$ but vanishes as $t\to\infty$.
Taking $\limsup_{t\to\infty}$, we obtain
\[
\limsup_{t\to\infty}
\gamma_{\max}\bigg(\frac{1}{t} \sum_{i\leq t}\big[\b{Y}_i(\omega) - \overline{\b{Y}}_i(\omega)\big]\bigg)
\le \eps.
\]
Since $\eps>0$ was arbitrary, this implies
\[
\gamma_{\max}\bigg(\frac{1}{t} \sum_{i\leq t}\big[\b{Y}_i - \overline{\b{Y}}_i\big]\bigg)
\;\xrightarrow[t\to\infty]{\text{a.s.}}\; 0.
\]
Combining this with
\[
\frac{1}{t}\sum_{i\le t}\overline{\b{Y}}_i \xrightarrow{\text{a.s.}} \Sigma_\psi,
\]
we conclude that
\[
\gamma_{\max}\bigg(\frac{1}{t}\sum_{i\le t}\b{Y}_i\bigg)
=
\gamma_{\max}\bigg(\frac{1}{t}\sum_{i\le t}\psi_E(|\bX_i - \bhatX_i|)\bigg)
\xrightarrow{\text{a.s.}}\gamma_{\max}(\Sigma_\psi),
\]
which is what we wanted to show. The rest of the proof follows that of Lemma~\ref{lem:precise-width}.

\section{Additional theoretical results}

\subsection{On the width of $C_t^\prpl$}
\label{app:wsr-width}

Let \begin{equation*}
    W_t^\prpl = \frac{\log(2/\alpha) + \sum_{i\leq t} (X_i - \muhat_{i-1})^2 \psi_E(\lambda_i)}{\sum_{i\leq t} \lambda_i},
\end{equation*}
be the halfwidth of the closed-form empirical Bernstein bound of \citet{waudby2024estimating}. 
They instantiate this bound with 
\[\lambda_t = \sqrt{\frac{2\log(2/\alpha)}{\sigmahat_{t-1}^2 t\log (1+t)}}\wedge \frac{1}{2}, \quad \sigmahat_{t}^2 = \frac{1/4 + \sum_{i\leq t}(X_i - \muhat_i)^2}{t+1}, \quad \muhat_t = \frac{1/2 + \sum_{i\leq t}X_i}{t+1}.\]
The goal of this section is to analyze the asymptotic width of $W_t^\prpl$ by proving the following proposition: 

\begin{proposition}
\label{prop:wsr-asymp-width}
    Let $(X_t)$ be iid. If $\muhat_t$, $\sigmahat_t$, and $\lambda_t$ are as above, then 
    \begin{equation}
        W_t^\prpl \sim \frac{\sigma}{2} \sqrt{\frac{\log(2/\alpha) \log(t)}{2t}} (1 + \log\log(t)). 
    \end{equation}
\end{proposition}
We prove Proposition~\ref{prop:wsr-asymp-width} via a sequence of lemmas. The first two are general statements regarding almost sure convergence that we had trouble finding elsewhere. 

\begin{lemma} 
\label{lem:sum-asymptotics}
    Suppose a nonnegative sequence $(a_t)$ has partial sums which diverge, and suppose another sequence $(b_t)$ converges to 0. Then $\sum_{i\leq t} a_i(1 + b_i) \sim \sum_{i\leq t} a_i$. The same statement holds if convergence and divergence hold almost surely. 
\end{lemma}
\begin{proof}
Let $A_t = \sum_{i\leq t}a_i$ and $C_t = \sum_{i\leq t} a_ib_i$. 
Write $
\sum_{i=1}^t a_i(1+b_i) = A_t + C_t$. 
We claim $C_t = o(A_t)$, which implies
\[
\frac{\sum_{i=1}^t a_i(1+b_i)}{A_t}
= 1 + \frac{C_t}{A_t} \longrightarrow 1,
\]
which is the desired conclusion. 
Fix $\eps>0$. Since $b_t \to 0$, there exists $N$ such that
$|b_i|\le \eps$ for all $i\ge N$. For $t \ge N$,
\[
|C_t|
\le \sum_{i=1}^{N-1} a_i |b_i| + \sum_{i=N}^t a_i |b_i|
\le C_N + \eps \sum_{i=N}^t a_i,
\]
where $C_N := \sum_{i=1}^{N-1} a_i |b_i| < \infty$ is a constant. Dividing by $A_t$ and using $A_t = \sum_{i=1}^t a_i$,
\[
\frac{|C_t|}{A_t}
\le \frac{C}{A_t} + \eps \cdot \frac{\sum_{i=N}^t a_i}{A_t}
= \frac{C}{A_t} + \eps \left(1 - \frac{\sum_{i=1}^{N-1} a_i}{A_t}\right).
\]
As $t\to\infty$ we have $A_t\to\infty$, so $C/A_t\to 0$ and
$\sum_{i=1}^{N-1} a_i / A_t \to 0$. Hence
\[
\limsup_{t\to\infty} \frac{|C_t|}{A_t} \le \eps.
\]
Since $\eps>0$ was arbitrary, we conclude $C_t/A_t \to 0$, i.e.\ $C_t = o(A_t)$ as claimed.
\end{proof}

\begin{lemma}[Weighted SLLN]
\label{lem:weighted-slln}
    Let $(Y_t)$ be a sequence of bounded iid random variables with mean $\mu$ and let $(w_t)$ be a sequence of positive deterministic weights such that $\sum_{i=1}^\infty w_i = \infty$ and $\sum_{i=1}^\infty w_i^2<\infty$.  Then 
    \begin{equation}
      \frac{\sum_{i\leq t} w_i Y_i}{\sum_{i\leq t} w_i} \xrightarrow{a.s.}\mu,  
    \end{equation}
    and consequently 
    \begin{equation}
        \sum_{i\leq t}w_i Y_i \sim \mu\sum_{i\leq t}w_i. 
    \end{equation}
\end{lemma}
\begin{proof}
    We appeal to Kolmogorov's three-series theorem, which states that if $(Z_n)$ is a sequence of independent random variables, then $\sum_{n=1}^\infty Z_n$ converges almost surely if there exists some $c>0$ such that, letting $T_n = Z_n \mathbf{1}\{|Z_n| \leq c\}$, the following three series converge: (i) $\sum_{n=1}^\infty P(|Z_n|\geq c)$, (ii) $\sum_{n=1}^\infty \E T_n$, and (iii) $\sum_{n=1}^\infty \E[(T_n - \E T_n)^2]$.  Let us apply this result with $Z_n = w_n(Y_n-\mu)$ and any $c\geq \sup_n |w_n|K$ where $K \geq\sup_n |Y_n - \mu|$ (which exists by assumption), so that $T_n = Z_n$.  Let $\sigma^2 = \E[(Y_1-\mu)^2]$.  For the series in (i), we have $P(|w_n (Y_n-\mu)\geq c)=0$ by our selection of $c$. Similarly, 
    for (ii), we have $\E[T_n] = \E[w_n (Y_n-\mu)] = w_n \E[Y_n-\mu] = 0$. And for (iii), $c$ once again has no effect, and 
    $\sum_{n=1}^\infty \E[(T_n - \E T_n)^2] = \sum_{n=1}^\infty \E[w_n^2 (Y_n-\mu)^2]  \leq \sigma^2 \sum_{n=1}^\infty w_n^2 < \infty$. We conclude that $\sum_{n=1}^\infty w_n(Y_n-\mu)$ converges almost surely, hence $\sum_n w_n(Y_n -\mu)/ \sum_n w_n$ converges almost surely to 0. This is the desired conclusion. 
\end{proof}

The following two lemmas help us analyze the asymptotics of $\sum_{i\leq t} \lambda_i$. 

\begin{lemma}
\label{lem:1/ilogi}
    \begin{equation*}
        \sum_{k=1}^t \frac{1}{\sqrt{k\log(k+1))}} \sim 2\sqrt{\frac{t}{\log t}}. 
    \end{equation*}
\end{lemma}
\begin{proof}
For $x\geq 1$ let $
f(x) = 1/\sqrt{x\log (x+1)}$. 
Then $f$ is positive and decreasing on $(1,\infty)$, so 
\[
\int_1^t f(x)\,dx \;\le\; \sum_{k=1}^t f(k)
\;\le\; f(1) + \int_1^t f(x)\,dx.
\]
Define
\[
I(t) := \int_1^t \frac{dx}{\sqrt{x\log (x+1)}}.
\]
The inequalities above show that $
\sum_{k=1}^t f(k) = I(t) + O(1)$. 
In particular, if $I(t)\to\infty$ then 
\begin{equation}
\label{eq:pf-ilogi-1}
\frac{\sum_{k=1}^t f(k)}{I(t)} \longrightarrow 1.    
\end{equation}
In particular, if we can show that $I(t) \sim 2\sqrt{t/\log(t)}=:F(t)$, since in that case~\eqref{eq:pf-ilogi-1} holds and we  thus have $\sum_{k=1}^t f(k) \sim I(t) \sim F(t)$ which is the desired result. 

To see that $I(t) \sim F(t)$ first compute 
\[
F'(t)
= \frac{1}{\sqrt{t\log t}}
  \left(1 - \frac{1}{\log t}\right) \text{~ and ~} I'(t) = \frac{1}{\sqrt{t\log(t+1)}}. 
\]
Next note that $I(t)\xrightarrow{t\to\infty}\infty$ since 
\[I(t) \geq \int_{t/2}^t f(x) dx \geq \frac{t}{2} f(t) = \frac{1}{2}\sqrt{\frac{t}{\log(t+1)}} \to\infty. \]
Therefore, since both $I(t)$ and $F(t)$ tend to $+\infty$ as $t\to\infty$, so we can
apply l'H\^opital's rule to analyze the limit of the ratio $I(t)/F(t)$:
\[
\lim_{t\to\infty} \frac{I(t)}{F(t)}
= \lim_{t\to\infty} \frac{I'(t)}{F'(t)}
= \lim_{t\to\infty}
\frac{\sqrt{t\log t}}{\sqrt{t\log(t+1)}(1  - 1/\log t)} 
= 1.
\]
Thus $I(t) \sim F(t)$, which completes the proof. 
\end{proof}

\begin{lemma}
\label{lem:lambdat-asymptotics}
    \begin{equation}
        \sum_{i=1}^t\lambda_i \sim 2\sqrt{\frac{2\log(2/\alpha)}{\sigma^2}} \sqrt{\frac{t}{\log t}}. 
    \end{equation}
\end{lemma}
\begin{proof}
The strong law of large numbers guarantees that there exists a sequence $(s_t)$ which converges to 0 almost surely such that $\sigma/\sigmahat_{t-1} = 1 + s_t$. Therefore, 
\begin{align*}
    \sum_{i\leq t} \lambda_i &= \sum_{i\leq t} \sqrt{\frac{2\log(2/\alpha)}{\sigmahat_{i-1}^2 i\log(i+1)}\cdot \frac{\sigma^2}{\sigma^2}} 
    = \sum_{i\leq t} \sqrt{\frac{2\log(2/\alpha)}{\sigma^2 i\log(i+1)}} (1 + s_i) \sim \sum_{i\leq t} \sqrt{\frac{2\log(2/\alpha)}{\sigma^2 i\log(i+1)}}. 
\end{align*}
The proof is then completed by appealing to Lemmas~\ref{lem:sum-asymptotics} and~\ref{lem:1/ilogi}. 
\end{proof}

Finally, we analyze the asymptotics of the empirical variance-like term. 
\begin{lemma}
\label{lem:psiE-asymptotics}
\begin{equation}
    \sum_{i\leq t} (X_i-\muhat_{i-1})^2\psi_E(\lambda_i) \sim \log(2/\alpha)\log\log t,
\end{equation}  
\end{lemma}
\begin{proof}
    \citet{waudby2024estimating} show that $\lambda^2/(2\psi_E(\lambda)) \to 1$ as $\lambda\to 0$ (note their definition of $\psi_E(\lambda)$ is scaled by a factor of 4). Therefore, 
    \[\frac{\lambda_t^2/2}{\psi_E(\lambda_t)} = 1 + s_t,\]
    where $s_t\to 0$ almost surely. Moreover, $\sigmahat_{t-1}^2/\sigma^2 = 1 + r_t$ for $r_t\to 0$ almost surely, hence 
    \begin{align}
        \sum_{i\leq t} (X_i - \muhat_{i-1})^2 \psi_E(\lambda_i) & = \frac{1}{2}\sum_{i\leq t} \lambda_i^2 (X_i - \muhat_{i-1})^2 (1 + s_i) \\ 
        &= \log(2/\alpha) \sum_{i\leq t}  \frac{(X_i - \muhat_{i-1})^2}{\sigmahat_{i-1}^2 i\log(i+1)} (1 + s_i) \\ 
        &= \frac{\log(2/\alpha)}{\sigma^2} \sum_{i\leq t}  \frac{(X_i - \muhat_{i-1})^2}{i\log(i+1)} (1 + s_i)(1 + r_i). \label{eq:pf-psiE-asymptotics-1}
    \end{align}
    Let $A_t = \sum_{i\leq t} (X_i - \muhat_{i-1})^2 / (i\log(i+1))$ and note that $(X_i - \muhat_{i-1})^2 = (X_i - \mu)^2 + (\muhat_{i-1}-\mu)^2 - 2(X_i -\mu)(\muhat_{i-1}-\mu)$, so 
    \begin{equation*}
        A_t = \sum_{i\leq t} \frac{(X_i - \muhat_{i-1})^2}{i\log(i+1)} = \underbrace{\sum_{i\leq t} \frac{(X_i - \mu)^2}{i\log(i+1)}}_{(i)} + \underbrace{\sum_{i\leq t} \frac{(\mu - \muhat_{i-1})^2}{i\log(i+1)}}_{(ii)} -\underbrace{2 \sum_{i\leq t} \frac{(X_i - \mu)(\muhat_{i-1}-\mu)}{i\log(i+1)}}_{(iii)}. 
    \end{equation*}
    Letting $w_i = (i \log(i+1))^{-1}$, term (i) is asymptotically equivalent to $\sigma^2\log\log(t)$ by  Lemma~\ref{lem:weighted-slln}.  For (ii), note that $\E[(\muhat_{t-1} - \mu)^2] \leq C/t$ for some constant $C>0$. Therefore, applying the monotone convergence theorem, 
    \begin{align*}
        \E\bigg[\sum_{i\leq t} \frac{(\mu - \muhat_{i-1})^2}{i \log(i+1)}\bigg] &= \sum_{i\leq t} \frac{\E(\mu - \muhat_{i-1})^2}{i\log(i+1)} \leq \sum_{i\leq t} \frac{C}{i^2\log(i+1)} <\infty.  
    \end{align*}
    This shows that the sum in (ii) is $O(1)$ almost surely. Indeed, if it was $o(1)$ with some positive probability, then the expectation would be infinite. 
    Finally, (iii) can be handled by a martingale convergence theorem. Note that $\mu_{t-1}-\mu$ is $\calF_{t-1}$-measurable. Hence $N_t = (X_t - \mu)(\muhat_{t-1}-\mu)$ satisfies $\E[N_t|\calF_{t-1}] = 0$, so letting $M_t = \sum_{i\leq t}w_i N_i$ where $w_i = 1/(i\log(i+1))$ we have 
    \begin{equation*}
        \E\left[M_t\right] = 
        \E\left[\E\left[M_t|\calF_{t-1}\right]\right] = \E[M_{t-1}] = \dots = 0, 
    \end{equation*}
    once again implying that $\sup_t M_t$ is $O(1)$ almost surely, which in turn demonstrates that the sum in (iii) is $O(1)$ almost surely. Therefore, $A_t \sim \sigma^2\log\log(t)$. We may therefore invoke Lemma~\ref{lem:sum-asymptotics} to see that 
    \begin{align*}
        \frac{\log(2/\alpha)}{\sigma^2}\sum_{i\leq t}  \frac{(X_i - \muhat_{i-1})^2}{i\log(i+1)} (1 + s_i)(1 + r_i)  \sim \frac{\log(2/\alpha)}{\sigma^2} A_t \sim \log(2/\alpha) \log\log(t), 
    \end{align*}
    which when combined with~\eqref{eq:pf-psiE-asymptotics-1} is the desired conclusion.  
\end{proof}

From here, using Lemmas~\ref{lem:lambdat-asymptotics} and \ref{lem:psiE-asymptotics}, we have 
\begin{align*}
    W_t^\prpl &\sim \frac{\log(2/\alpha) + \log(2/\alpha)\log\log t}{\frac{2}{\sigma}\sqrt{2t\log(2/\alpha)/\log(t)}} \\ 
    %&= \frac{\sigma\sqrt{\log(2/\alpha)}}{2\sqrt{2}} \sqrt{\frac{\log(t)}{t}} ( 1 + \log\log t) \\ 
    &= \frac{1}{2} \sqrt{\frac{\sigma^2 \log(2/\alpha) \log(t)}{2t}}(1 + \log\log t),
\end{align*}
which is precisely the claim in Proposition~\ref{prop:wsr-asymp-width}. 

Meanwhile, we might be curious about the case when the asymptotic width does not depend on $\alpha$. Consider choosing $\lambda_t$ independent of $\alpha$, so  
\[\lambda_t = \sqrt{\frac{2}{\sigma_{t-1}^2 t\log(1+t)}}\wedge \frac{1}{2}.\]
Then by slightly modifying the lemmas above we may conclude that $\sum_{i\leq t} \lambda_i \sim 2\sqrt{2t / (\sigma^2\log t)}$ and $\sum_{i\leq t}(X_i - \muhat_{i-1})^2 \psi_E(\lambda_i)\sim \log\log(t)$. Therefore, in this case,  
\begin{align*}
    W_t^\prpl \sim \frac{\log(2/\alpha) + \log\log t}{\frac{2}{\sigma}\sqrt{2t/\log(t)}} = \frac{1}{2}\sqrt{\frac{\sigma^2 \log(t)}{2t}} (\log(2/\alpha) + \log\log(t)).  
\end{align*}

\subsection{On the asymptotic width of the predictable plug-in Hoeffding CS}
\label{sec:hoeff}

For the curious reader, let us also analyze the asymptotic width of the natural CS for sub-Gaussian random variables, following the predictable plug-in approach. Let $(X_t)_{t\geq 1}$ share a constant conditional mean $\mu$ and be $\sigma$-sub-Gaussian. That is, $\E_{t-1}[\exp(\lambda(X_t - \mu)]\leq \exp(\lambda^2\sigma^2/2)$ for all $t\geq 1$. Then, a mildly generalized argument of \citet[Proposition 1]{waudby2024estimating} gives that 
\begin{equation}
    C_t^\hoeff := \bigg(\frac{\sum_{i\leq t}\lambda_i X_i}{\sum_{i\leq t}\lambda_i} \pm \frac{\log(2/\alpha) + \frac{\sigma^2}{2}\sum_{i\leq t}\lambda_i^2}{\sum_{i\leq t}\lambda_i}\bigg). 
\end{equation}
constitutes $(1-\alpha)$-CS for $\mu$. 
Let $W_t^\hoeff$ denote its halfwidth. A natural choice is 
\[\lambda_t = \sqrt{\frac{2\log(2/\alpha)}{\sigma^2t\log(t+1)}}.\]
Using the analysis in the previous section, it follows that 
\[\sum_{i\leq t} \lambda_i \sim 2\sqrt{\frac{2\log(2/\alpha) t}{\sigma^2\log t}},\]
and $\sum_{i\leq t} \lambda_i^2 \sim 2\log(2/\alpha) \log\log(t)/\sigma^2$. 
Therefore, with this choice of $\lambda_t$,
\begin{align*}
    W_t^\hoeff \sim \frac{1}{2} \sqrt{\frac{\sigma^2\log(2/\alpha)\log(t)}{2t}}(1 + \log\log(t)). 
\end{align*}
Note that this is exactly precisely the same width as $W_t^\wsr$. Of course, also like in the case of $W_t^\wsr$, one can get rid of $\alpha$ dependence in the limit by choosing $\lambda_t$ to be independent of $\alpha$. In this case, for any finite $t$, the width will depend on $\log(1/\alpha)$ instead of $\sqrt{\log(1/\alpha)}$.

\subsection{Comparison to Robbins' mixture}
\label{sec:robbins-mixture}

Suppose $(X_t)$ are $\sigma^2$-sub-Gaussian with conditional mean $\mu$. That is, 
\begin{equation*}
 \E[\exp(\lambda X_t - \lambda^2\sigma^2/2)|\calF_{t-1}]\leq 1 \text{~ for all ~} t\geq 1.    
\end{equation*} 
Then, for any $a>0$, one can show that $(C_t^\rob)$ is a $(1-\alpha)$-CS for $\mu$, where $C_t^\rob := (\Xbar_t \pm W_t^\rob)$ and  
\begin{equation}
    W_t^\rob := \sqrt{\frac{2(1 + at\sigma^2)}{at^2}\left(\log(1/\alpha) + \frac{1}{2}\log(1 + at\sigma^2)\right)}.
\end{equation}
This CS originates with \citet{robbins1970statistical} who was one of the pioneers of using the method of mixtures to obtain CSs via nonnegative supermartingales. We use the superscript ``Rob'' to refer to Robbins. This explicit CS can be found in \citet[Section 3.2]{howard2021time} and in \citet[Equation (6)]{waudby2024time}. A version for random vectors can be found in \citet{chugg2025time}. 

Let us compare $W_t^\rob$ to $W_t^\apx$ in Theorem~\ref{thm:closed-form}. Set $U_t^\circ := 1 + at\sigma^2$. Then 
\begin{equation}
    W_t^\rob = \sqrt{\frac{2 U_t^\circ}{t^2}\left(\log(1/\alpha) + \frac{1}{2}\log(U_t^\circ)\right)}.
\end{equation}
This bears a heavy resemblance to $W_t^\apx = \frac{2}{t}\sqrt{U_t(\ell_\alpha + \frac{1}{2}\log(2U_t))}$, though the latter has several additional constants due to fact that we use a \emph{truncated} Gaussian in the mixture. Note that $U_t^\circ \sim at\sigma^2$, so $\log(U_t^\circ) /\log(t) =1 + \oas(1)$. Therefore,  
\begin{align*}
    \sqrt{\frac{t}{\log t}} W_t^\rob = \sqrt{\frac{2U_t^\circ}{t}\left(\frac{\log(1/\alpha)}{\log t} + \frac{1}{2} + \oas(1)\right)}. 
\end{align*}
Observe that that this term has the same dependence on $\alpha$ as $R_t$ in~\eqref{eq:apx-asymp-ratio}. In particular, 
\begin{equation}
    \lim_{t\to\infty} \sqrt{\frac{t}{\log t}}W_t^\rob = \sigma, 
\end{equation}
which like $A_t^\apx$, also has no dependence on $\alpha$. 

\subsection{\texorpdfstring{On the relationship between $\psi_E$ and $\sigma^2$}{On the relationship between sub-exp and variance}}
\label{app:psiE-vs-sigma}

Let us first prove that $\E\psi_E(|X-\mu|)\leq C_\mu \sigma^2$ for some constant $C_\mu$ depending only on $\mu$.  
Let $r(x) = \sum_{k\geq 2} x^{k-2}/k$. 
Note that $\psi_E(x) = \sum_{k\geq 2}x^k/k = x^2r(x)$. Since $r$ is an increasing function of $x\geq 0$ and $|X-\mu|\leq \max\{ \mu,1-\mu\} =: g(\mu)$, we have 
\begin{align*}
    \E\psi_E(|X-\mu|) \leq \E\bigg(|X-\mu|)^2 \sum_{k\geq 2} \frac{g^{k-2}(\mu)}{k}\bigg) = \sigma^2 r(g(\mu)) = \sigma^2 \frac{\psi_E(g(\mu))}{g^2(\mu)}.  
\end{align*}
Next, we show that this upper bound is achievable. For a fixed $\mu<1/2$ consider $X = 1$ with probability $\eps$ and $X = \frac{\mu-\eps}{1-\eps}$ with probability $1-\eps$. Then $\E X= \mu$ 
and $\E(X- \mu)^2 = g^2(\mu)\eps + (\frac{\mu-\eps}{1-\eps} - \mu)^2(1-\eps) = g^2(\mu) \eps + \eps^2 g^2(\mu)/(1-\eps)$. Also, 
$\E\psi_E(|X-\mu|) = \psi_E(1-\mu)\eps + \psi_E(g(\mu)\eps/(1-\eps))(1-\eps)$. 
So, 
\begin{align*}
\frac{\E\psi_E(|X-\mu|)}{\sigma^2} &= \frac{\eps\psi_E(g(\mu)) + \psi_E(\frac{g(\mu)\eps}{1-\eps})(1-\eps)}{\eps g^2(\mu) + \eps^2 g^2(\mu)/(1-\eps)} \\
&= \frac{\psi_E(g(\mu)) + \psi_E(\frac{g(\mu)\eps}{1-\eps})\frac{1-\eps}{\eps}}{ g^2(\mu) + \eps g^2(\mu)/(1-\eps)} \xrightarrow{\eps\to 0} \frac{\psi_E(g(\mu)}{g^2(\mu)}.    
\end{align*}

To give a sense of how the worst case difference between $\E\psi_E(|X-\mu|)$ and $\sigma^2$ affects our bounds, Figure~\ref{fig:psiE-vs-sigma} plots $C_t^\apx$ against the oracle Bernstein bound of \citet{howard2021time} for this two-point distribution. We find that $\eps=0.01$ maximizes the differences between the two CSs.

\begin{figure}
    \centering
    \begin{subfigure}{0.4\textwidth}
         \includegraphics[width=0.9\linewidth]{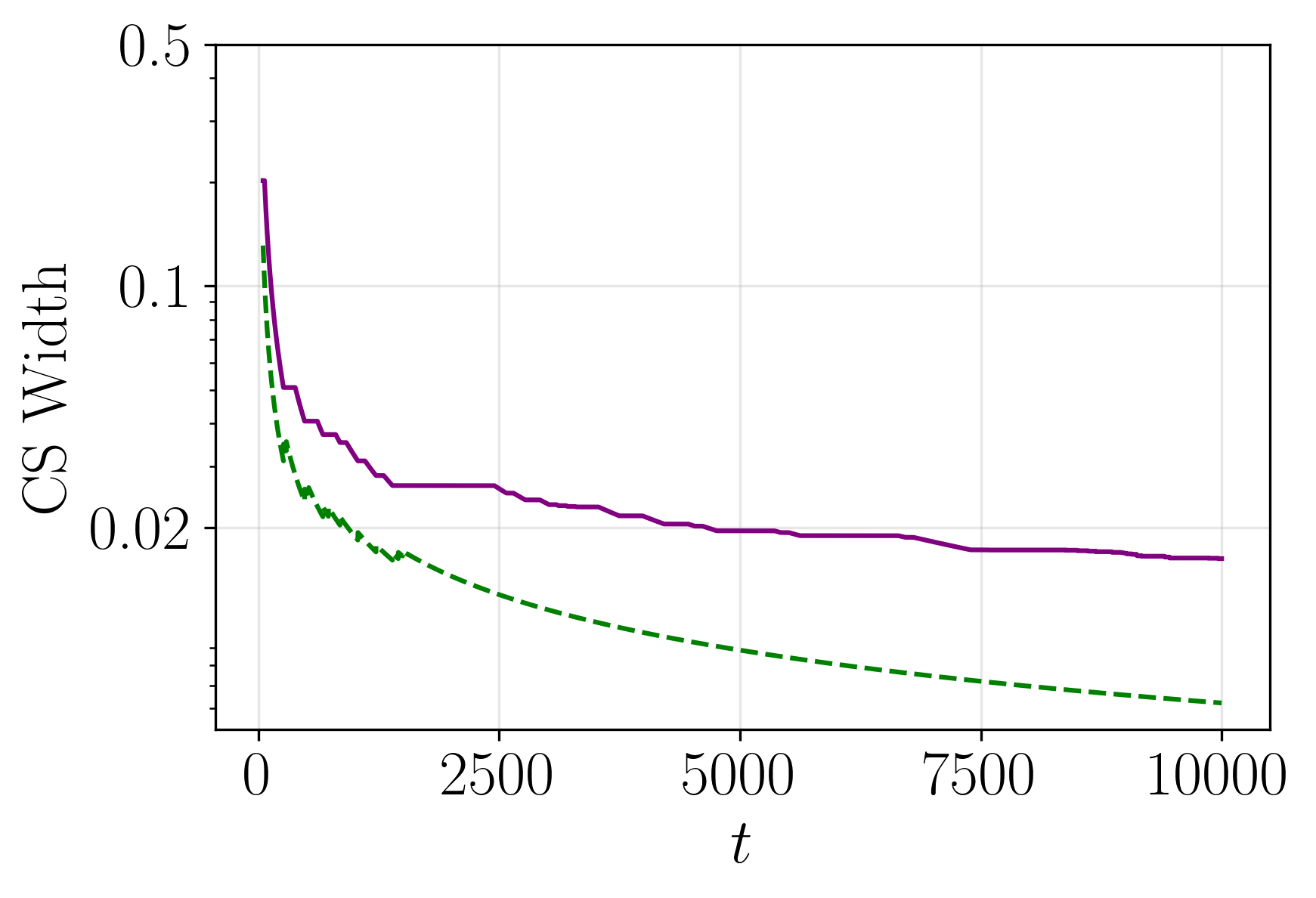}     
    \end{subfigure}
    \begin{subfigure}{0.4\textwidth}
         \includegraphics[width=0.9\linewidth]{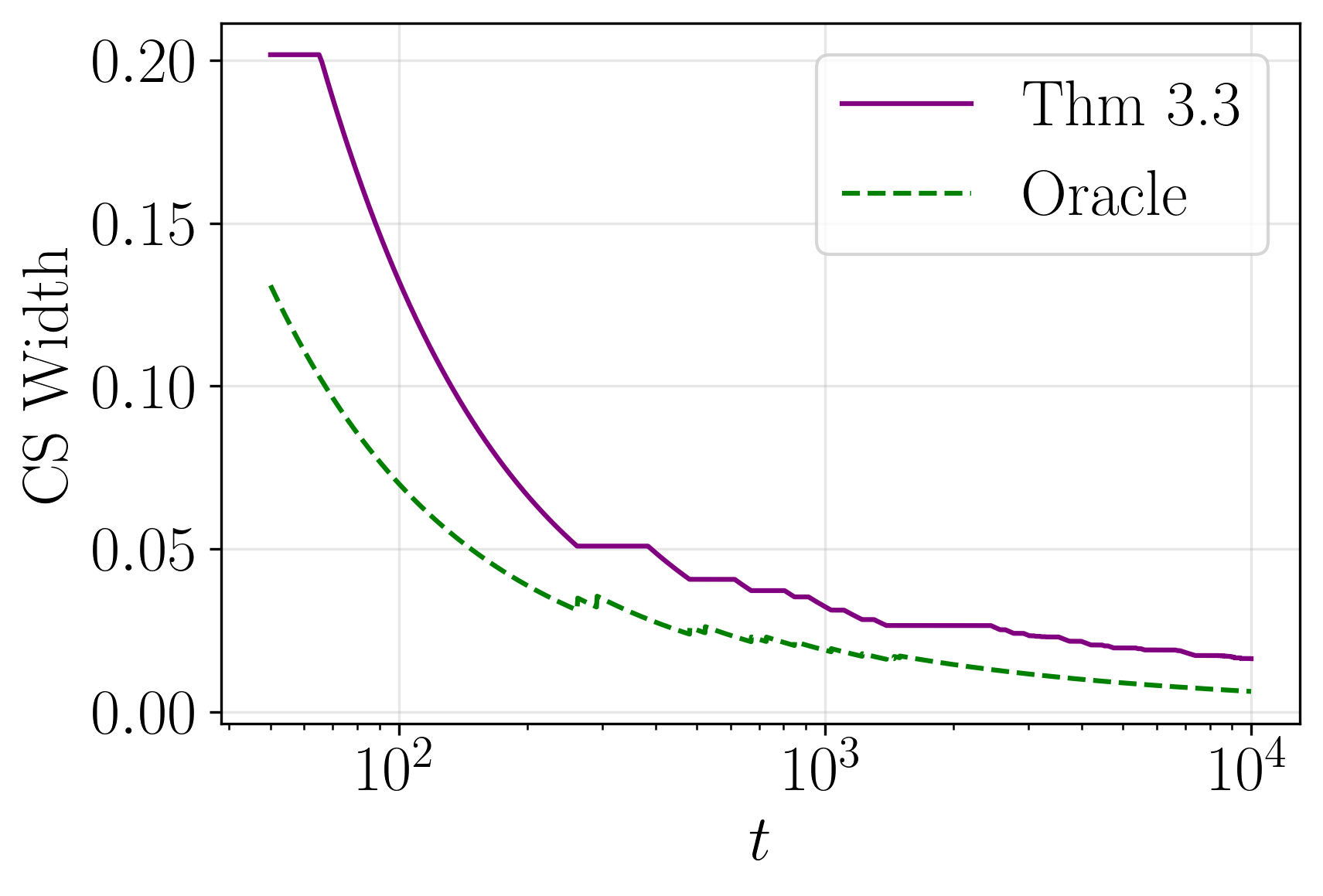}     
    \end{subfigure}
    \caption{Comparison of $C_t^\apx$ against the oracle Bernstein Bern on the two point distribution discussed in Appendix~\ref{app:psiE-vs-sigma} which maximizes the ratio $\E\psi_E(|-\mu|)/\sigma^2$. We use $\eps=0.01$ and $\alpha=0.05$. For $C_t^\apx$ we use $\kappa=0.25$ as usual.}
    \label{fig:psiE-vs-sigma}
\end{figure}

\subsection{On the width of two oracle Bernstein CSs}
\label{sec:oracle-bernstein}

Here we comment on two oracle (i.e., non-empirical) Bernstein CSs, both given by \citet{howard2021time}. The first uses the method of mixtures and has width scaling as $O(\sqrt{\log(t)/t})$, and the second uses stitching (see Section~\ref{sec:stitching}) and has width $O(\sqrt{\log\log(t)/t})$. Here we examine their asymptotic widths (multiplied by $\sqrt{\log(t)/t}$ and $\sqrt{\log\log(t)/t}$ respectively), showing that they are both asymptotically sharp, as defined in the introduction. That is, the mixture bound has limiting with $\sigma$ and the stitched bound has limiting width $\sqrt{2\sigma^2}$. 

To present both bounds, first let us recall that a bounded random variable $X\in[0,1]$ satisfies \emph{Bernstein's} condition with parameter $c=1$. That is, $\E[|X|^k]\leq \frac{k!}{2}\sigma^2$. 
In the language of \citet{howard2021time} this implies that the process $(\sum_{i\leq t} (X_i - \E_{i-1}[X_i]))_{t\geq 1}$ is sub-gamma with $c=1$ and variance proxy $V_t = t\sigma^2$. We can thus apply both the gamma-exponential conjugate mixture boundary \citet[Proposition 9]{howard2021time} and the stitched boundary \citet[Theorem 1]{howard2021time} to obtain two CSs for $\mu_t$. 

Let us first consider the conjugate mixture boundary. 
\citet[Proposition 2]{howard2021time} implies that the width $W_t$ of this CS behaves as 
\[W_t = \sqrt{\frac{\sigma^2}{t} \log\left(\frac{t\sigma^2}{2\pi\alpha f^2(0)}\right) + o\left(\frac{1}{t}\right)},\]
where $f$ is the density of the gamma distribution with respect to the Lebesgue measure (see Equation (66) in Howard et al.\ for the precise parameters of the gamma distribution used to construct the CS). Similarly to the analysis done in Section~\ref{sec:robbins-mixture}, we see that $W_t$ obeys $\lim_{t\to\infty} \sqrt{t/\log(t)} \cdot W_t = \sigma$. 

Next, the stitched boundary behaves similarly to the CS presented in Section~\ref{sec:sota}, but with the variance proxy $V_t = t\sigma^2$ instead of the empirical variance. As was done in Section~\ref{sec:stitching}, we can take limits and then let $\eta,s\to 1$ to see that the width $W_t$ obeys $\sqrt{t/\log\log(t)} W_t \to \sqrt{2\sigma^2}$ as $t\to\infty$. Note that, since we're letting $\eta$ and $s$ approach 1, it's really a \emph{sequence} of such bounds that has limiting width $\sqrt{2\sigma^2}$. The limit is not achieved by any one bound. 
This is also demonstrated by \citet[Corollary 1]{howard2021time}.

\section{Additional experiments and simulation details}
\label{app:experiments}

\begin{figure}[t]
	\centering
    \begin{subfigure}[t]{0.40\linewidth}
        \includegraphics[height=4.5cm]{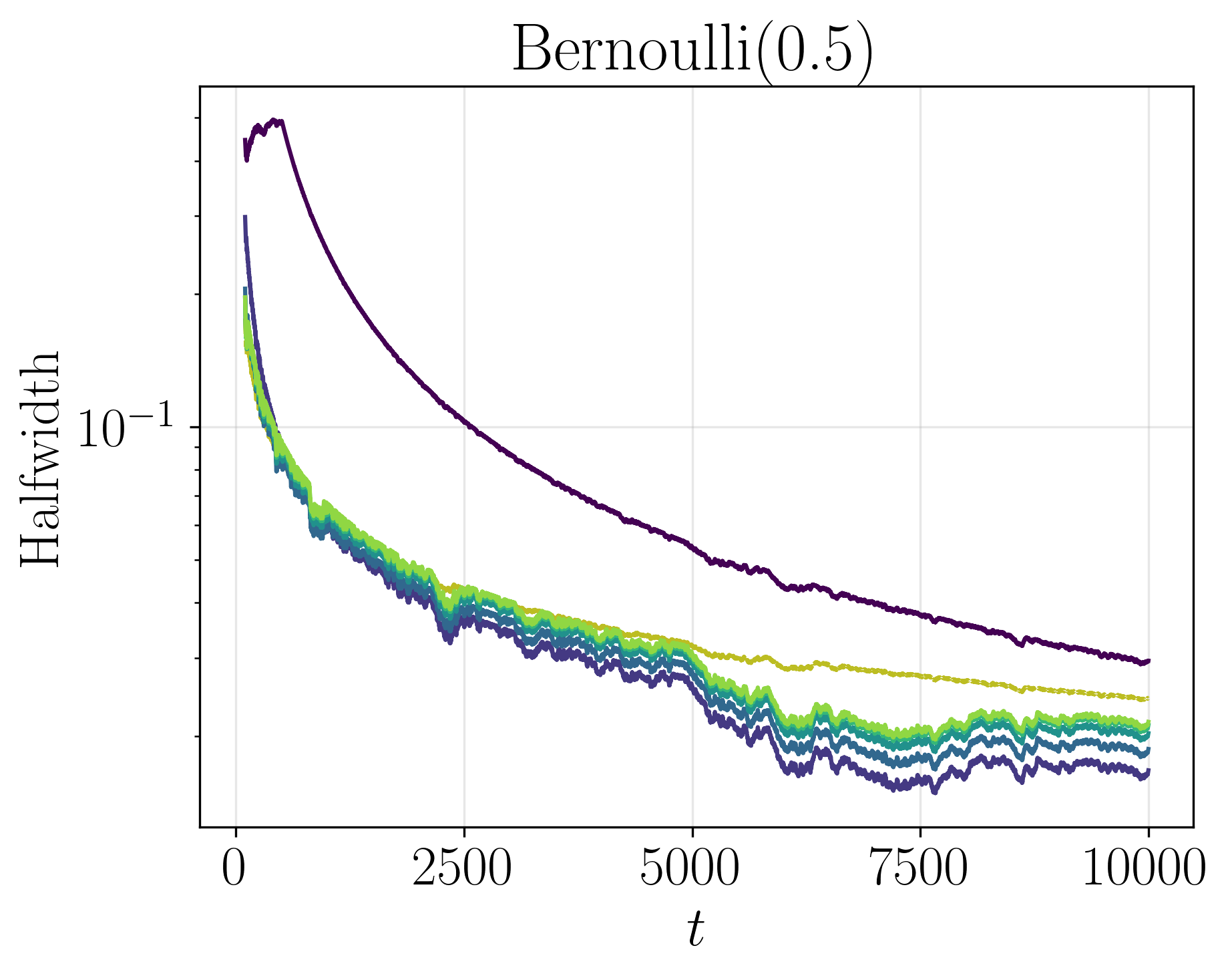}
        % \caption{}
    \end{subfigure}
    \begin{subfigure}[t]{0.59\linewidth}
        \includegraphics[height=4.5cm]{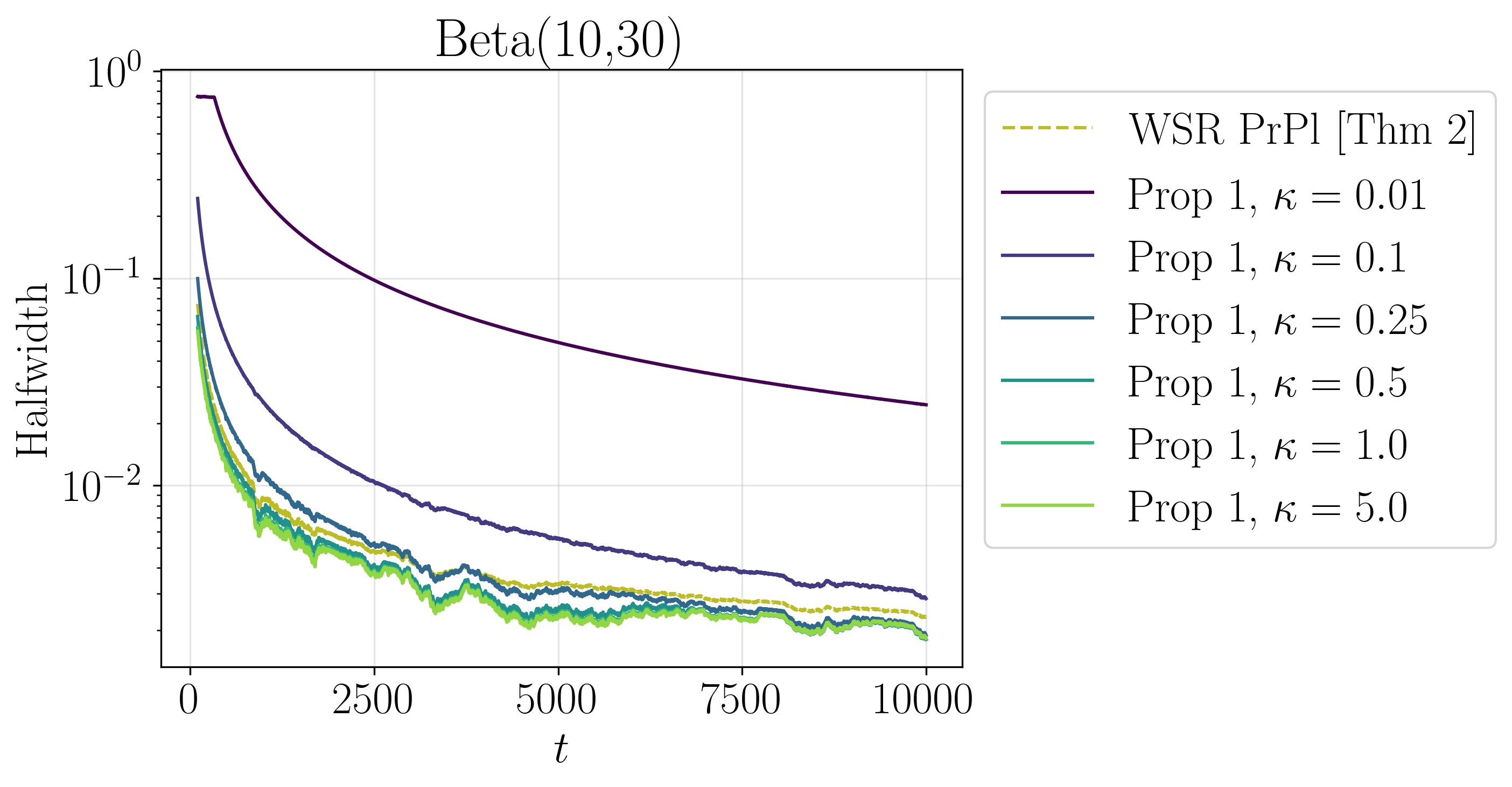}
        % \caption{}
    \end{subfigure}    
	\caption{The effect of $\kappa$ on the width of $C_t^\mix$. We compare the resulting bounds with the empirical Bernstein bound of \citet{waudby2024estimating}. Observations were drawn iid from either Ber(0.5) or $\text{ Beta(10,30)}$. We use $\alpha=0.01$.} 
	\label{fig:effect-of-kappa}
\end{figure}

Code to implement our CSs and replicate the figures can be found at \url{https://github.com/bchugg/empirical-bernstein}. To implement the CSs of \citet{waudby2024estimating} and \citet{howard2021time} we use the confseq package~\citep{howardconfseq}, which can be found at \url{https://github.com/gostevehoward/confseq}. At the time of this writing, the latter package requires Python 3.10 or lower. 

Here we present some additional experiments to get a handle on the effect of $\kappa$ in Proposition~\ref{prop:mixture-cs-tn} and Theorem~\ref{thm:closed-form}. Figure~\ref{fig:effect-of-kappa} plots the effect of $\kappa$ on the shape of $C_t^\mix$. We see that $\kappa$ can have a noticeable effect on the width of the CS, but that as $\kappa$ grows the widths of the bounds converge. This is predicted by the analysis below, and also by Remark~\ref{rem:alternative-dists}. Unfortunately, for small values of $\kappa$, the effect of $\kappa$ on the width is frustratingly difficult to predict. As demonstrated by the distinction between the  Bernoulli and Beta distribution in the figure, small values of $\kappa$ for some distributions, but have the best performance for others. As we noted in the main paper, we find that $\kappa=0.25$ is a good default choice across all distributions. 

We also emphasize that $\kappa$ cannot be data-dependent for the bound to retain its validity. That is, a practitioner should not compute the bound multiple times with different values of $\kappa$ and then take the tightest one. $\kappa$ should either be fixed at the beginning of the experiment or chosen based on a holdout dataset. If a practitioner \emph{does} recompute the bound multiples times with different values of $\kappa$, then such selection effects should be taken into account with a Bonferroni correction. 

Finally, let us note that we can get a handle on the effect of $\kappa$ analytically by noting that $\kappa$ affects the width of the boundary through $\ell_\alpha$ and the term $\kappa Z$, where $Z = \Phi(1/\kappa) - \Phi(-1/\kappa) = \erf(1/\sqrt{2\kappa^2})$. The Maclaurin series of the error function is $\erf(z) = \frac{2}{\sqrt{\pi}}( z - z^3/3 + z^5/10 - z^7/42 + \dots)$, hence for $\kappa$ large ($z$ small) we have $\kappa Z \sim 2/\sqrt{\pi}$, which is independent of $\kappa$. This is, of course, consistent with the fact that as $\kappa\to\infty$, a truncated Gaussian approaches a uniform distribution.

\end{document}